\documentclass[a4paper, 10pt]{amsart}

\usepackage{ifpdf}
\ifpdf 
    \usepackage[pdftex]{graphicx}   
    \usepackage[pdftex,     
            plainpages=false,   
            breaklinks=true,    
            colorlinks=true,
            pdftitle=My Document
            pdfauthor=My Good Self
           ]{hyperref} 
\else 
    \usepackage{graphicx}       
    \usepackage{hyperref}       
\fi 

\usepackage{amsfonts,amsmath}	
\usepackage{amssymb}
\usepackage{verbatim}
\usepackage{amsopn}
\usepackage[english]{babel}
\usepackage{amsthm}
\usepackage{enumerate}
\usepackage{mathrsfs}	
\usepackage{mathtools}

\author{Stefano Bianchini and Stefano Modena}
\title{On a quadratic functional for scalar conservation laws}
\date{\today}

\newcounter{assu}
\newcounter{clai}
\theoremstyle{definition} \newtheorem{definition}{Definition}[section]
\theoremstyle{definition} \newtheorem{remark}[definition]{Remark}
\theoremstyle{plain} \newtheorem{lemma}[definition]{Lemma}
\theoremstyle{plain} \newtheorem{proposition}[definition]{Proposition}
\theoremstyle{plain} \newtheorem{theorem}[definition]{Theorem}
\theoremstyle{plain} \newtheorem{corollary}[definition]{Corollary}
\theoremstyle{definition} \newtheorem{example}[definition]{Example}
\theoremstyle{plain} \newtheorem{claim}[clai]{Claim}

\newtheorem*{theorem1}{Theorem \ref{main_thm}}
\newtheorem*{theorem2}{Theorem \ref{W_main_thm}}
\newtheorem{theorem01}[assu]{Theorem}
\newtheorem{theorem02}[assu]{Theorem}

\DeclareMathOperator{\conv}{conv}
\DeclareMathOperator{\conc}{conc}
\DeclareMathOperator{\card}{card}
\DeclareMathOperator{\sign}{sign}
\newcommand{\R}{\mathbb{R}}

\newcommand{\N}{\mathbb{N}}
\newcommand{\Z}{\mathbb{Z}}
\newcommand{\TV}{\text{\rm Tot.Var.}}
\newcommand{\I}{\textbf{I}}

\newcommand{\W}{\mathcal{W}}
\newcommand{\C}{\mathcal{C}}
\newcommand{\e}{\varepsilon}
\newcommand{\const}{\mathcal{O}(1)}
\newcommand{\fQ}{\mathfrak{Q}}

\numberwithin{equation}{section} 
\begin{document}

\begin{abstract}
We prove a quadratic interaction estimate for approximate solutions to scalar conservation laws obtained by the wavefront tracking approximation or the Glimm scheme. This quadratic estimate has been used in the literature to prove the convergence rate of the Glimm scheme.

The proof is based on the introduction of a quadratic functional $\mathfrak Q(t)$, decreasing at every interaction, and such that its total variation in time is bounded. 

Differently from other interaction potentials present in the literature, the form of this functional is the natural extension of the original Glimm functional, and  coincides with it in the genuinely nonlinear case.
\end{abstract}

\thanks{The authors wish to thank prof. Fabio Ancona for many helpful discussions}

\maketitle

\tableofcontents

\section{Introduction}
\label{S_introduction}

Consider the Cauchy problem for a scalar conservation law in one space variable
\begin{equation}
\label{cauchy}
\left\{
\begin{array}{lll}
u_t + f(u)_x & = & 0 \\
u(0,x) &       = & \bar{u}(x)
\end{array}
\right.
\end{equation}
where $\bar u \in BV(\R)$, $f: \R \to \R$ smooth (by \emph{smooth} we mean at least of class $C^2(\R,\R)$). It is well known \cite{kru_70} that there exists a unique entropic solution $u(t,\cdot)$ satisfying
\begin{equation}
\label{E_2_lyapunov}
\|u(t,\cdot)\|_{L^\infty} \leq \|\bar u\|_{L^\infty}, \qquad \TV(u(t,\cdot)) \leq \TV(\bar u).
\end{equation}
In particular we can assume w.l.o.g. that for each $u \in \R$, $0 < f'(u) < 1$ and $f''$ uniformly bounded.

The solution to \eqref{cauchy} can be constructed in several ways, for example nonlinear semigroup theory \cite{cra_72}, finite difference schemes \cite{smo_83}, and in particular wavefront tracking \cite{daf_72} and Glimm scheme \cite{gli_65}. In the scalar case, in fact, the above functionals \eqref{E_2_lyapunov} are Lyapunov functionals, i.e. functionals decreasing in time, (or \emph{potential}, as they are usually called in the literature) for all the schemes listed, yielding compactness of the approximated solution (even in the multidimensional case).

In \cite{bia_bre_02} it is shown that there exists an additional Lyapunov functional $Q^{\text{BB}}[u]$ to the ones of \eqref{E_2_lyapunov}. In the case of wavefront tracking solution (see the references above or the beginning of Section \ref{sect_wavefront} for a short presentation), outside the interaction-cancellation times $\{t_j\}_j$, this functional takes the simple form
\begin{equation}
\label{E_cubic_Q_def}
Q^{\text{BB}}(t) = Q^{\text{BB}}[u(t,\cdot)] = \sum_{w \not= w'} |\sigma(w) - \sigma(w')| |w| |w'|,
\end{equation}
where $w$, $w'$ are wavefronts of the approximate solution $u(t,\cdot)$ traveling with speed $\sigma(w)$, $\sigma(w')$ respectively. To avoid confusion/misunderstanding, we will call \emph{wavefronts} the shocks/rarefactions/contact discontinuities of the (approximate) solution, while the word \emph{wave} will be reserved to subpartitions of wavefronts. The precise definition is given at the beginning of Section \ref{Ss_main_results}. Here and in the future, we call the collision between two wavefronts $w$, $w'$ an \emph{interaction} if the wavefronts which collide have the same sign, namely $w \cdot w' > 0$, otherwise we call it a \emph{cancellation} (see Definition \ref{W_int_canc_points}). In \cite{bia_bre_02} a general form of $Q^{\text{BB}}$ valid at all $t$ is given. 

This functional is of fundamental importance when studying the existence of solutions to systems of conservation laws in one space variable, where the total variation of the (approximate) solution is not decreasing in time. In particular, when two wavefronts $w$, $w'$ interact in the point $(\bar t,\bar x)$, it is possible to show that $Q^{\text{BB}}$ decreases of at least
\[
Q^{\text{BB}}(\bar t-) - Q^{\text{BB}}(\bar t+) \leq |\sigma(w) - \sigma(w')| |w| |w'|,
\]
where the speeds are computed at $\bar t-$.

Since it is well known that
\begin{equation}
\label{E_Lip_speed_1}
|\sigma(w) - \sigma(w')| \leq \|f''\|_{L^\infty} \TV(u(t,\cdot)),
\end{equation}
the following estimate holds
\begin{equation}
\label{E_cubic_est}
Q^{\text{BB}}(u(t,\cdot)) \leq \|f''\|_{L^\infty} \TV(u(t, \cdot))^3.
\end{equation}
It is customary to say that $Q^{\text{BB}}$ is a \emph{cubic functional}, referring precisely to the exponent of \eqref{E_cubic_est}. It follows in particular that the \emph{total amount of interaction} is cubic, 
\[
\sum_{\text{interactions}} |\sigma(w) - \sigma(w')| |w| |w'| \leq \|f''\|_{L^\infty} \TV(\bar u)^3.
\]
(See \cite{bia_03} for the general definition.)

In order to prove a convergence rate estimate for the Glimm schemes, in \cite{anc_mar_11_CMP}, \cite{hua_jia_yan_10} it is shown that if $u_\e$ is the approximate solution constructed by the Glimm scheme and $u$ is the entropic solution to a systems of conservation laws in one space dimension, then
\[
\|u(t, \cdot) - u_\e(t, \cdot)\|_{L^1} \leq o(1) \sqrt{\e} |\log \e|,
\]
under the assumption that the following estimate holds:
\begin{equation}
\label{E_quadrati_est}
\sum_{\text{interactions}} \frac{|\sigma(w) - \sigma(w')| |w| |w'|}{|w| + |w'|} \leq \mathcal O(\|f''\|_{L^\infty}) \TV(\bar u)^2.
\end{equation}
(Here and in the following $\mathcal O(\|f''\|_{L^\infty})$ is a constant which can depend on the $L^\infty$ norm of $f''$, but not on the initial datum $\bar u$.)

\noindent The above estimates is written in the case that only two wavefronts at a time interact for simplicity, the general form is presented in the statements of Theorem \ref{main_thm} (for the Glimm scheme) and Theorem \ref{W_main_thm} (for the wavefront tracking algorithm).

The key idea to prove estimate \eqref{E_quadrati_est} is to introduce a suitable functional $Q = Q(t)$, depending on the time, which decreases in time and is of quadratic order with respect to the total variation of the solution, and then to use this functional to get estimate \eqref{E_quadrati_est}.

In the case of genuinely nonlinear or linearly degenerate systems, because of the particular structure of solutions to the Riemann problem, the interaction functional $Q^{\text{GL}}$ introduced by Glimm \cite{gli_65} gives the estimate
\[
\sum_{\text{interactions}} |w| |w'| \leq \mathcal O(\|f''\|_{L^\infty}) \TV(\bar u)^2,
\]
and by the Lipschitz regularity of $\sigma(w)$ (inequality \eqref{E_Lip_speed_1}) it is immediate to deduce \eqref{E_quadrati_est}. However the functional $Q^{\text{BB}}$ introduced above cannot produce a quadratic estimate, being cubic as observed earlier.

Since 2006 many attempts have been made in order to get a proof of estimate \eqref{E_quadrati_est} (in addition to \cite{anc_mar_11_CMP}, \cite{hua_jia_yan_10}, see also \cite{hua_yang_10}). However, the proofs presented in \cite{hua_jia_yan_10}, \cite{hua_yang_10}) have some problems, as it is shown by the counterexamples in \cite{anc_mar_11_DCDS}. 

On the other hand, we discover an incorrect estimate in the proof of \cite{anc_mar_11_CMP}, precisely in Lemma 2, pag. 614,  formulas (4.84), (4.85). An explicit counterexample is presented in Appendix \ref{App_conter}, here we just notice that the 2 cited formulas try to estimate terms which are linear in the total variation (e.g. the difference in speed across interactions) with the decrease of the cubic functional $Q^{\text{BB}}$ defined in \eqref{E_cubic_Q_def}.

\subsection{Main result}
\label{Ss_main_results}

The main result of this paper is a new and correct proof of the estimate \eqref{E_quadrati_est} for approximate solutions constructed by wavefront tracking or by the Glimm scheme, in the case of scalar conservation laws. Our aim is to simplify as much as possible the technicalities in order to single out the ideas behind our approach. In a forthcoming paper we will study the general vector case.

In order to state precisely the two main theorems of this paper (one referring to the Glimm scheme, the other one to the wavefront tracking), we need to introduce what we call an \emph{enumeration of waves} in the same spirit as the \emph{partition of waves} considered in \cite{anc_mar_11_CMP}, see also \cite{anc_mar_10}. Roughly speaking, we assign an index $s$ to each piece of wave, and construct two functions $\mathtt x(t,s)$, $\sigma(t,s)$ which give the position and the speed of the wave $s$ at time $t$, respectively.

More precisely, let $u_\e$ be the Glimm approximate solution, with grid points $(t_n,x_m) = (n\e,m\e)$: for definiteness we assume $u_\e$ to be right continuous in space. Consider the interval $\mathcal W = (0,\TV(u_\e(0, \cdot))]$, which will be called the \emph{set of waves}. In Section \ref{S_front_glimm} we construct for the Glimm scheme a function
\begin{equation}
\label{E_x_def_glimm}
\begin{array}{ccccc}
\mathtt x &:& [0,
+\infty) \times \mathcal W &\to& (-\infty,+\infty] \\
&& (t,s) &\mapsto& \mathtt x(t,s)
\end{array}
\end{equation}
with the following properties:
\begin{enumerate}
\item the set $\{t: \mathtt x(t,s) < +\infty\}$ is of the form $[0,T(s))$ with $T(s) \in (0,+\infty)$: define $\mathcal W(t)$ as the set
\[
\W(t) := \big\{ s \in W \ | \ \mathtt x(t,s) < +\infty \big\};
\]
\item the function $t \mapsto \mathtt x(t,s)$ is increasing, $1$-Lipschitz and linear in each interval $(t_n,t_{n+1}) = (n,n+1) \e$ if $t_n \in [0,T(s))$;
\item if $t_n \in [0,T(s))$, then $\mathtt x(t_n,s) = x_m = m\e$ for some $m \in \Z$, i.e. it takes values in the grid points at each time step;
\item for $s < s'$ such that $\mathtt x(t,s),\mathtt x(t,s') < +\infty$ it holds
\[
\mathtt x(t,s) \leq \mathtt x(t,s');
\]
\item there exists a time-independent function $\mathcal S(s) \in \{-1,1\}$, the \emph{sign} of the wave $s$, such that
\begin{equation}
\label{E_push_forw}
D_x u_\e(t_n, \cdot) = \mathtt x(t_n,\cdot)_\sharp \big( \mathcal S(s) \mathcal L^1 \llcorner_{\W(t)} \big)
\end{equation}
for all $t_n \in [0,T(s))$.
\end{enumerate}
The last formula means that for all test functions $\phi \in C(\R,\R)$ it holds
\[
- \int_\R u_\e(t_n,x) D_x \phi(x) dx = \int_{\W(t_n)} \phi(\mathtt x(t_n,s)) \mathcal S(s) ds.
\]
The fact that $\mathtt x(t,s) = +\infty$ means that the wave has been removed from the solution $u_\e$ by a cancellation occurring at time $T(s)$.

\noindent Formula \eqref{E_push_forw} and a fairly easy argument, based on the monotonicity properties of the Riemann solver and used in the proof of Lemma \ref{W_lemma_eow}, yield that to each wave $s$ it is associated a unique value $\hat u(s)$ (independent of $t$) by the formula
\[
\hat u(s) = \bar u(-\infty) + \int_{\W(t) \cap [0,s]} \mathcal S(s) ds.
\]
We finally define the \emph{speed function} $\sigma : [0,+\infty) \times \W \to [0,1] \cup \{+\infty\}$ as follows: if $t \in [t_n,t_{n+1})$, then
\begin{equation}
\label{E_Glimm_speed}
\sigma(t,s) := \left\{
\begin{array}{ll}
+\infty & \text{if } \mathtt x(t_n,s) = +\infty, \\
\Big(\frac{d}{du}\conv_{[u_\e(t_n,\mathtt x(t_n,s)-),u_\e(t_n,\mathtt x(t_n,s))]}f\Big)\big(\hat u(s)\big) & \text{if } \mathcal{S}(s) = +1, \\ 
\Big(\frac{d}{du}\conc_{[u_\e(t_n,\mathtt x(t_n,s)),u_\e(t_n,\mathtt x(t_n,s)-)]}f\Big)\big(\hat u(s)\big) & \text{if } \mathcal{S}(s) = -1. 
\end{array}
\right.
\end{equation}
In other words, to the wave $s \in \W(t)$ and for $t \in [t_n,t_{n+1}) = [n,n+1) \e$ we assign the speed given by the Riemann solver in $(t_n,x_m) = (t_n,\mathtt x(t_n,s))$ to the wavefront containing the value $\hat u(s)$.

\noindent We can now state our theorem for Glimm approximate solutions.

\begin{theorem01}
\label{main_thm}
The following estimate holds:
\begin{equation}
\label{E_G_est_fin}
\sum_{n=1}^{+\infty} \int_{\W(n\varepsilon)} \big| \sigma(n \varepsilon, s) - \sigma((n-1)\varepsilon, s) \big| \, ds \leq (3 + 2 \log(2)) \lVert f'' \lVert_{L^\infty} \TV(\bar u)^2.
\end{equation}
\end{theorem01}

In the case of wavefront tracking, since the waves $s$ have size $k \e$, with $\e$ the discretization parameter and $k \in \Z$, it is possible to choose $\W \subseteq \N$, and in Section \ref{Front_Waves} it is shown that the function $\mathtt x$ defined in \eqref{E_x_def_glimm} satisfies slightly different properties: Property (3) is meaningless, and Property (5) holds for all $t \in [0,T(s))$.

\noindent The speed $\sigma$ is now defined as
\[
\sigma(t,s) = \frac{d}{dt} \mathtt x(t,s)
\]
outside the interaction/cancellation points, and it is extended to $[0,+\infty)$ by right-continuity. Notice that outside interaction-cancellation times, the strength of the wavefront $w$ at $(t,x)$ is given by
\[
|w| = \e \, \sharp \big\{ s : \mathtt x(t,s) = x \big\},
\]
i.e. the strength of the wavefront $w$ is the sum of the strength of all waves $s$ which are mapped by $\mathtt x$ into $x$.

\noindent The main estimate for the wavefront tracking solution is contained in the following result.

\begin{theorem02}
\label{W_main_thm}
The following holds: if $\{t_j\}_j$ are the interaction-cancellation times, then
\begin{equation}
\label{E_W_est_fin}
\sum_j \sum_{s \in \W(t_j)} |\sigma(t_j, s) - \sigma(t_{j-1}, s)||s| \leq (3+2 \log(2)) \lVert f'' \lVert_{L^\infty} \TV(\bar u)^2,
\end{equation}
where $|s| := \e$ is \emph{the strength of the wave $s$}.
\end{theorem02}

As it is shown in the proof of Theorem \ref{W_decreasing}, formula \eqref{W_zzero}, the estimate \eqref{E_W_est_fin} yields immediately \eqref{E_quadrati_est} for wavefront tracking solution. The corresponding computation for Glimm scheme is given in the proof of Theorem \ref{decrease_thm_enunciato}, formula \eqref{zero}.

We observe here that for the interaction of two wavefronts $w$, $w'$, the quantity
\begin{equation}
\label{E_quadr_singl_inter}
h = \frac{|\sigma(w) - \sigma(w')| |w| |w'|}{|w| + |w'|}
\end{equation}
has a nice geometric interpretation (see Figure \ref{fig:figura36}), in the same spirit as the area interpretation of the cubic functional $Q^{\text{BB}}$ [see bianchini-bressan curve shortening]. In fact, if $w$, $w'$ correspond to the jumps $[u^L,u^M]$, $[u^M,u^R]$ with $u^L < u^M < u^R$, then it is fairly easy to see that \eqref{E_quadr_singl_inter} is equal to the height of the triangle $(u^L,f(u^L))$, $(u^M,f(u^M))$, $(u^R,f(u^R))$, more precisely
\[
h = \frac{|\sigma(w) - \sigma(w')| |w| |w'|}{|w| + |w'|} = f(u^M) - \frac{(u^M - u^L) f(u^R) + (u^R - u^M) f(u^L)}{u^R - u^L}.
\]

For the Glimm scheme, each interaction involves the several wavefronts (not just two), and it corresponds to replacing two adjacent Riemann problems, namely $[u^L,u^M]$ and $[u^M,u^R]$, with the Riemann problem $[u^L,u^R]$: the corresponding quantity is then given by
\[
h = f(u^M) - \conv_{[u^L,u^R]} f(u^M),
\]
assuming again $u^L < u^M < u^R$ for definiteness (see Figure \ref{fig:figura19}). In this way, one can rewrite \eqref{E_quadrati_est} also for the Glimm scheme.

\begin{figure}
  \begin{center}
    \includegraphics[height=5cm,width=8cm]{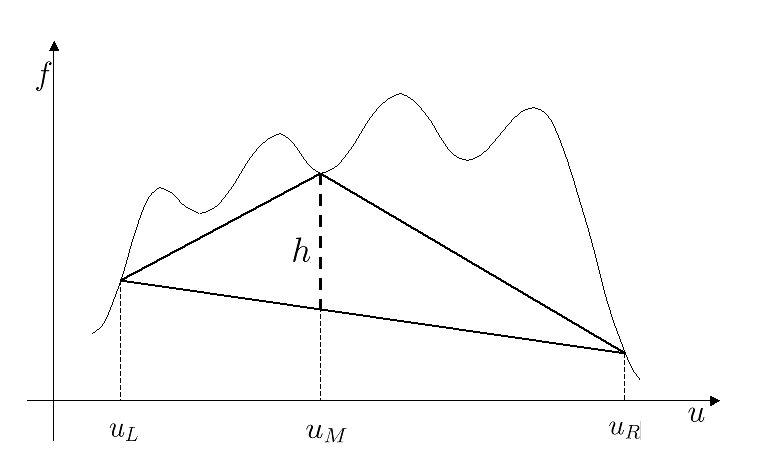}
    \caption{Geometric interpretation of $h$ in the wavefront tracking.}
    \label{fig:figura36}
    \end{center}
\end{figure}

\begin{figure}
  \begin{center}
    \includegraphics[height=7cm,width=10cm]{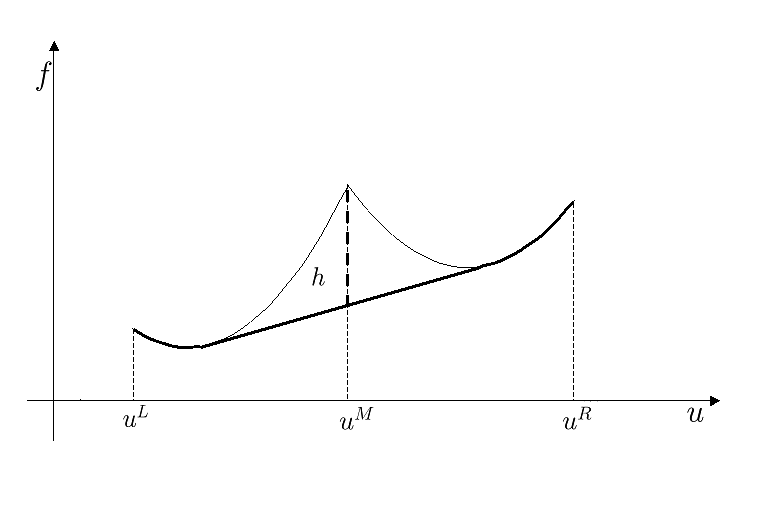}
    \caption{Geometric interpretation of $h$ in the Glimm scheme}
    \label{fig:figura19}
    \end{center}
\end{figure}

Our choice to present two separate theorems is motivated by the following facts.

First of all, due to the number of papers present in the literature concerning this estimate, we believe necessary to give a correct proof in the most simple case, i.e. wavefront tracking for scalar conservation laws.

Since however the estimate has been used mainly to prove the convergence rate of the Glimm scheme, we felt necessary to offer also a direct proof for this approximation scheme. It turns out that even if the fundamental ideas are the same, the two proofs are sufficiently different in some points to justify a separate analysis. In our opinion, in fact, it is not trivial to deduce one from the other.

\subsection{Sketch of the proof}
\label{Sss_sketch_proof}

As observed in \cite{anc_mar_11_DCDS}, one of the main problems in obtaining an estimate of the form \eqref{E_G_est_fin}, \eqref{E_W_est_fin} is that the study of wave interactions cannot be local in time, but one has to take into accoun the whole sequence of interactions-cancellations for every couple of waves. This is a striking difference with the Glimm interaction potential $Q^\text{{GL}}$, where the past history of the solution is irrelevant.

Previous attempts tried to adapt Glimm's idea of finding a \emph{quadratic potential $Q$} which is decreasing in time and at every interaction has a jump of the amount \eqref{E_quadr_singl_inter}. Apart technical variations, the idea is to transform the function $Q^{\text{BB}}$ of \eqref{E_cubic_Q_def} into
\[
Q(t) = Q[u(t,\cdot)] := \sum_{w \not= w'} \frac{|\sigma(w) - \sigma(w')| |w| |w'|}{|w| + |w'|}.
\]
For monotone initial data $\bar u$, this functional is sufficient; however in \cite{anc_mar_11_DCDS} F. Ancona and A. Marson show that $Q$ defined above may be not bounded, so that in \cite{anc_mar_11_CMP} they consider the functional
\begin{equation}
\label{E_Q_AM}
Q^\text{{AM}}(t) := \sum_{w \not= w'} \frac{|\sigma(w) - \sigma(w')| |w| |w'|}{|w| + \TV(u(t,\cdot),(x_w(t),x_{w'}(t))) + |w'|},
\end{equation}
where $x_w(t)$ is the position of the wave $w$ at time $t$. Notice that at the time of interaction of the wavefronts $w$, $w'$ one has $\TV(u(t,\cdot),(x_w(t),x_{w'}(t))) = 0$ (there are no wavefronts between the two interacting), so that for the couple of waves $w$, $w'$ one has
\[
\frac{|\sigma(w) - \sigma(w')| |w| |w'|}{|w| + \TV(u(t,\cdot),(x_w(t),x_{w'}(t))) + |w'|} = \frac{|\sigma(w) - \sigma(w')| |w| |w'|}{|w| + |w'|}.
\]

If the flux function $f$ has no finite inflection points, the waves $s \in \W$ can join in wavefronts (because of an interaction) and split again (because of a cancellation) an arbitrary large number of times. This implies that the functional $Q^\text{{AM}}$ (as well as the other quadratic functionals introduced in the literature) does not decay in time, but can increase due to cancellations. Hence, instead of proving directly that the quadratic functional $Q^\text{{AM}}(t)$ controls the interactions, in \cite{anc_mar_11_CMP} the authors consider a term $G(t)$ which bounds the oscillations of $Q^\text{{AM}}(t)$ for the waves not involved in an interaction, and prove that
\[
\sum_{\text{interactions in }(t_1,t_2]} \frac{|\sigma - \sigma'| s s'}{s + s'} \leq (Q^\text{{AM}}(t_1) + G(t_1)) - \big( Q^\text{{AM}}(t_2) + G(t_2) \big).
\]
Since $Q^\text{{AM}}$ is quadratic by construction because of the Lipschitz regularity of the speed $\sigma(w)$ (inequality \eqref{E_Lip_speed_1}) and $G \leq 0$, $G(+\infty) = 0$, they reduce the quadratic interaction estimate to the following estimate:
\begin{equation}
\label{E_quadratic_G}
\TV^+(G) \leq \mathcal O(\|f''\|_{L^\infty}) \TV(\bar u)^2.
\end{equation}

Our approach is slightly different: we construct a quadratic functional $\mathfrak Q$ such that its total variation in time is bounded by $\mathcal O(\|f''\|_{L^\infty}) \TV(\bar u)^2$, and at any interaction decays at least of the quantity \eqref{E_quadr_singl_inter} (or more precisely of the quantities in the l.h.s. of \eqref{E_G_est_fin} or \eqref{E_W_est_fin} concerning that interaction). The functional can increase due to cancellations, but in this case we show that its positive variation is controlled by the total variation of the solution times the amount of cancellation. Being the total variation a Lyapunov functional, it follows that
\[
\text{positive total variation of} \ \mathfrak Q(t) \leq \mathcal O(\|f''\|_{L^\infty}) \TV(\bar u)^2,
\]
so that, being $\mathfrak Q(0) \leq \mathcal O(\|f''\|_{L^\infty}) \TV(\bar u)^2$, the functional $t \mapsto \mathfrak Q(t)$ has total variation of the order of $\TV(\bar u)^2$. In particular,
\begin{equation*}
\text{left hand side of \eqref{E_G_est_fin} or \eqref{E_W_est_fin} at interactions} \leq \text{negative variation of} \ \mathfrak Q \leq \mathcal O(\|f''\|_{L^\infty}) \TV(\bar u)^2.
\end{equation*}
The estimates \eqref{E_G_est_fin}, \eqref{E_W_est_fin} concerning cancellations is much easier (and already done in the literature, see \cite{anc_mar_11_CMP}), and we present it in Propositions \ref{W_canc_3}, \ref{canc_3}, depending on the approximation scheme considered. In the case of cancellations, in fact, there is a first order functional decreasing, namely the total variation.

For simplicity we present the sketch of the proof in the wavefront tracking case.

We define again a functional $\mathfrak Q(t)$ of a form similar to \eqref{E_Q_AM}, 
\[
\mathfrak{Q}(t) :=  \sum_{\substack{s,s' \in \W(t) \\ s < s'}} \mathfrak{q}(t, s, s') |s||s'|,
\]
but with $4$ main differences.
\begin{enumerate}
\item First of all its definition involves the waves $s$, not the wavefronts $w$.

\item \label{Point_Glimm} If the waves $s$, $s'$ have not yet interacted, then the weight $\mathfrak{q}(t,s,s')$ is a large constant. In our case, it suffices $\|f''\|_{L^\infty}$.

\noindent If the waves have already interacted, then the weight $\mathfrak q$ has the form
\begin{equation}
\label{E_weight_form}
\mathfrak q(t,s,s') := \dfrac{|\Delta \sigma(t,s,s')|}{|\e \, \sharp \{s'' \text{ such that }s \leq s'' \leq s'\}|}.
\end{equation}

\item In the above formula \eqref{E_weight_form}, the quantity $\Delta \sigma(t,s,s')$ is the difference in speed given to the waves $s$, $s'$ by an artificial Riemann problem, which roughly speaking collects all the previous common history of the waves $s$, $s'$. This makes $\mathfrak Q(t)$ not local in time and space.

\item The denominator of \eqref{E_weight_form} is not the total variation of $u$ between two wavefronts $w$, $w'$ (containing $s$, $s'$ respectively) as it is the case in \eqref{E_Q_AM}, but only the total variation between the two waves $s$, $s'$. Observe that
\[
|\e \, \sharp \{s'' \text{ such that } s \leq s'' \leq s'\}| \leq |w| + \TV(u(t,\cdot),(x_w(t),x_{w'}(t))) + |w'|.
\]
\end{enumerate}

By the Lipschitz regularity of the speed given to the waves by solving a Riemann problem, it is fairly easy (and proved in Section \ref{W_functional_Q}) that
\[
\mathfrak Q(t) \leq \|f''\|_{L^\infty} \TV(u(t,\cdot))^2.
\]

We observe first that the functional $\mathfrak Q$ restricted to the couple of waves which have never interacted is exactly (apart from the constant $\|f''\|_{L^\infty}$) the original Glimm functional 
\[
Q^\text{{GL}}(t) = Q^\text{{GL}}[u(t,\cdot)] := \sum_{w \not= w'} |w| |w'|.
\]
Since the couple of waves which have never interacted is decreasing, this part of the functional is decreasing, and as observed before it is sufficient to control the quadratic estimates for couple of waves which have never interacted.

The choice of the denominator in \eqref{E_weight_form} yields that our functional $\mathfrak{Q}$ is not affected by the issue of large oscillations, as observed in \cite{anc_mar_11_DCDS}, even if its form is similar to $Q^\text{{AM}}$: indeed, in our case, the denominator we choose does not depend on the shock component which the waves belong to, and thus cancellations do not affect it.

Next, the Riemann problem used to compute the quantity $\Delta \sigma(t,s,s')$ is made of all waves $s''$ which have interacted with both waves $s$ and $s'$. We now show how it evolves with time.
\begin{description}
\item[interaction] If at $t_j$ an interaction occurs, and $s$, $s'$ are not involved in the interaction, the set of waves which have interacted with both is not changing (since they are separated, at most one of them is involved in the interaction!), which means that $\Delta \sigma(t_j-,s,s') = \Delta \sigma(t_j+,s,s')$. If $s$, $s'$ are involved in the interaction, then the couple disappears from the sum, because when two wavefronts with the same sign interact a single wavefront comes out.

\item[cancellation] If a cancellation occurs at $t_j$, then one can check that again if both $s$, $s'$ are not involved in the wavefront collision then $\Delta \sigma$ is constant. Otherwise the change in $\Delta \sigma$ corresponds to the change in speed obtained by removing some waves in a Riemann problem, and adding all these variations one obtains that the oscillation of $\mathfrak Q$ can be estimated explicitly as in the case of a single cancellation (i.e. the total variation of the solution times the amount of cancellation).
\end{description}

The cancellation case, corresponding to the positive total variation in time of $\mathfrak Q$, is thus controlled by
\[
\mathcal O(\|f''\|_{L^\infty}) \TV(\bar u)^2,
\]
and since $\mathfrak Q(0) \leq \mathcal O(\|f''\|_{L^\infty}) \TV(\bar u)^2$ and $\mathfrak Q \geq 0$ it follows that
\begin{equation}
\label{E_tv_frak_q}
\TV(\mathfrak Q(t)) \leq \mathcal O(\|f''\|_{L^\infty}) \TV(\bar u)^2.
\end{equation}

When an interaction between $w$, $w'$ occurs, the discussion above shows that we can split the waves involved in the interaction into $4$ sets:
\begin{enumerate}
\item a set $\mathcal L_1$ of waves in $w$ which have never interacted with the waves in $w'$;
\item a set $\mathcal L_2$ of waves in $w$ which have interacted with a set $\mathcal R_1$ of waves in $w'$;
\item a set $\mathcal R_2$ of waves in $w'$ which have never interacted with the waves in $w$.
\end{enumerate}
The speed assigned by the artificial Riemann problems for all couples of waves $s \in w$, $s' \in w'$ yields that the decrease of $\mathfrak Q$ at the interaction time is (larger than the one) given by the wave pattern depicted in Figure \ref{fig:figura37}. 
\begin{figure}
  \begin{center}
    \includegraphics[height=6cm,width=11cm]{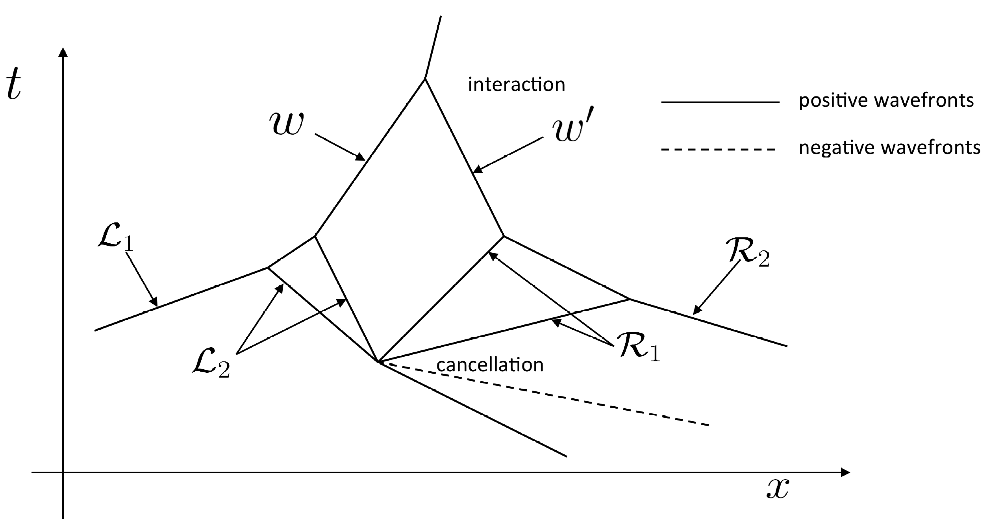}
    \caption{Pattern of waves involved in the interaction at time $t_j$.}
    \label{fig:figura37}
    \end{center}
\end{figure}
By an explicit computation one can check that
\begin{equation*}
\text{left hand side of \eqref{E_W_est_fin} at interaction} \leq \text{negative variation of} \ \mathfrak Q \overset{\eqref{E_tv_frak_q}}{\leq} \mathcal O(\|f''\|_{L^\infty}) \TV(\bar u)^2,
\end{equation*}
thus obtaining the desired estimate \eqref{E_W_est_fin}.

The proof for the Glimm scheme follows the same philosophy, but, as we said, due to the different structure of the approximating scheme it present different technical aspects. 

\subsection{Structure of the paper}
\label{Ss_structure}

The paper is organized as follows.

Section \ref{S_convex_env} provides some useful results on convex envelopes. Part of these results are already present in the literature, others can be deduced with little effort. We decided to collect them for reader's convenience. Two particular estimates play a key role in the main body of the paper: the regularity of the speed function (Theorem \ref{convex_fundamental_thm} and Proposition \ref{convex_fundamental_thm_affine}) and the behavior of the speed assigned to a wave by the solution to Riemann problem $[u^L,u^R]$ when the left state $u^L$ or the right state $u^R$ are varied (Propositions \ref{tocca}, \ref{vel_aumenta}, \ref{differenza_vel} and \ref{incastro}).

The next two sections contain the main results of the paper. As we said, for technical reasons the proofs differ depending on the approximation scheme considered, but, in order to simplify the exposition, we tried to keep a similar structure of both sections.

Section \ref{sect_wavefront} is devoted to the proof of Theorem \ref{W_main_thm}.

\noindent After recalling how a wavefront approximated solution $u_\e$ is constructed, we begin with the construction of the wave map $\mathtt x$ in Section \ref{Front_Waves}. As we said, due to the fixed size $\e$ of the waves, the set of waves $\mathcal W$ will be a finite subset of $\N$, and we replace the properties of $\mathtt x$ given at the beginning of this introduction with the (easier to work with) definition of \emph{enumeration of waves}, Definition \ref{W_eow}. This is the triple $(\mathcal W, \mathtt x,\hat u)$, where $\mathtt x$ is the position of the waves $s$ and $\hat u$ is its right state. The equivalence of the two definitions is straightforward. In Section \ref{W_pswaves} we show that it is possible to construct a function $\mathtt x(t,s)$ such that at any time $(\mathcal W,\mathtt x(t),\hat u)$ is an enumeration of waves, with $\hat u$ independent on $t$.

\noindent Once we have an enumeration of waves, we can start the proof of Theorem \ref{W_main_thm} (Section \ref{section_W_main_thm}). First we study the estimate \eqref{E_W_est_fin} when a single cancellation occurs. This estimate is standard, since the cancellation is controlled by the decay of a first order functional, namely $\TV(u(t,\cdot))$. The precise estimate is reported in Proposition \ref{W_canc_3}, where the dependence w.r.t. $\TV(u(t,\cdot))$ and $\|f''\|_{L^\infty}$ is singled out. Corollary \ref{W_canc_4} completes the estimate \eqref{E_W_est_fin} for the case of cancellation points.

\noindent The rest of Section \ref{sect_wavefront} is the construction and analysis of the functional $\mathfrak Q$ described above, in order to prove Proposition \ref{W_thm_interaction}. This proposition proves \eqref{E_W_est_fin} for the case of interaction points, completing the proof of Theorem \ref{W_main_thm}.

\noindent As we said, one of the main features of the enumeration of waves is that we can speak of couple of waves which have already interacted. In Section \ref{W_waves_collision} we prove some important properties of these couples of waves. Lemma \ref{W_interagite_stesso_segno} shows that they must have the same sign, and, because in the scalar case no new waves are created, all the waves between $s$ and $s'$ have interacted with $s$ and $s'$, Lemma \ref{W_quelle_in_mezzo_hanno_int}. In this section it is also defined the interval of waves $\mathcal I(t,s,s')$, which is the set of waves $p \in \W$ which have interacted with both $s$ and $s'$. Propositions \ref{W_unite_realta} and \ref{W_divise_tocca} prove that, even if the solution to the Riemann problem generated by the waves $\mathcal I(t,s,s')$ assigns artificial speeds to $s$, $s'$, the property of being separated or not in the real solution can be deduced from the solution to the Riemann problem $\mathcal I(t,s,s')$. From this fact we can infer a lot of interesting properties of the Riemann problem $\mathcal I(t,s,s')$: the most important one is that we can know the values $u$ where $f(u) = \conv_{\hat u(\mathcal I(t,s,s'))} f(u)$.

\noindent In Section \ref{W_functional_Q} we write down the functional $\mathfrak Q$ and conclude the proof of Theorem \ref{W_main_thm}. We study separately the behavior of $\mathfrak Q$ at interactions and cancellations. Theorem \ref{W_decreasing} proves that the functional $\mathfrak Q$ decreases at least of the quantity \eqref{E_W_est_fin} at a single interaction point, while Theorem \ref{W_increasing} shows that the increase of $\mathfrak Q$ at each cancellation point is controlled by the total variation of the solution times the cancellation. These two facts conclude the proof of Proposition \ref{W_thm_interaction}, as shown in Section \ref{Sss_sketch_proof}.

Section \ref{S_glimm_scheme} is devoted to the proof of Theorem \ref{main_thm}. 

\noindent As we said, the ideas of the proof are similar, but from the technical point of view there are substantial differences, making this case slightly more complicated. In this introduction we will underline these variations, so that the reader can easily pass from one proof to the other, also because we tried to keep the structure of the two main sections similar.

\noindent As we already said, the first main difference is in the definition of enumeration of waves, Definition \ref{eow}. In fact, for the Glimm scheme the set of waves $\mathcal W$ is a subset of an interval in the real line (namely $(0,\TV(\bar u)]$), and moreover the map $t \mapsto \mathtt x(t,s)$ has to pass trough the grid points. This forces us to define the speed $\sigma(t,s)$ of a wave $s$ at time $t$ not by just taking the time-derivative of $\mathtt x(t,s)$, but by considering the real speed given to $s$ by the Riemann problem at each grid point as in \eqref{E_Glimm_speed}. This analysis is done in Section \ref{pswaves}.

\noindent Another difference is that the Glimm scheme, due to the choice of the sampling points, ``interacts'' with the solution (i.e. it may alter the curve $t \mapsto \mathtt x(t,s)$) even when no real interaction-cancellation occurs. This is why we need to study an additional case, namely when no interaction/cancellation of waves occurs at a given grid point, and this is done in Proposition \ref{P_no_interaction}: the statement is that trivially nothing is going on in these points, but we felt the need of a proof.

\noindent Proposition \ref{canc_3}, namely the case of cancellation points, is analog to the wavefront tracking case. Also the structure and properties of the functional $\mathfrak Q$ we are going to construct in the Glimm scheme case are similar to the wavefront approximation analysis, the main difference being that several interactions and cancellations occur at each time step. We thus require that the set of pairs of waves present in the solution at time step $n \e$ (i.e. not moved to $\mathtt x = +\infty$), namely $\W(n\e) \times \W(n\e)$, can be split into two parts: one part concerns the interactions, and decreases of the right amount, the other one concerns cancellations and increases of a quantity controlled by the total variation of the solution times the amount of cancellation. The remain part of the section is the proof of these two estimates, from which one deduces Theorem \ref{main_thm} along the same line outlined in Section \ref{Sss_sketch_proof}.

\noindent First, we define the notion of waves which have already interacted, Definition \ref{D_Glimm_interacted}, and waves which are separated (or \emph{divided}), Definition \ref{waves_divided}. Notice that even if they occupy the same grid position for some time step, they are considered divided in the real solution if the Riemann problem at that grid point assigns different speeds to them. The statement that the artificial Riemann problems we consider separate waves as in the real solution is completely similar to the wavefront case, Proposition \ref{unite_realta}, but the proof is quite longer.

\noindent In the last section, Section \ref{Ss_glimm_funct_Q} we define the functional $\mathfrak Q$, and due to the continuity of the set of waves $\mathcal W$ we show a regularity property of the weight $\mathfrak q(t,s,s')$ so that no measurability issues arise. The proof of the two estimates (interactions and cancellations) is somehow longer than in the wavefront tracking case but it is based on the same ideas, and this concludes the section.

In Appendix \ref{App_conter} we present a counterexample to formula (4.84) in Lemma 2, pag. 614, of \cite{anc_mar_11_CMP}, which justifies the need of the analysis in the one-dimensional case.

\subsection{Notations}
\label{Sss_notations}

For usefulness of the reader, we collect here some notations used in the subsequent sections.

\begin{itemize}
\item $g(u+) = \lim_{u \rightarrow u^+} g(u)$, $g(u-) = \lim_{u \rightarrow u^-} g(u)$;
\item $g'(u-)$ (resp. $g'(u+)$) is the left (resp. right) derivative of $g$ at point $u$;
\item If $(a_k)_k$ is a sequence of real numbers, we write $a_k \nearrow a$ (resp. $a_k \searrow a$) if $(a_k)_k$ is increasing (resp. decreasing) and $\lim_{k \rightarrow +\infty} a_k = a$;
\item Given $E \subseteq \R^n$, we will denote equivalently by $\mathcal{L}^n(E)$ or by $|E|$ the $n$-dimensional Lebesgue measure of $E$.
\item If $g: [a,b] \to \R$, $h: [b,c] \to \R$ are two functions which coincide in $b$, we define the function $g \cup h: [a,c] \to \R$ as
 \begin{equation*}
 g \cup h (x) = 
 \left\{
 \begin{array}{ll}
 g(x) & \text{if } x \in [a,b], \\
 h(x) & \text{if } x \in [b,c].
 \end{array}
 \right.
 \end{equation*}
\item Sometime we will write $\R_x$ instead of $\R$ (resp. $[0,+\infty)_t$ instead of $[0,+\infty)$) to emphasize the symbol of the variables (resp. $x$ or $t$) we refer to.
\item For any $f: \R \to \R$ and for any $\e >0$, the \emph{piecewise affine interpolation of $f$ with grid size $\e$} is the piecewise affine function $f_\e: \R \to \R$ which coincides with $f$ in the points of the form $m\e$, $m \in \Z$. 
\item By $\const$ we mean a quantity which does not depend on the data of the problem, neither on $f$ nor on the initial datum $u(0,\cdot) = \bar u$.
\item For any $\alpha \in \R$, $[\alpha]^+$ denotes its positive part.
\end{itemize}

\section{Convex Envelopes}
\label{S_convex_env}

In this section we define the convex envelope of a continuous function $f: \R \to \R$ in an interval $[a,b]$ and we prove some related results. The first section provides some well-known results about convex envelopes, while in the second section we prove some propositions which will be frequently used in the paper.

The aim of this section is to collect the statements we will need in the main part of the paper. In particular, the most important results are Theorem \ref{convex_fundamental_thm} and Proposition \ref{convex_fundamental_thm_affine}, concerning the regularity of convex envelopes, and Proposition \ref{tocca}, Proposition \ref{vel_aumenta}, Corollary \ref{stesso_shock}, Proposition \ref{incastro} and Proposition \ref{diff_vel_proporzionale_canc}, referring to the behavior of convex envelopes when the interval $[a,b]$ is varied: these estimates will play a major role for the study of the Riemann problems.

\subsection{Definitions and elementary results}

\begin{definition}
\label{convex_fcn}
Let $f: \R \to \R$ be continuous and $[a,b] \subseteq \R$. We define \emph{the convex envelope of $f$ in the interval $[a,b]$} as
\[
\conv_{[a,b]}f (u) := \sup\bigg\{g(u) \ \Big| \ g: [a,b] \to \R \text{ is convex and } g \leq f\bigg\}.
\]
\end{definition}

A similar definition holds for \emph{the concave envelope of $f$ in the interval $[a,b]$} denoted by $\conc_{[a,b]}f$. All the results we present here for the convex envelope of a continuous function $f$ hold, with the necessary changes, for its concave envelope.

\begin{lemma}
In the same setting of Definition \ref{convex_fcn}, $\conv_{[a,b]}f$ is a convex function and $\conv_{[a,b]}f(u) \leq f(u)$ for each $u \in [a,b]$.  
\end{lemma}

The proof is straightforward.

Adopting the language of Hyperbolic Conservation Laws, we give the next definition.

\begin{definition}
Let $f$ be a continuous function on $\R$, let $[a,b] \subseteq \R$ and consider $\conv_{[a,b]}f$. A \emph{shock interval} of $\conv_{[a,b]}f$ is an open interval $I \subseteq [a,b]$ such that for each $u \in I$, $\conv_{[a,b]}f(u) < f(u)$.

A \emph{maximal shock interval} is a shock interval which is maximal with respect to set inclusion.
\end{definition}

Notice that, if $u \in [a,b]$ is a point such that $\conv_{[a,b]}f(u) < f(u)$, then, by  continuity of $f$ and $\conv_{[a,b]}f$, it is possible to find a maximal shock interval $I$ containing $u$.

It is fairly easy to prove the following result.

\begin{proposition}
\label{shock}
Let $f: \R \to \R$ be continuous; let $[a,b] \subseteq \R$. Let $I$ be a shock interval for $\conv_{[a,b]}f$. Then $\conv_{[a,b]}f$ is affine on $I$.
\end{proposition}

The following theorem provides a description of the regularity of the convex envelope of a given function $f$.
 
\begin{theorem}
\label{convex_fundamental_thm}
Let $f$ be a $\mathcal{C}^2$-function. Then:
\begin{enumerate}
\item \label{convex_fundamental_thm_1} the convex envelope $\conv_{[a,b]} f$ of $f$ in the interval $[a,b]$ is differentiable on $[a,b]$; 
\item \label{convex_fundamental_thm_2} for each $u \in (a,b)$, if $f(u) = \conv_{[a,b]}f(u)$, then 
\[
\frac{d}{du}f(u) = \frac{d}{du}\conv_{[a,b]}f(u);
\]
\item \label{convex_fundamental_thm_3} $\frac{d}{du}\conv_{[a,b]} f$ is Lipschtitz-continuous with Lipschitz constant less or equal than $\lVert f''\lVert_{L^\infty(a,b)}$.
\end{enumerate}
\end{theorem}

By `differentiable on $[a,b]$' we mean that it is differentiable on $(a,b)$ in the classical sense and that in $a$ (resp. $b$) the right (resp. the left) derivative exists. While the proof is elementary, we give it for completeness.

\begin{proof}
(\ref{convex_fundamental_thm_1}) and (\ref{convex_fundamental_thm_2}). Let $h := \conv_{[a,b]} f$. Since $h$ is convex, then it admits left and right derivatives at each point $\bar u \in (a,b)$. Moreover in $a$ (resp. $b$) it admits right (resp. left) derivative. 

For each $\bar u \in (a,b)$, it holds 
\begin{equation}
\label{lr_der}
h'(\bar u-) \leq h'(\bar u+).
\end{equation}
In order to prove that $h$ is differentiable at $\bar u$ it is sufficent to prove that equality holds in \eqref{lr_der}.

If $h(\bar u) < f(\bar u)$, then, by Proposition \ref{shock}, $\bar u$ lies in a shock interval and so clearly $h$ is differentiable at $\bar u$ and the derivative is locally constant.

Hence we assume that $\bar u \in (a,b)$ is a point such that $h(\bar u) = f(\bar u)$. We claim that $f'(\bar u) \leq h'(\bar u-)$ and $h'(\bar u+) \leq f'(\bar u)$. By \eqref{lr_der}, this is sufficient to prove that $h$ is differentiable at $\bar u$ and that $h'(\bar u) = f'(\bar u)$. Assume by contradiction that $f'(\bar u) > h'(\bar u-)$. Then, by definition of left derivative, and by the fact that $h(\bar u) = f(\bar u)$, there exist $\eta >0, \delta >0$, such that for each $u \in (\bar u-\delta, \bar u)$, 
\[
f(\bar u) + f'(\bar u)(u -\bar u) < h(u) + \eta (u - \bar u).
\]
By Taylor expansion, for each $u \in (\bar u - \delta, \bar u)$,
\begin{equation}
\label{f_h}
\begin{split}
f(u) =&~ f(\bar u) + f'(\bar u)(u-\bar u) + o(u-\bar u) \\
     <&~ h(u) + \eta (u - \bar u) + o(u-\bar u),
\end{split}
\end{equation}
where $o(u - \bar u)$ is any quantity which goes to zero faster than $u - \bar u$. There exists $\delta'$ such that for each $u \in (\bar u - \delta', \bar u)$
\[
\eta (u - \bar u) + o(u-\bar u) \leq 0
\]
and so by \eqref{f_h}, $f(u) < h(u)$, a contradiction, since $h(u) = \conv_{[a,b]}f(u) \leq f(u)$ for each $u$. 
In a similar way one can prove that $h'(\bar u+) \leq f'(\bar u)$ and thus
\[
f'(\bar u) \leq h'(\bar u-) \leq h'(\bar u+) \leq f'(\bar u).
\]

(\ref{convex_fundamental_thm_3}). Let $u,v \in [a,b]$, $u < v$. Since $h$ is convex and differentiable, $h'(u) \leq h'(v)$. We want to estimate $h'(v) - h'(u)$, so let us assume that $h'(v) > h'(u)$. Then there exist $u_1, v_1 \in [a,b]$ such that $u \leq u_1 < v_1 \leq v$ and $h'(u) = f'(u_1)$, $h'(v) = h'(v_1)$. Indeed, if $h(u) = f(u)$, then you can choose $u_1 := u$; if $h(u) < f(u)$, you choose $u_1 := \sup I$, where $I$ is the maximal shock interval which $u$ belongs to (recall that we already know that $h$ is differentiable). In a similar way you choose $v_1$ and it holds $u_1 < v_1$. Hence
\[
\frac{h'(v)-h'(u)}{v-u} \leq \frac{f'(v_1) - f'(u_1)}{v_1 - u_1} \leq \lVert f'' \Vert_{L^{\infty}[a,b]}
\]
and so $h'$ is Lipschitz continuous with Lipschitz constant less or equal than $\lVert f'' \Vert_{L^{\infty}(a,b)}$.
\end{proof}

A similar result holds for the piecewise affine interpolation of a smooth function $f$.

\begin{proposition}
\label{convex_fundamental_thm_affine}
Let $\e >0$ be fixed. Assume $a,b \in \Z\e$. Let $f$ be a smooth function and let $f_\e$ be its piecewise affine interpolation with grid size $\e$. Then the derivative $\frac{d}{du}\conv_{[a,b]} f_\e$ of its convex envelope is a piecewise constant increasing function defined on $[a,b] \setminus \Z\e$, which enjoys the following Lipschtitz-like property:
for any $m<m'$ in $\Z$ such that $m\e,m'\e \in (a,b]$
\begin{equation}
\label{convex_fundamental_thm_4}
\begin{split}
\bigg(\frac{d}{du}\conv_{[a,b]} f_\e \bigg) \big( ( (m'-1)\e, m'\e ) \big) - \bigg(\frac{d}{du}\conv_{[a,b]} f_\e \bigg)& \big( ( (m-1)\e, m\e ) \big) \\
&\leq \lVert f''\lVert_{L^\infty(a,b)} \big( m'\e - (m-1)\e \big).
\end{split}
\end{equation}
\end{proposition}

\begin{proof}
Arguing as in Point (\ref{convex_fundamental_thm_3}) of Theorem \ref{convex_fundamental_thm}, assume that the l.h.s. of \eqref{convex_fundamental_thm_4} is strictly positive. Let $m_1\e:= \sup I$, where $I$ is the maximal shock interval which $\Big((m-1)\e,m\e \Big)$ belongs to (if such an interval $I$ does not exist, set $m_1:=m$). By definition of convex envelope 
$$f_\e\Big((m_1-1)\e\Big) \geq \conv_{[a,b]} f_\e\Big((m_1-1)\e\Big)$$
and by definition of maximal shock interval 
$$f_\e ( m_1\e ) = \conv_{[a,b]} f_\e ( m_1\e ).$$
Hence
\begin{equation}
\label{sinistra}
\begin{split}
\bigg(\frac{d}{du}\conv_{[a,b]} f_\e \bigg) \big( ( (m-1)\e, m\e ) \big) \geq&~
\bigg(\frac{d}{du} f_\e \bigg)\big( ( (m_1-1)\e, m_1\e )\big) \\
=&~ \frac{f_\e(m_1\e) - f_\e((m_1-1)\e)}{\e} \\
\text{(by definition of $f_\e$)} =&~ \frac{f(m_1\e) - f((m_1-1)\e)}{\e} \\
=&~ f'(\xi),
\end{split}
\end{equation}
for some $\xi \in ((m_1-1)\e,m_1\e)$.

Similarly you can find $\xi'$  such that $(m-1)\e \leq \xi < \xi' \leq m'\e$ and
\begin{equation}
\label{destra}
\bigg(\frac{d}{du}\conv_{[a,b]} f_\e \bigg) \big( ( (m'-1)\e, m'\e ) \big) \leq f'(\xi').
\end{equation}
Hence
\begin{equation*}
\begin{split}
\bigg(\frac{d}{du}\conv_{[a,b]} f_\e \bigg) \big( ((m'-1)\e, m'\e) \big) - \bigg(\frac{d}{du}\conv_{[a,b]} f_\e \bigg)& \big( ( (m-1)\e, m\e ) \big) \\
\text{(by \eqref{sinistra} and \eqref{destra})} \leq&~ f'(\xi') - f'(\xi) \\
\leq&~ \lVert f''\lVert_{L^\infty(a,b)} (\xi' - \xi) \\
\leq&~ \lVert f''\lVert_{L^\infty(a,b)} \Big[ m'\e - (m-1)\e \Big], 
\end{split}
\end{equation*}
concluding the proof.
\end{proof}

\subsection{Further estimates}

We are now able to state some useful results about convex envelopes, which we will frequently use in the following sections.

\begin{proposition}
\label{tocca}
Let $f: \R \to \R$ be continuous and let $a < \bar{u} < b$. If $\conv_{[a,b]}f(\bar{u}) = f(\bar{u})$, then 
\[
\conv_{[a,b]}f = \conv_{[a,\bar u]}f \cup \conv_{[\bar u,b]}f.
\]
\end{proposition}
\begin{proof}
We have to prove that 
\[
\conv_{[a,b]}f|_{[a, \bar u]} = \conv_{[a, \bar u]}f
\]
and 
\[
\conv_{[a,b]}f|_{[\bar u, b]} = \conv_{[\bar u, b]}f.
\]
Let $h:= \conv_{[a,b]}f|_{[a, \bar u]}$. By contradiction, assume $h \neq  \conv_{[a,\bar u]}f$. Then there exists a function $g$ defined on $[a, \bar u]$, convex, such that $h \leq g \leq f$ on $[a, \bar u]$ and such that $h(\tilde u) < g(\tilde u)$ for some $\tilde u \in (a,\bar u)$. Then a direct verification yields that $g \cup  \conv_{[a,b]}f|_{[\bar u, b]}$ is a convex function, and it is less or equal then $f$ on $[a,b]$. Hence, by definition of convex envelope, 
\[
g(\tilde u) = (g \cup  \conv_{[a,b]}f|_{[\bar u, b]})(\tilde u) \leq \conv_{[a,b]}f(\tilde u) = h(\tilde u) < g(\tilde u),
\]
a contradiction. In a similar way one can prove that $\conv_{[a,b]}f|_{[\bar u, b]} = \conv_{[\bar u, b]}f$. 
\end{proof}


\begin{corollary}
\label{RP_ridotto}
Let $f: \R \to \R$ be continuous and let $a < \bar{u} < b$. Assume that $\bar{u}$ belongs to a maximal shock interval $(u_1,u_2)$ with respect to $\conv_{[a,b]}f$. Then $\conv_{[a,\bar u]}f|_{[a,u_1]} = \conv_{[a,b]}f|_{[a,u_1]}$.
\end{corollary}
\begin{proof}
It is an easy consequence of Proposition \ref{tocca}, just observing that by maximality of $(u_1, u_2)$, $\conv_{[a,b]}f(u_1) = f(u_1)$.
\end{proof}

\begin{proposition}
\label{vel_aumenta}
Let $f: \R \to \R$ be continuous; let $a < \bar{u} < b$. Then
\begin{enumerate}
 \item $\big(\frac{d}{du}\conv_{[a,\bar{u}]}f\big)(u+) \geq \big(\frac{d}{du}\conv_{[a,b]}f\big)(u+)$ for each $u \in [a, \bar{u})$;
 \item $\big(\frac{d}{du}\conv_{[a,\bar{u}]}f\big)(u-) \geq \big(\frac{d}{du}\conv_{[a,b]}f\big)(u-)$ for each $u \in (a, \bar{u}]$; 
 \item $\big(\frac{d}{du}\conv_{[\bar{u},b]}f\big)(u+) \leq \big(\frac{d}{du}\conv_{[a,b]}f\big)(u+)$ for each $u \in [\bar{u},b)$;
 \item $\big(\frac{d}{du}\conv_{[\bar{u},b]}f\big)(u-) \leq \big(\frac{d}{du}\conv_{[a,b]}f\big)(u-)$ for each $u \in (\bar{u},b]$.
\end{enumerate}
\end{proposition}
\begin{proof}
We prove only the first point, the other ones being similar. Let $u \in [a, \bar u)$. 

If $\conv_{[a,b]}f(\bar u) = f(\bar u)$, then by Proposition \ref{tocca}, $\conv_{[a,b]}f|_{[a, \bar u]} = \conv_{[a,\bar u]}f$ and then we have done. 

Otherwise, if $\conv_{[a,b]}f(\bar u) < f(\bar u)$, denote by $(u_1, u_2)$ the maximal shock interval of $\conv_{[a,b]}f$, such that $\bar u \in (u_1, u_2)$. Hence, by Corollary \ref{RP_ridotto}, $\conv_{[a,b]}f|_{[a, u_1]} = \conv_{[a,\bar u]}f|_{[a, u_1]}$. Thus if $u \in [a, u_1)$, we have completed.

For $u = u_1$, if 
\begin{equation*}
\Big(\frac{d}{du}\conv_{[a,\bar{u}]}f\Big)(u_1+) < \Big(\frac{d}{du}\conv_{[a,b]}f\Big)(u_1+),
\end{equation*}
since $\conv_{[a,\bar u]} f(u_1) = \conv_{[a,b]}f(u_1)$, then $\conv_{[a,\bar u]} f < \conv_{[a,b]}f$ on $(u_1, u_1 + \delta)$ for some $\delta$ sufficiently small, but this is a contradiction, by definition of convex envelope.

Finally, if $u \in (u_1, \bar u)$, then 
\[
\begin{split}
\Big(\frac{d}{du}\conv_{[a,\bar{u}]}f\Big)(u+) \geq&~ \Big(\frac{d}{du}\conv_{[a,\bar{u}]}f\Big)(u_1+) \\
\geq&~ \Big(\frac{d}{du}\conv_{[a,b]}f\Big)(u_1+) \\
=&~ \Big(\frac{d}{du}\conv_{[a,b]}f\Big)(u).
\end{split}
\]
We have used that $\conv_{[a.b]} f$ is affine on $[u_1,u_2]$, by Proposition \ref{shock}.
\end{proof}

\begin{proposition}
\label{differenza_vel}
Let $f: \R \to \R$ be continuous; let $a < \bar{u} < b$. Then
\begin{enumerate}
\item for each $u_1, u_2 \in [a, \bar{u})$, $u_1 < u_2$,
\begin{align*}
\Big(\frac{d}{du}\conv_{[a,\bar{u}]}f\Big)(u_2+) - \Big(& \frac{d}{du}\conv_{[a,\bar{u}]}f\Big)(u_1+) \\
\geq&~ \Big(\frac{d}{du}\conv_{[a,b]}f\Big)(u_2+) - \Big(\frac{d}{du}\conv_{[a,b]}f\Big)(u_1+);  
\end{align*}
 \item for each $u_1, u_2 \in (a, \bar{u}]$, $u_1 < u_2$, 
\begin{align*}
\Big(\frac{d}{du}\conv_{[a,\bar{u}]}f\Big)(u_2-) - \Big(&\frac{d}{du}\conv_{[a,\bar{u}]}f\Big)(u_1-) \\ \geq&~
\Big(\frac{d}{du}\conv_{[a,b]}f\Big)(u_2-) - \Big(\frac{d}{du}\conv_{[a,b]}f\Big)(u_1-);
\end{align*}
 \item for each $u_1, u_2 \in [\bar{u},b)$, $u_1 < u_2$,
\begin{align*}
\Big(\frac{d}{du}\conv_{[\bar{u},b]}f\Big)(u_2+) -\Big(&\frac{d}{du}\conv_{[\bar{u},b]}f\Big)(u_1+) \\ \geq&~ \Big(\frac{d}{du}\conv_{[a,b]}f\Big)(u_2+) - \Big(\frac{d}{du}\conv_{[a,b]}f\Big)(u_1+);
\end{align*}
 \item for each $u_1, u_2 \in (\bar{u},b]$, $u_1 < u_2$, 
\begin{align*}
\Big(\frac{d}{du}\conv_{[\bar{u},b]}f\Big)(u_2-) -\Big(&\frac{d}{du}\conv_{[\bar{u},b]}f\Big)(u_1-) \\ \geq&~ \Big(\frac{d}{du}\conv_{[a,b]}f\Big)(u_2-) - \Big(\frac{d}{du}\conv_{[a,b]}f\Big)(u_1-).
\end{align*}
\end{enumerate}
\end{proposition}

\begin{proof}
Easy consequence of previous proposition.
\end{proof}

\begin{corollary}
\label{stesso_shock}
Let $f: \R \to \R$ be continuous and let $a < \bar{u} < b$. Let $u_1,u_2 \in [a,\bar u]$, $u_1<u_2$. If $u_1,u_2$ belong to the same shock interval of $\conv_{[a,\bar u]}f$, then they belong to the same shock interval of $\conv_{[a,b]}f$.
\end{corollary}

\begin{proposition}
\label{incastro}
Let $f: \R \to \R$ be smooth; let $u_1 < u_2 < u_3 < u_4 < u_5$ be real numbers. Assume that
\begin{enumerate}
 \item $\conv_{[u_1,u_4]}f(u_2) = f(u_2)$;
 \item $\conv_{[u_2,u_5]}f(u_3) = f(u_3)$.
\end{enumerate}
Then $\conv_{[u_1,u_5]}f(u_2) = f(u_2)$.
\end{proposition}

Let us first prove the following lemma.
\begin{lemma}
In the same setting of Proposition \ref{incastro}, let $g: [u_1,u_3] \to \R$, $h: [u_2,u_5] \to \R$ be convex functions. If $g = h$ on $[u_2,u_3]$, then $g \cup h|_{[u_3,u_5]}$ is convex.  
\end{lemma}
\begin{proof}
Immediate consequence of the fact that `being convex' is a local property.
\end{proof}

Let us now prove the proposition. 

\begin{proof}[Proof of Proposition \ref{incastro}]
Since $\conv_{[u_1,u_4]}f(u_2) = f(u_2)$, by Corollary \ref{RP_ridotto}, $\conv_{[u_1,u_3]}f(u_2) = f(u_2)$; hence 
\begin{equation*}
g := \conv_{[u_1,u_3]} f = \conv_{[u_1,u_2]} f \cup \conv_{[u_2,u_3]} f.  
\end{equation*}
Moreover, by Proposition \ref{tocca},
\begin{equation*}
h:= \conv_{[u_2,u_5]} f = \conv_{[u_2,u_3]} f \cup \conv_{[u_3,u_5]} f.  
\end{equation*}
Hence, by previous lemma, $\varphi:= g \cup h|_{[u_3,u_5]}$ is convex; moreover, by definition, $\varphi \leq f$ and, for this reason, $\varphi \leq \conv_{[u_1,u_5]}f$. Hence
\[
f(u_2) = \conv_{[u_1,u_2]}f(u_2) = g(u_2) = \varphi(u_2) \leq \conv_{[u_1,u_5]}f (u_2)\leq f(u_2)
\]
and so the thesis follows.
\end{proof}

\begin{proposition}
\label{convex_approximation}
Let $f$ be a continuous function on $\R$. Let $a < b$ and let $a_k < b_k$, $k \in \N$, two sequences such that $a_k \searrow a, b_k \nearrow b$. Assume that for each $k \in \N$ there is $u_k \in [a_k,b_k]$ such that 
\[
\conv_{[a_k,b_k]} f(u_k) = f(u_k).
\]
If $u_k \rightarrow \bar u$, then
\[
\conv_{[a,b]} f(\bar u) = f(\bar u).
\]
\end{proposition}
\begin{proof}
By simplicity, assume $a_k = a$ for each $k$, the general case being entirely similar. By contradiction, assume 
$\conv_{[a,b]} f(\bar u) < f(\bar u)$. There is a maximal shock interval $(\tilde a, \tilde b)$ containing $\bar u$ and w.l.o.g. we can assume $\tilde a = a$. 

If $\tilde b < b$, for $k$ sufficiently large, $u_k < \tilde b < b_k \leq b$. Hence
\[
f(\tilde b) = \conv_{[a,b]} f(\tilde b) \overset{\text{(Cor. \ref{RP_ridotto})}}{=} \conv_{[a,b_k]} f(\tilde b),
\]
and thus
\[
f(u_k) = \conv_{[a,b_k]} f(u_k) \overset{\text{(Prop. \ref{tocca})}}{=} \conv_{[a,\tilde b]} f(u_k) \overset{\text{(Prop. \ref{tocca})}}{=} \conv_{[a,b]} f(u_k).
\]
Passing to the limit, we get a contradiction.

On the other hand, if $\tilde b = b$, one can find a spherical neighborhood $N_1$ of $(\bar u, f(\bar u))$ in $\R^2$ and a spherical neighborhood $N_2$ of $(b,f(b))$ in $\R^2$ with the following property: for any $(u_1,v_1) \in N_1$ and for any $(u_2,v_2) \in N_2$, if $r = r(u)$ is the line joining $(a,f(a))$ and  $(u_2,v_2)$, then $v_1 > r(u_1)$. Now observe that by continuity of $f$, for $k$ sufficiently large, $(u_k,f(u_k)) \in N_1$ and $(b_k,f(b_k)) \in N_2$. Hence, denoting by $r_k$ the line joining $(a,f(a))$ and $(b_k,f(b_k))$, we get 
\[
\conv_{[a,b_k]} f(u_k) = f(u_k) > r_k(u_k),
\]
a contradiction, since $\conv_{[a,b_k]}$ is a convex function whose graph contains points $(a,f(a)), (b_k,f(b_k))$.
\end{proof}

\begin{proposition}
\label{diff_vel_proporzionale_canc}
Let $f$ be a smooth function, let $a < \bar u < b$. Then
\begin{equation*}
\bigg(\frac{d}{du}\conv_{[a,\bar u]} f\bigg)(\bar u-) - \bigg(\frac{d}{du} \conv_{[a,b]}f\bigg)(\bar u) 
\leq \lVert f'' \lVert_{L^\infty(a,b)} (b - \bar u).
\end{equation*}
Moreover, if $f_\e$ is the piecewise affine interpolation of $f$ with grid size $\e$, it holds
\begin{equation*}
\bigg(\frac{d}{du}\conv_{[a,\bar u]} f_\e\bigg)(\bar u -) - \bigg(\frac{d}{du} \conv_{[a,b]}f_\e\bigg)(\bar u-) 
\leq \lVert f'' \lVert_{L^\infty(a,b)} (b - \bar u).
\end{equation*}
\end{proposition}

\begin{proof}
Let us first prove the inequality for $f$ smooth. Set $g := \conv_{[a,b]}f$, $h := \conv_{[a,\bar u]}f$. Let $(c,d) \subseteq [a,b]$ be the maximal shock interval of $g$ which $\bar u$ belongs to (if it does not exist, the proof is trivial, because of Theorem \ref{convex_fundamental_thm}). Let $r = r(u)$ be the line passing through $(\bar u, f(\bar u))$ with slope $h'(\bar u)$ and let $\tilde c$ be the first coordinate of the intersection point between $r(u)$ and $g$ in the interval $[a,\bar u]$.

It holds
\[
h'(\bar u-) = \frac{r(\bar u) - r(\tilde c)}{\bar u - \tilde c},
\]
and
\begin{equation*}
g'(\bar u)  =  \frac{g(d) - g(\tilde c)}{d - \tilde c}. 
\end{equation*}
Moreover define
\begin{eqnarray*}
\lambda :=  \frac{g(d) - r(\bar u)}{d-\bar u}.
\end{eqnarray*}

Clearly
\[
g'(\bar u)(d-\tilde c) = h'(\bar u-)(\bar u - \tilde c) + \lambda (d-\bar u),
\]
and thus
\[
g'(\bar u) = \frac{h'(\bar u-)(\bar u - \tilde c) + \lambda (d-\bar u)}{d -\tilde c}.
\]

Hence
\begin{equation}
\label{diff_lambda}
h'(\bar u-) - g'(\bar u) = \frac{h'(\bar u-) - \lambda}{d-\tilde c}(d-\bar u).
\end{equation}

Now observe that there must be $\xi \in [\tilde c,\bar u]$ such that $h'(\bar u-) = f'(\xi)$. Indeed, if there exists a strictly increasing sequence $(u_n)$ converging to $\bar u$ such that $f(u_n) = h(u_n) = \conv_{[a,\bar u]} f(u_n)$, then by Theorem \ref{convex_fundamental_thm}, Point \ref{convex_fundamental_thm_2}, $f'(u_n) = h'(u_n)$ for each $n$; passing to the limit one obtains $f'(\bar u) = h'(\bar u-)$. Otherwise, if such a sequence does not exist, then one can easily find $c^* \geq c$ such that $(c^*, \bar u)$ is a maximal shock interval. This means that
\[
h'(\bar u-) = \frac{h(\bar u) - h(c^{*})}{\bar u -  c^{*}} = \frac{f(\bar u) - f(c^{*})}{\bar u - c^{*}} = f'(\xi),
\]
for some $\xi \in (\tilde c,\bar u)$.

Moreover, since
\[
\lambda =  \frac{g(d) - r(\bar u)}{d-\bar u} = \frac{f(d) - f(\bar u)}{d - \bar u},
\]
there must be $\eta \in (\bar u, d)$ such that $\lambda = f'(\eta)$.

Hence, from \eqref{diff_lambda}, we obtain
\begin{align*}
h'(\bar u-) - g'(\bar u) =&~ \frac{h'(\bar u-) - \lambda}{d-\tilde c}(d-\bar u) \\
=&~ \frac{f'(\xi) - f'(\eta)}{d-\tilde c}(d-\bar u) \\ 
\leq&~ \lVert f'' \lVert_{L^\infty(a,b)} (d-\bar u) \\
\leq&~ \lVert f'' \lVert_{L^\infty(a,b)} (b-\bar u).
\end{align*}

Concerning the piecewise affine case, substituting $f$ with $f_\e$ one obtains the same inequality \eqref{diff_lambda}. Now observe that as in the smooth case one can find $\eta \in (\bar u, d)$ such that $\lambda = f'(\eta)$. Moreover it is also easy to see that there must be $\xi \in [\tilde c, \bar u]$ such that $h'(\bar u) = f'(\xi)$. Using this, one concludes the proof as in the smooth case.
\end{proof}


\section{Wavefront Tracking Approximation}
\label{sect_wavefront}

In this section we prove the main interaction estimate \eqref{E_W_est_fin} for an approximate solution of the Cauchy problem \eqref{cauchy} obtained by the wavefront tracking algorithm, thus giving a proof of Theorem \ref{W_main_thm}. First we briefly recall how such an approximate solution is constructed, mainly with the aim to set the notations used later on.

For any $\e > 0$ let $f_\e$ be the piecewise affine interpolation of $f$ with grid size $\e$; let $\bar u_\e$ an approximation of the initial datum $\bar u$ (in the sense that $\bar u_\e \rightarrow \bar u$ in $L^1$-norm, as $\e \rightarrow 0$) of the Cauchy problem \eqref{cauchy}, such that $\bar u_\e$ has compact support, it takes values in the discrete set $\Z\e$, and 
\begin{equation}
\label{bd_su_dato_iniziale}
\TV(\bar u_\e) \leq \TV(\bar u).
\end{equation}
Let $\xi_1 < \dots < \xi_K$ be the points where $\bar u_\e$ has a jump. At each $\xi_k$, consider the left and the right limits $\bar u_\e (\xi_k -), \bar u_\e(\xi_k+) \in \Z\e$. Solving the corresponding Riemann problems with flux function $f_\e$, we thus obtain a local solution $u_\e = u_\e(t) = u_\e(t,x)$, defined for $t$ sufficiently small. From each $\xi_k$ some wavefronts supported on discontinuity lines of $u_\e$, referred also as shocks (positive or negative, according to the sign of the jump) or contact discontinuities, emerge. When two (or more) discontinuity lines supporting wavefronts meet (we will refer to this situation as an \emph{interaction-cancellation}), we can again solve the new Riemann problem generated by the interaction-cancellation, according to the above procedure, with flux $f_\e$, since the values of $u_\e(t,\cdot)$ always remain within the set $\Z\e$. The solution is then prolonged up to a time $t_2 > t_1$ where other wavefronts meet, and so on. \\
One can prove \cite{bre_00}, \cite{daf_72} that the total number of interaction-cancellation points is finite, and hence the solution can be prolonged for all $t \geq 0$, thus defining an approximate solution $u_\e = u_\e(t,x)$, piecewise constant, with values in the set $\Z\e$. 

Let $\{(t_j,x_j)\}$, $j \in \{1,2,\dots,J\}$, be the point in the $t,x$-plane where an interaction-cancellation between two (or more) wavefronts occurs in the approximate solution $u_\e$. Let us suppose that $t_j < t_{j+1}$ and for every $j$ exactly two discontinuities meet in $(t_j, x_j)$. This is a standard assumption, achieved by slightly perturbing the wavefront speed. We also set $t_0 := 0$.


\subsection{Definition of waves for wavefront tracking}
\label{Front_Waves}

In this section we define the notion of \emph{wave}, the notion of \emph{position of a wave} and the notion of \emph{speed of a wave}. By definition of wavefront solution, for each time $t \geq 0$, $u_\e(t,\cdot)$ is a piecewise constant function, which takes values in the set $\Z\e$. Hence $\TV(u_\e(0,\cdot))$ is an integer multiple of $\e$.

\subsubsection{Enumeration of waves}
\label{Front_eow}

In this section we define the notion of \emph{enumeration of waves} related to a function $u: \R_x \to \R$ of the single variable $x$: in the following sections, $u$ will be the piecewise constant, $\e$-approximate solution of the Cauchy problem \eqref{cauchy} for fixed time $t$, considered as a function of $x$.

\begin{definition}
\label{W_eow}
Let $u: \R \to \R$, $u \in BV(\R)$, be a piecewise constant, right continuous function, which takes values in the set $\Z\e$. An \emph{enumeration of waves} for the function $u$ is a 3-tuple $(\mathcal{W}, \mathtt x, \hat u)$, where 
\[
\begin{array}{ll}
\mathcal{W} \subseteq \N & \text{is \emph{the set of waves}},\\
\mathtt x: \mathcal{W} \to (-\infty, +\infty] & \text{is \emph{the position function}}, \\
\hat u: \mathcal{W} \to \Z\e & \text{is \emph{the right state function}}, \\
\end{array}
\]
with the following properties:
\begin{enumerate}
 \item the function $\mathtt x$ takes values only in the set of discontinuity points of $u$;
 \item the restriction of the function $\mathtt x$ to the set of waves where it takes finite values is increasing; 
 \item for given $x_0 \in \R$, consider $\mathtt x^{-1}(x_0) = \{s \in \mathcal{W} \ | \ \mathtt x(s) = x_0\}$; then it holds: \begin{enumerate}
   \item if $u(x_0-) < u(x_0)$, then $\hat u|_{\mathtt x^{-1}(x_0)}: \mathtt x^{-1}(x_0) \to (u(x_0-), u(x_0)] \cap \Z\e$ is strictly increasing and bijective; 
   \item if $u(x_0-) > u(x_0)$, then $\hat u|_{\mathtt x^{-1}(x_0)}: \mathtt x^{-1}(x_0) \to [u(x_0), u(x_0-)) \cap \Z\e$ is strictly decreasing and bijective; 
   \item if $u(x_0-) = u(x_0)$, then $\mathtt x^{-1}(x_0) = \emptyset$.
 \end{enumerate}
\end{enumerate}
\end{definition}

Given an enumeration of waves as in Definition \ref{W_eow}, we define the \emph{sign of a wave $s \in \W$} with finite position (i.e. such that $\mathtt x(s) < +\infty$) as follows:
\begin{equation}
\label{W_sign}
\mathcal{S}(s) := \sign \Big[u(\mathtt x(s)) - u(\mathtt x(s)-) \Big].
\end{equation}

We immediately present an example of enumeration of wave which will be fundamental in the sequel. 

\begin{example}
\label{W_initial_eow}
Fix $\e>0$ and let $\bar u_\e \in BV(\R)$ be the approximate initial datum of the Cauchy problem \eqref{cauchy}, with compact support and taking values in $\Z\e$. The total variation of $\bar u_\e$ is an integer multiple of $\e$. Let 
\[
U: \R \to [0, \TV(\bar u_\e)], \quad x \mapsto U(x) := \TV(\bar u_\e; (-\infty, x]),
\]
be the total variation function. Then define:
\[
\mathcal{W} := \Big\{1,2,\dots, \frac{1}{\e}\TV(\bar u_\e) \Big\}
\]
and
\[
\mathtt x_0 : \mathcal{W} \to (-\infty, +\infty], \quad s \mapsto \mathtt x_0(s) := \inf \Big\{x \in (-\infty, +\infty] \ | \ \e s \leq U(x) \Big\}. 
\]
Moreover, recalling \eqref{W_sign}, we define 
\begin{equation*}
\hat u: \mathcal{W} \to \R, \quad s \mapsto \hat u(s) := \bar u_\e(\mathtt x_0(s)-) + \mathcal{S}(s)\Big[\e s - U(\mathtt x_0(s)-) \Big ].
\end{equation*}

It is fairly easy to verify that $\mathtt x_0$, $\hat u$ are well defined and that they provide an enumeration of waves, in the sense of Definition \ref{W_eow}.
\end{example}


Let us now give another definition.

\begin{definition}
\label{W_speed_function}
Consider a function $u$ as in Definition \ref{W_eow} and let $(\mathcal{W}, \mathtt x, \hat u)$ be an enumeration of waves for $u$. The \emph{speed function} $\sigma: \mathcal{W} \to [0,1] \cup \{+\infty\}$ is defined as follows:
\begin{equation}
\label{E_speed_W_def_2}
\sigma(s) := \left\{
\begin{array}{ll}
+\infty & \text{if } \mathtt x(s) = +\infty, \\
\Big(\frac{d}{du}\conv_{[u(\mathtt x(s)-),u(\mathtt x(s))]}f_\e \Big) \big( (\hat u(s)-\e, \hat u(s) ) \big) & \text{if } \mathcal{S}(s) = +1, \\ 
\Big(\frac{d}{du}\conc_{[u(\mathtt x(s)),u(\mathtt x(s)-)]}f_\e \Big) \big( (\hat u(s), \hat u(s) +\e) \big) & \text{if } \mathcal{S}(s) = -1. 
\end{array}
\right.
\end{equation}
\end{definition}
Roughly speaking, $\sigma(s)$ is the speed given to the wave $s$ by the Riemann problem located at $\mathtt x(s)$.

\begin{remark}
\label{W_speed_increasing_wrt_waves}
Notice that for each $x_0 \in \R$, $\sigma$ restricted to $\mathtt x^{-1}(x_0)$ is increasing by the monotonicity properties of derivatives the convex/concave envelopes (or by the study of the Riemann solver).
\end{remark}


\subsubsection{Position and speed of the waves}
\label{W_pswaves}

Consider the Cauchy problem \eqref{cauchy} and fix $\e >0$; let $u_\e$ be piecewise constant wavefront solution.
For the initial datum $u_\e(0,\cdot)$, consider the enumeration of waves $(\mathcal{W}, \mathtt x_0, \hat u)$ provided in Example \ref{W_initial_eow}; let $\mathcal{S}$ be the sign function defined in \eqref{W_sign} for this enumeration of waves. 

Now our aim is to define two functions
\[
\mathtt x: [0, +\infty)_t \times \mathcal{W} \to \R_x \cup \{+\infty\}, \qquad
\sigma: [0, +\infty)_t \times \mathcal{W} \to [0,1] \cup \{+\infty\},
\]
called \emph{the position at time $t \in [0, +\infty)$ of the wave $s \in \mathcal{W}$} and \emph{the speed at time $t \in [0, +\infty)$ of the wave $s \in \mathcal{W}$}. As one can imagine, we want to describe the path followed by a single wave $s \in \mathcal{W}$ as time goes on and the speed assigned to it by the Riemann problems it meets along the way. Even if there is a slight abuse of notation (in this section $\mathtt x$ depends also on time), we believe that the context will avoid any misunderstanding.

The function $\mathtt x$ is defined by induction, partitioning the time interval $[0, +\infty)$ in the following way
\[
[0, +\infty) = \{0\} \cup (0, t_1] \cup \dots \cup (t_j, t_{j+1}] \cup \dots \cup (t_{J-1}, t_J] \cup (t_J, +\infty).
\]

First of all, for $t = 0$ we set $\mathtt x(0,s) := \mathtt x_0(s)$, where $\mathtt x_0(\cdot)$ is the position function in the enumeration of waves of Example \ref{W_initial_eow}. Clearly $(\mathcal{W}, \mathtt x(0, \cdot), \hat u)$ is an enumeration of waves for the function $u_\e(0, \cdot)$ as a function of $x$ ($\hat u$ being the right state function, as in the example above).

%

Assume to have defined $\mathtt x(t,\cdot)$ for every $t \leq t_j$ and let us define it for $t \in (t_j, t_{j+1}]$ (or $t \in (t_J, +\infty)$). The speed $\sigma(t,s)$ for $t \in [0,t_j]$ is computed accordingly to \eqref{E_speed_W_def_2}.

\noindent For $t < t_{j+1}$ (or $t_J < t < +\infty$) set
\begin{equation*}
\mathtt x(t,s) := \mathtt x(t_j, s) + \sigma(t_j, s)(t-t_j).
\end{equation*}

\noindent For $t=t_{j+1}$ set
\begin{equation*}
\mathtt x(t_{j+1},s) := \mathtt x(t_j,s) +\sigma(t_j,s)(t_{j+1} - t_j)
\end{equation*}
if $\mathtt x(t_j,s) +\sigma(t_j,s)(t_{j+1} - t_j)$ is not the point of interaction/cancellation $x_{j+1}$; otherwise for the waves $s$ such that $\mathtt x(t_j,s) +\sigma(t_j,s)(t_{j+1} - t_j) =  x_{j+1}$ and
\[
\mathcal{S}(s)u_\e(t_{j+1},x_{j+1}-) \leq \mathcal{S}(s) \hat u(s) - \e \leq \mathcal{S}(s) \hat u(s) \leq \mathcal{S}(s) u_\e(t_{j+1},x_{j+1})\]
(i.e. the ones surviving the possible cancellation in $(t_{j+1},x_{j+1})$)
define
\[
\mathtt x(t_{j+1},s) := \mathtt x(t_j,s) +\sigma(t_j,s)(t_{j+1} - t_j) =  x_{j+1}.
\]
where $\mathcal{S}(s)$ is defined in \eqref{W_sign}, using the enumeration of waves for the initial datum. To the waves $s$ canceled by a possible cancellation in $(t_{j+1},x_{j+1})$ we assign $\mathtt x(t_{j+1},s) := +\infty$.

The following lemma proves that the above procedure produces an enumeration of waves. 

\begin{lemma}
\label{W_lemma_eow}
For any $\bar t \in (t_j, t_{j+1}]$ (resp. $\bar t \in (t_J, +\infty)$), the 3-tuple $(\mathcal{W}, \mathtt x(\bar t, \cdot), \hat u)$ is an enumeration of waves for the piecewise constant function $u_\e(\bar t,\cdot)$.
\end{lemma}

\begin{proof} We prove separately that the Properties (1-3) of Definition \ref{W_eow} are satisfied.

\smallskip
\noindent {\it Proof of Property (1).} By definition of wavefront solution, $\mathtt x(\bar t, \cdot)$ takes values only in the set of discontinuity points of $u_\e(\bar t, \cdot)$.

\smallskip
\noindent {\it Proof of Property (2).} Let $s < s'$ be two waves and assume that $\mathtt x(\bar t, s), \mathtt x(\bar t, s') < +\infty$. By contradiction, suppose that $\mathtt x(\bar t, s) > \mathtt x(\bar t, s')$.  Since by the inductive assumption at time $t_j$, the 3-tuple $(\mathcal{W}, \mathtt x(t_j, \cdot), \hat u)$ is an enumeration of waves for the function $u_\e(t_j,\cdot)$, it holds $\mathtt x(t_j,s) \leq \mathtt x(t_j,s')$. Two cases arise:
\begin{itemize}
\item If $\mathtt x(t_j, s) = \mathtt x(t_j, s')$, then it must hold $\sigma(t_j, s) > \sigma(t_j, s')$, but this is impossible, due to Remark \ref{W_speed_increasing_wrt_waves}. 

\item If $\mathtt x(t_j, s) < \mathtt x(t_j, s')$, then lines $t \mapsto \mathtt x(t_j,s) + \sigma(t_j, s)(t-t_j)$ and $t \mapsto \mathtt x(t_j,s') + \sigma(t_j, s')(t-t_j)$ must intersect at some time $\tau \in (t_j, \bar t)$, but this is impossible, by definition of wavefront solution and times $(t_j)_j$.
\end{itemize}

\smallskip
\noindent {\it Proof of Property (3).}  For $t < t_{j+1}$ or $t=t_{j+1}$ and for discontinuity points $x \neq x_{j+1}$, the third property of an enumeration of waves is straightforward. So let us check the third property only for time $t = t_{j+1}$ and for the discontinuity point $x_{j+1}$. Fix any time $\tilde t \in (t_j, t_{j+1})$; according to assumption on binary intersections, you can find two points $\xi_1, \xi_2 \in \R$ such that for any $s$ such that $\mathtt x(t_j,s) + \sigma(t_j,s)(t_{j+1} - t_j) = x_{j+1}$, either $\mathtt x(\tilde t,s) = \xi_1$ or $\mathtt x(\tilde t,s) = \xi_2$ and moreover $u_\e(\tilde t, \xi_1-) = u_\e(t_{j+1}, x_{j+1}-)$, $u_\e(\tilde t, \xi_2) = u_\e(t_{j+1}, x_{j+1})$,  $u_\e(\tilde t, \xi_1) = u_\e(\tilde t, \xi_2-)$. 

We now just consider two main cases: the other ones can be treated similarly. Recall that at time $\tilde t < t_{j+1}$, the $3$-tuple $(\W, \mathtt x(\tilde t, \cdot), \hat u)$ is an enumeration of waves for the piecewise constant function $u_\e(\tilde t, \cdot)$.

If $u_\e(\tilde t, \xi_1-) < u_\e(\tilde t, \xi_1) = u_\e(\tilde t,\xi_2-) < u_\e(\tilde t, \xi_2)$, then 
$$\hat u|_{\mathtt x^{-1}(\tilde t,\xi_1)}: \mathtt x^{-1}(\tilde t, \xi_1) \to (u_\e(\tilde t, \xi_1-), u_\e(\tilde t, \xi_1)] \cap \Z\e$$ 
and 
$$\hat u|_{\mathtt x^{-1}(\tilde t,\xi_2)}: \mathtt x^{-1}(\tilde t, \xi_2) \to (u_\e(\tilde t, \xi_2-), u_\e(\tilde t, \xi_2)] \cap \Z\e$$ 
are strictly increasing and bijective; observing that in this case $\mathtt x^{-1}(t_{j+1},x_{j+1}) = \mathtt x^{-1}(\tilde t, \xi_1) \cup \mathtt x^{-1}(\tilde t, \xi_2)$, one gets the thesis.

If $u_\e(\tilde t, \xi_1-) < u_\e(\tilde t, \xi_2) < u_\e(\tilde t, \xi_1) = u_\e(\tilde t, \xi_2-)$, then 
$$\hat u|_{\mathtt x^{-1}(\tilde t,\xi_1)}: \mathtt x^{-1}(\tilde t, \xi_1) \to (u_\e(\tilde t, \xi_1-), u_\e(\tilde t, \xi_1)] \cap \Z\e$$
is strictly increasing and bijective; observing that in this case 
$$\mathtt x^{-1}(t_{j+1},x_{j+1}) = \Big\{s \in \mathtt x^{-1}(\tilde t, \xi_1) \ | \ \hat u(s) \in (u_\e(\tilde t, \xi_1-), u_\e(\tilde t, \xi_2)] \Big\},$$
one gets the thesis. 
\end{proof}

\begin{remark}
For fixed wave $s$, functions $\mathtt x(\cdot, s), \sigma(\cdot,s)$ are right-continuous. Moreover $\sigma(\cdot,s)$ is piecewise constant.
\end{remark}


Finally we introduce the following notation. Given a time $t \in [0, +\infty)$ and a position $x \in (-\infty, +\infty]$, we set
\begin{equation*}
\W(t) :=  \{s \in \W \ | \ \mathtt x(t,s) < +\infty \}, \qquad
\W(t,x) := \{s \in \W \ | \ \mathtt x(t,s) = x \}.
\end{equation*}
We will call $\W(t)$ the set of the \emph{real waves}, while we will say that a wave $s$ is \emph{removed} at time $t$ if $\mathtt x(t,s) = + \infty$. It it natural to define the interval of existence of $s \in \W(0)$ by 
\[
T(s) := \sup \Big\{t \in [0, +\infty) \ | \ \mathtt x(t,s) < +\infty \Big\}.
\]

\subsubsection{Interval of waves}
\label{W_iow}

In this section we define the interval of waves and we prove an important result about it. The necessity of introducing this notion is due to the fact that we avoid relabeling the waves $s$.

\begin{definition}
\label{W_iow_def}
Let $\bar t$ be a fixed time and $\mathcal{I} \subseteq \W(\bar t)$. We say that $\mathcal{I}$ is an \emph{interval of waves} at time $\bar t$ if for any given $s_1, s_2 \in \mathcal{I}$, with $s_1 < s_2$, and for any $p \in \W(\bar t)$
\[
s_1 \leq p \leq s_2 \Longrightarrow p \in \mathcal{I}.
\]
We say that an interval of waves $\mathcal{I}$ is \emph{homogeneous} if for each $s,s' \in \mathcal{I}$, $\mathcal{S}(s) = \mathcal{S}(s')$. If waves in $\mathcal{I}$ are positive (resp. negative), we say that $\mathcal{I}$ is a \emph{positive} (resp. \emph{negative}) interval of waves. 
\end{definition}

\begin{proposition}
\label{W_interval_waves}
Let $\mathcal{I} \subseteq \W(\bar t)$ be a positive (resp. negative) interval of waves. Then the restricion of $\hat u$ to $\mathcal{I}$ is strictly increasing (resp. decreasing) and $\bigcup_{s \in \mathcal{I}} (\hat u(s) - \e, \hat u(s)]$ (resp. $\bigcup_{s \in \mathcal{I}} [\hat u(s), \hat u(s)+\e)$) is an interval in $\R$.
\end{proposition}
\begin{proof}
Assume $\mathcal{I}$ is positive, the other case being similar. First we prove that $\hat u$ restricted to $\mathcal{I}$ is increasing. Let $s, s' \in \mathcal{I}$, with $s < s'$. Let $\xi_0 := \mathtt x(\bar t, s) < \xi_1 < \dots < \xi_K := \mathtt x(\bar t, s')$ be the discontinuity points of $u_\e(\bar t, \cdot)$ between $\mathtt x(\bar t,s)$ and $\mathtt x(\bar t,s')$. By definition of `interval of waves' and by the fact that each wave in $\mathcal{I}$ is positive, for any $k = 0, 1, \dots, K$, $\W(\bar t, \xi_k)$ contains only positive waves. Thus, by Definition \ref{W_eow} of enumeration of waves, and by the fact that for each $k=0, \dots, K-1$, $u_\e(\bar t, \xi_k) = u_\e(\bar t, \xi_{k+1}-)$, the restriction 
\begin{equation}
\label{W_hat_u_interval_waves}
\hat u : \bigcup_{k=0}^{K} \W(\bar t, \xi_k) \to (u_\e(\bar t, \xi_0 -),u_\e(\bar t, \xi_K)] \cap \Z\e
\end{equation}
is strictly increasing and bijective, and so $\hat u(s) < \hat u(s')$; hence $\hat u|_{\mathcal{I}}$ is strictly increasing.

In order to prove that $\bigcup_{s \in \mathcal{I}} (\hat u(s) - \e, \hat u(s)]$ is a interval in $\R$, it is sufficient to prove the following: for any $s<s'$ in $ \mathcal{I}$ and for any $m \in \Z$ such that $\hat u(s) < m\e \leq \hat u(s')$, there is $p \in \mathcal{I}$, $s < p \leq s'$ such that $\hat u (p) = m\e$. This follows immediately from the fact that the map in \eqref{W_hat_u_interval_waves} is bijective and strictly increasing.
\end{proof}

\begin{definition}
\label{W_artificial_speed}
Let $I \subseteq \R$ be an interval in $\R$, such that $\inf I, \sup I \in \Z\e$. Let $s$ be any positive (resp. negative) wave such that $(\hat u(s) - \e, \hat u(s)) \subseteq I$ (resp. $(\hat u(s), \hat u(s)+\e) \subseteq I$). The quantity $\sigma(I,s) := \frac{d}{du} \conv_I f_\e \Big( (\hat u(s) - \e, \hat u(s) ) \Big)$ (resp. $\sigma(I,s) := \frac{d}{du} \conc_I f_\e \Big( (\hat u(s), \hat u(s)+\e ) \Big)$) is called \emph{the (artificial) speed given to the wave $s$ by the Riemann problem $I$}.

Moreover we will say that \emph{the Riemann problem $I$ divides $s,s'$} if $(\hat u(s) - \e, \hat u(s)), (\hat u(s') - \e, \hat u(s'))$ (resp.  $(\hat u(s), \hat u(s)+\e), (\hat u(s'), \hat u(s')+\e)$ ) do not belong to the same shock component of $\conv_I f_\e$ (resp. $\conc_I f_\e$).
\end{definition}

\begin{remark}
\label{W_artificial_speed_remark}
Let $\mathcal{I}$ be any positive (resp. negative) interval of waves at fixed time $\bar t$. By Proposition \ref{W_interval_waves}, the set $I := \bigcup_{s \in \mathcal{I}} (\hat u(s) -\e, \hat u(s)]$ 
(resp. $I = \bigcup_{s \in \mathcal{I}} [\hat u(s), \hat u(s)+\e)$) is an interval in $\R$. Hence, we will also write $\sigma(\mathcal{I},s)$ instead of $\sigma(I,s)$ and call it the speed given to the waves $s$ by the Riemann problem $\mathcal{I}$. Moreover, we will also say that the Riemann problem $\mathcal{I}$ divides $s,s'$ if the Riemann problem $I$ does.
\end{remark}

\subsection{The main theorem in the wavefront tracking approximation}
\label{section_W_main_thm}

Now we state the main result for the wavefront tracking approximation, namely Theorem \ref{W_main_thm}. For easiness of the reader we repeat the statement below.

As in the previous section, let $u_{\e} = u_{\e}(t,x)$ be an $\e$-wavefront solution of the Cauchy problem \eqref{cauchy}; consider the enumeration of waves and the related position function $\mathtt x = \mathtt x(t,s)$ and speed function $\sigma = \sigma(t,s)$ constructed in previous section.  Fix a wave $s \in \W(0)$ and consider the function $t \mapsto \sigma(t,s)$. By construction it is finite valued until the time $T(s)$, after which its value becomes $+\infty$; moreover it is piecewise constant, right continuous, with jumps possibly located at times $t = t_j, j \in 1,\dots,J$.

The results we are going to prove is

\begin{theorem2}
The following holds:
\[
\sum_{j=1}^{J} \sum_{s \in \W(t_j)} |\sigma(t_j, s) - \sigma(t_{j-1}, s)||s| \leq (3 + 2 \log(2)) \lVert f'' \lVert_{L^\infty} \TV(u(0,\cdot))^2,
\]
where $|s| := \e$ is \emph{the strength of the wave $s$}.
\end{theorem2}

We recall the following definition. 

\begin{definition}
\label{W_int_canc_points}
For each $j=1,\dots,J$, we will say that $(t_j,x_j)$ is an \emph{interaction point} if the wavefronts which collide in $(t_j,x_j)$ have the same sign. An interaction point will be called \emph{positive} (resp. \emph{negative}) if all the waves located in it are positive (resp. negative). Moreover we will say that $(t_j,x_j)$ is a \emph{cancellation point} if the wavefronts which collide in $(t_j,x_j)$ have opposite sign. 
\end{definition}

The first step in order to prove Theorem \ref{W_main_thm} is to reduce the quantity we want to estimate, namely
\[
\sum_{j=1}^{J} \sum_{s \in \W(t_j)} |\sigma(t_j, s) - \sigma(t_{j-1}, s)||s|,
\]
to a two separate estimates, according to $(t_j,x_j)$ being an interaction or a cancellation:
\[
\begin{split}
\sum_{j=1}^{J} \sum_{s \in \W(t_j)} |\sigma(t_j, s) - \sigma(t_{j-1}, s)||s| 
=&~ \sum_{\substack{(t_j,x_j) \\ \text{interaction}}} \sum_{s \in \W(t_j)} |\sigma(t_j, s) - \sigma(t_{j-1}, s)||s| \\
&~ + \sum_{\substack{(t_j,x_j) \\ \text{cancellation}}} \sum_{s \in \W(t_j)} |\sigma(t_j, s) - \sigma(t_{j-1}, s)||s|.
\end{split}
\]

The estimate on the cancellation points is fairly easy. First of all define for each cancellation point $(t_j,x_j)$ the \emph{amount of cancellation} as follows:
\begin{equation}
\label{W_canc_0}
\mathcal{C}(t_j,x_j) := \TV(u_\e(t_{j-1}, \cdot)) - \TV(u_\e(t_j,\cdot)).
\end{equation}

\begin{proposition}
\label{W_canc_3}
Let $(t_j,x_j)$ be a cancellation point. Then
\[
\sum_{s \in \W(t_j)} |\sigma(t_j, s) - \sigma(t_{j-1}, s)||s|  \leq \lVert f'' \lVert_{L^\infty}  \TV(u(0,\cdot)) \C(t_j,x_j).
\]
\end{proposition}
\begin{proof}
Let $u_L, u_M$ be respectively the left and the right state of the left wavefront involved in the collision at point $(t_j,x_j)$ and let $u_M, u_R$ be respectively the left and the right state of the right wavefront involved in the collision at point $(t_j,x_j)$, so that $u_L = {\displaystyle \lim_{x \nearrow x_j}} u_\e(t_j, x)$ and $u_R = u_\e(t_j, x_j)$. Without loss of generality, assume $u_L < u_R < u_M$.

Then we have
\begin{equation}
\label{W_canc_1}
\begin{split}
& \sum_{s \in \W(t_j)} |\sigma(t_j, s) - \sigma(t_{j-1}, s)||s|  \\ 
& \quad \ \ = \sum_{s \in \W(t_j)}  \bigg[ \bigg(\frac{d}{du} \conv_{[u_L, u_R]} f_\e \bigg) \Big( (\hat u(s) -\e, \hat u(s))\Big) - \bigg(\frac{d}{du} \conv_{[u_L, u_M]} f_\e \bigg)\Big( (\hat u(s)-\e, \hat u(s))\Big) \bigg]|s| \\
& \overset{\text{(Prop. \ref{differenza_vel})}}{\leq} 
\sum_{s \in \W(t_j)}  \bigg[ \bigg(\frac{d}{du} \conv_{[u_L, u_R]} f_\e \bigg)(u_R-) - \bigg(\frac{d}{du} \conv_{[u_L, u_M]} f_\e \bigg)(u_R-) \bigg]|s| \\
& \quad \ \ =  \bigg[ \bigg(\frac{d}{du} \conv_{[u_L, u_R]} f_\e \bigg)(u_R-) - \bigg(\frac{d}{du} \conv_{[u_L, u_M]} f_\e \bigg)(u_R-) \bigg]  \sum_{s \in \W(t_j)} |s| \\
& \quad \ \ \leq \bigg[ \bigg(\frac{d}{du} \conv_{[u_L, u_R]} f_\e \bigg)(u_R-) - \bigg(\frac{d}{du} \conv_{[u_L, u_M]} f_\e \bigg)(u_R-) \bigg]  \TV(u_\e(0,\cdot)) .
\end{split}
\end{equation}
Now observe that, by Proposition \ref{diff_vel_proporzionale_canc}, 
\begin{equation}
\label{W_canc_2}
\begin{split}
\bigg(\frac{d}{du} \conv_{[u_L, u_R]} f_\e \bigg)(u_R-) - \bigg(\frac{d}{du} \conv_{[u_L, u_M]} f_\e \bigg)(u_R-) 
\leq&~ \lVert f'' \lVert_{L^\infty} (u_M - u_R) \\
\leq&~ \lVert f'' \lVert_{L^\infty} \mathcal{C}(t_j,x_j).
\end{split}
\end{equation}

Hence, from \eqref{W_canc_1} and \eqref{W_canc_2}, we obtain
\begin{equation*}
\sum_{s \in \W(t_j)} |\sigma(t_j, s) - \sigma(t_{j-1}, s)||s|  \leq   
\lVert f'' \lVert_{L^\infty} \TV(u_\e(0,\cdot)) \mathcal{C}(t_j,x_j).
\end{equation*}
Together with \eqref{bd_su_dato_iniziale}, this concludes the proof.
\end{proof}

\begin{corollary}
\label{W_canc_4}
It holds
\[
\sum_{\substack{j \emph{ such that} \\ (t_j,x_j) \emph{ is a} \\ \emph{cancellation point}}} \sum_{s \in \W(t_j)} |\sigma(t_j, s) - \sigma(t_{j-1}, s)||s| \leq \lVert f'' \lVert_{L^\infty} \TV(u(0,\cdot))^2.
\]
\end{corollary}
\begin{proof}
From \eqref{bd_su_dato_iniziale}, \eqref{W_canc_0} and Proposition \ref{W_canc_3} we obtain
\[
\begin{split}
\sum_{\substack{(t_j,x_j) \\ \text{cancellation}}} & \sum_{s \in \W(t_j)} |\sigma(t_j, s) - \sigma(t_{j-1}, s)||s|\\
& \leq  \lVert f'' \lVert_{L^\infty} \TV(u(0,\cdot)) \sum_{j=1}^J \Big[ \TV(u_\e(t_{j-1}, \cdot)) - \TV(u_\e(t_j, \cdot)) \Big] \\
& \leq  \lVert f'' \lVert_{L^\infty} \TV(u(0,\cdot)) \Big[ \TV(u_\e(0,\cdot)) - \TV(u_\e(t_J, \cdot)) \Big] \\
& \leq  \lVert f'' \lVert_{L^\infty} \TV(u(0,\cdot))^2,
\end{split}
\]
thus concluding the proof of the corollary.
\end{proof}

From now on, our aim is to prove that 
\begin{equation*}
\sum_{\substack{(t_j,x_j) \\ \text{interaction}}} \sum_{s \in \W(t_j)} |\sigma(t_j, s) - \sigma(t_{j-1}, s)||s| \leq \const \lVert f'' \lVert_{L^\infty} \TV(u(0,\cdot))^2.
\end{equation*}

As outlined in Section \ref{Sss_sketch_proof}, the idea is the following: we  define a positive valued functional $\fQ = \fQ(t)$, $t \geq 0$, such that $\fQ$ is piecewise constant in time, right continuous, with jumps possibly located at times $t_j, j=1,\dots,J$ and such that 
\begin{equation}
\label{boundQzero}
\fQ(0) \leq \lVert f'' \lVert_{L^\infty} \TV(u(0,\cdot))^2.
\end{equation}
Such a functional will have two properties:
\begin{enumerate}
\item for each $j$ such that $(t_j,x_j)$ is an interaction point, $\fQ$ is decreasing at time $t_j$ and its decrease bounds the quantity we want to estimate at time $t_j$ as follows:
\begin{equation}
\label{W_decrease}
\sum_{s \in \W(t_j)} |\sigma(t_j, s) - \sigma(t_{j-1}, s)||s| \leq 2 \Big[ \fQ(t_{j-1}) - \fQ(t_j) \Big];
\end{equation}
this is proved in Theorem \ref{W_decreasing};
\item for each $j$ such that $(t_j, x_j)$ is a cancellation point, $\fQ$ can increase at most by 
\begin{equation}
\label{W_increase}
\fQ(t_j) - \fQ(t_{j-1}) \leq \log(2) \lVert f'' \lVert_{L^\infty} \TV(u(0,\cdot)) \big[ \mathcal{C}(t_j,x_j) \big];
\end{equation}
this is proved in Theorem \ref{W_increasing}.
\end{enumerate}
Using the two estimates above, we obtain the following proposition, which completes the proof of Theorem \ref{W_main_thm}.

\begin{proposition}
\label{W_thm_interaction}
It holds
\begin{equation*}
\sum_{\substack{(t_j,x_j) \\ \emph{interaction}}} \sum_{s \in \W(t_j)} |\sigma(t_j, s) - \sigma(t_{j-1}, s)||s| \leq 2 ( 1 + \log(2)) \lVert f'' \lVert_{L^\infty} \TV(u(0,\cdot))^2.
\end{equation*}
\end{proposition}

\begin{proof}
By direct computation,
\begin{multline*}
\sum_{\substack{(t_j,x_j) \\ \text{interaction}}} \sum_{s \in \W(t_j)} |\sigma(t_j, s) - \sigma(t_{j-1}, s)||s| \\
\begin{aligned}
\text{(by \eqref{W_decrease})}  \leq &~ 2 \sum_{\substack{(t_j,x_j) \\ \text{interaction}}}  \big[ \fQ(t_{j-1}) - \fQ(t_j) \big] \\
\leq &~ 2 \Bigg[ \sum_{\substack{(t_j,x_j) \\ \text{interaction}}}  \big[ \fQ(t_{j-1}) - \fQ(t_j) \big] 
+ \sum_{\substack{(t_j,x_j) \\ \text{cancellation}}}  \big[ \fQ(t_{j-1}) - \fQ(t_j) \big] \\
& \qquad \qquad - \sum_{\substack{(t_j,x_j) \\ \text{cancellation}}}  \big[ \fQ(t_{j-1}) - \fQ(t_j) \big]  \Bigg] \\
= &~ 2 \Bigg[ \sum_{j=1}^J \big[ \fQ(t_{j-1}) - \fQ(t_j) \big] 
+ \sum_{\substack{(t_j,x_j) \\ \text{cancellation}}}  \big[ \fQ(t_j) - \fQ(t_{j-1}) \big]   \Bigg] \\
\text{(by \eqref{W_increase})}  \leq &~  2 \Bigg[ \sum_{j=1}^J \big[ \fQ(t_{j-1}) - \fQ(t_j) \big] 
+ \log(2) \sum_{\substack{(t_j,x_j) \\ \text{cancellation}}}  \lVert f'' \lVert_{L^\infty} \TV(u(0,\cdot)) \mathcal{C}(t_j,x_j) \Bigg] \\
\text{(by \eqref{W_canc_0}, \eqref{boundQzero})} \leq &~ 
2 \Big[ \fQ(0) + \log(2) \lVert f'' \lVert_{L^\infty} \TV(u(0,\cdot))^2 \Big] \\
\leq&~ 2 (1 + \log(2)) \lVert f'' \lVert_{L^\infty} \TV(u(0,\cdot))^2. 
\end{aligned}
\end{multline*}
\end{proof}

In the remaining part of this section we prove estimates \eqref{W_decrease} and \eqref{W_increase}.

\subsection{Analysis of waves collisions for wavefront tracking}
\label{W_waves_collision}

In this section we define the notion of pairs of waves which \emph{have never interacted before a fixed time $t$} and pairs of waves which \emph{have already interacted} and, for any pair of waves which have already interacted, we associate an artificial speed difference, which is some sense summarize their past common history.

\begin{definition}
\label{interagite_non_interagite}
Let $\bar t$ be a fixed time and let $s,s' \in \W(\bar t)$. We say that \emph{$s,s'$ interact at time $\bar t$} if $\mathtt x(\bar t, s) = \mathtt x(\bar t, s')$.  

We also say that \emph{they have already interacted at time $\bar t$} if there is $t \leq \bar t$ such that $s,s'$ interact at time $t$. Moreover we say that \emph{they have not yet interacted at time $\bar t$} if for any $t \leq \bar t$, they do not interact at time $t$. 
\end{definition}

\begin{lemma}
\label{W_interagite_stesso_segno}
Assume that the waves $s, s'$ interact at time $\bar t$. Then they have the same sign.
\end{lemma}
\begin{proof}
Easy consequence of definition of enumeration of waves and the fact that $\mathcal S(s)$ is independent of $t$. 
\end{proof}

\begin{lemma}
\label{W_quelle_in_mezzo_hanno_int}
Let $\bar t$ be a fixed time, $s,s' \in \W(\bar t)$, $s < s'$. Assume that $s, s'$ have already interacted at time $\bar t$. If $p, p' \in \W(\bar t)$ and $s \leq p \leq p' \leq s'$, then $p, p'$ have already interacted at time $\bar t$.
\end{lemma}
\begin{proof}
Let $t$ be the time such that $s,s'$ interact at time $t$. Clearly $s,s',p,p' \in \W(t) \supseteq \W(\bar t)$. Since for $t$ fixed, $\mathtt x$ is increasing on $\W(t)$, it holds $\mathtt x(t,s) = \mathtt x(t,p) = \mathtt x(t, p') = \mathtt x(t, s')$. 
\end{proof}

Let $\bar s \in \W(\bar t)$. By Lemmas \ref{W_interagite_stesso_segno} and \ref{W_quelle_in_mezzo_hanno_int}, the set
\[
\mathcal{I}(\bar t, \bar s) := \Big\{s \in \W(\bar t) \ \Big| \ s \text{ has already interacted with } \bar s \text{ at time } \bar t\Big\}
\]
is an homogeneous interval of waves.
Moreover set
\begin{equation*}
L(\bar t, \bar s)  := \min \mathcal{I}(\bar t,\bar s), \ \ \  R(\bar t, \bar s)  := \max \mathcal{I}(\bar t,\bar s).  
\end{equation*}

Let $s, s'$ be two waves. Assume that $s < s'$ and that they have already interacted at a fixed time $\bar t$. Consider now the set 
\[
\mathcal{I}(\bar t, s, s') : = \mathcal{I}(\bar t, s) \cap \mathcal{I}(\bar t, s').
\]
This is clearly an interval of waves. Observe that $\min \mathcal{I}(\bar t,s,s') = L(\bar t,s')$ and $\max \mathcal{I}(\bar t, s,s') = R(\bar t,s)$. A fairly easy argument based on Lemma \ref{W_quelle_in_mezzo_hanno_int} implies that $\mathcal I(\bar t,s,s')$ is made of the waves $p$ which have interacted with both $s$ and $s'$.


\begin{definition}
\label{W_waves_divided}
Let $s,s' \in \W(\bar t)$ be two waves which have already interacted at time $\bar t$. We say that \emph{$s,s'$ are divided in the real solution at time $\bar t$} if 
\[
(\mathtt x(\bar t, s), \sigma(\bar t, s)) \neq (\mathtt x(\bar t, s'), \sigma(\bar t, s')),
\]
i.e. if at time $\bar t$ they have either different position, or the same position but different speed.

\noindent If they are not divided in the real solution, we say that \emph{they are joined in the real solution}.
\end{definition}

\begin{remark}
\label{rem_divise_solo_in_cancellazioni}
It $\bar t \neq t_j$ for each $j$, then two waves are divided in the real solution if and only if they have different position. The requirement to have different speed is needed at collision times, more precisely at cancellations.
\end{remark}

\begin{proposition}
\label{W_unite_realta}
Let $\bar t$ be a fixed time. Let $s,s' \in \W(\bar t)$. If $s,s'$ are not divided in the real solution at time $\bar t$, then the Riemann problem $\mathcal{I}(\bar t,s,s')$ does not divide them.
\end{proposition}

For the definition of the Riemann problem $\mathcal I(\bar t,s,s')$ see Definition \ref{W_artificial_speed} and Remark \ref{W_artificial_speed_remark}.

\begin{proof}
Let $\bar x = \mathtt x(\bar t, s) = \mathtt x(\bar t,s')$. Clearly $\W(\bar t, \bar x) \subseteq \mathcal{I}(\bar t,s,s')$. Observe that $\W(\bar t, \bar x)$ is an interval of waves and that by definition the real speed of the waves $\sigma(\bar t,s) = \sigma(\bar t,s')$ is the speed given $s,s'$ by the Riemann problem $\W(\bar t,\bar x)$. The conclusion is then a consequence of Corollary \ref{stesso_shock}.
\end{proof}

The remaining part of this section is devoted to prove the following proposition, which is in some sense the converse of the previous one and is a key tool in order estimate the increase and the decrease of the functional $\fQ$.

\begin{proposition}
\label{W_divise_tocca}
Let $\bar t$ be a fixed time. Let $s,s'$ be two waves which have already interacted at time $\bar t$. Assume that $s,s'$ are divided in the real solution. Let $p,p' \in \mathcal{I}(\bar t, s,s')$.
If $p,p'$ are divided in the real solution at time $\bar t$, then the Riemann problem $\mathcal{I}(\bar t, s,s')$ divides them.
\end{proposition}

\begin{proof}
Fix two waves $s < s'$. It is sufficient to prove the proposition only for times $t_j, j=0,\dots, J$. We proceed by induction on $j$. For $j=0$ the proof is obvious. Let us assume the lemma is true for $j-1$ and let us prove it for $j$. Suppose $s,s'$ to be divided in the real solution at time $t_j$. We can also assume w.l.o.g. that $s,s'$ are both positive.

When $(t_j,x_j)$ is an interaction the analysis is quite simple, while the cancellation case requires more effort.

\smallskip

\noindent \textbf{$(t_j,x_j)$ interaction.} Let us distinguish two cases. 

If $\mathcal{I}(t_j, s,s') \neq \mathcal{I}(t_{j-1},s,s')$, then waves $s,s'$ must be involved in the interaction, i.e. $\mathtt x(t_j,s) = \mathtt x(t_j,s') = x_j$. Since $(t_j,x_j)$ is an interaction point, $\sigma(t_j,s) = \sigma(t_j,s')$, and so $s,s'$ are not divided at time $t_j$, hence the statement of the proposition can not apply to $s,s'$.

If $\mathcal{I}(t_j, s,s') = \mathcal{I}(t_{j-1},s,s')$, take $p,p' \in \mathcal{I}(t_j, s,s')$, such that $p,p'$ are divided in the real solution. Since an interaction does not divide waves which were joined before the interaction, $p,p'$ were already divided at time $t_{j-1}$ and so by inductive assumption we have done.
\smallskip

\noindent \textbf{$(t_j,x_j)$ cancellation.}  Assume that $s,s'$ are divided in the real solution after the cancellation at time $t_j$. Moreover, w.l.o.g., assume that in $(t_j,x_j)$ two wavefronts collide, the one coming from the left is positive, the one coming from the right is negative; assume also that waves in $\W(t_j,x_j)$ are positive, and that $\mathtt x(t_j,s) \leq  \mathtt x(t_j,s') \leq x_j$ (the proof in the case $x_j < \mathtt x(t_j,s) \leq \mathtt x(t_j,s')$ is similar, but easier).

Set $\min \W(t_j,x_j) := r_1$, $\max \W(t_j,x_j) := r_2$. If $R(t_j, s) < r_1$, then $s,s'$ were already divided before the collision (i.e. at time $t_{j-1}$), $L(t_j, s') = L(t_{j-1},s')$ and $R(t_j,s) = R(t_{j-1},s)$ and so by inductive assumption we conclude. Hence assume $R(t_j,s) \in [r_1,r_2]$. If $L(t_j,s') = r_1$, then $\mathcal{I}(t_j,s,s') \subseteq \W(t_j,x_j)$ and so by Corollary \ref{stesso_shock}, we can again conclude; hence let us assume $L(t_j,s') < r_1$. 

Finally observe that we can assume $[L(t_j, s'), r_2] \cap \W(t_j) = [L(t_j,s'),r_2]$, i.e. no wave in  $[L(t_j,s'),r_2]$ has been canceled up to time $t_j$ (the general case can be treated in a similar way). See Figure \ref{fig:figura03}.

\begin{figure}
  \begin{center}
    \includegraphics[height=7.5cm,width=12cm]{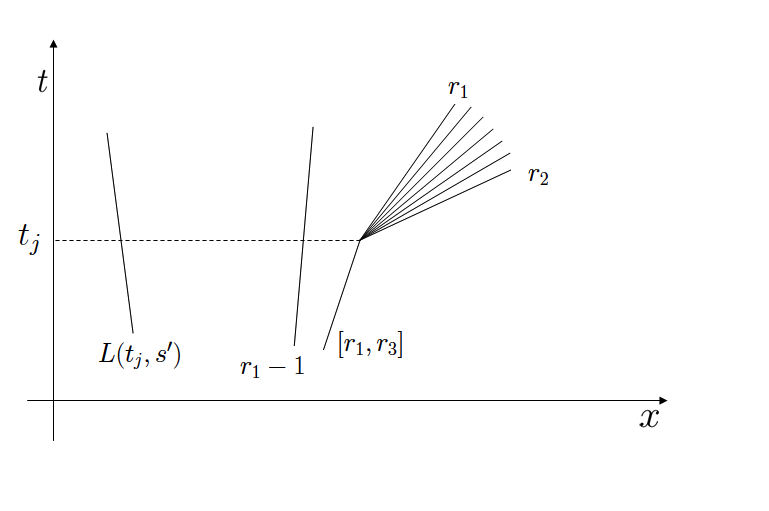}
    \caption{Graphical description of our notations.}
    \label{fig:figura03}
    \end{center}
\end{figure}

Let us observe that 
\begin{equation*}
L(t_j,s') = L(t_{j-1},s'), \  \ \ \  R(t_j,s)  \leq R(t_{j-1},s).
\end{equation*}
Now set 
\begin{equation*}
\tilde s :=
\begin{cases}
s & \text{if } s,s' \text{ were already divided at time } t_{j-1}, \\
L(t_j,s') & \text{otherwise}.
\end{cases}
\end{equation*}

We need now the following three claims.

\begin{claim}
\label{W_s_tilde_sep}
Waves $\tilde s, s'$ are divided in the real solution at time $t_{j-1}$.
\end{claim}
\begin{proof}[Proof of Claim \ref{W_s_tilde_sep}]
If $s,s'$ are divided in the real solution at time $t_{j-1}$, then by definition $\tilde s = s$ and then we have done. If $s,s'$ are not divided, this means that $\mathtt x(t_j,s) = \mathtt x(t_j,s') = x_j$, while by one of our assumption $\tilde s = L(t_j,s') < r_1 = \min \W(t_j,x_j)$; hence $\tilde s, s'$ have different position at time $t_j$. 
\end{proof}

\begin{claim}
\label{W_riemann_prima}
Waves $r_1-1$ and $r_1$ are divided by the Riemann problem $\mathcal{I}(t_{j-1},\tilde s,s')$. 
\end{claim}

\begin{proof}[Proof of Claim \ref{W_riemann_prima}]
By Claim \ref{W_s_tilde_sep}, $\tilde s, s'$ are divided in the real solution at time $t_{j-1}$. Moreover, by definition of $r_1$, also $r_1-1$ and $r_1$ are divided in the real solution at time $t_j$; since $r_1 = \min \W(t_j,x_j)$, $r_1-1$ and $r_1$ are divided also at time $t_{j-1}$. Hence by inductive assumption, the Riemann problem $\mathcal{I}(t_{j-1},\tilde s,s')$ divides $r_1-1$ and $r_1$.
\end{proof}

\begin{claim}
\label{W_riemann_dopo}
Waves $r_1-1$ and $r_1$ are divided by the Riemann problem $\mathcal{I}(t_j,s,s')$. 
\end{claim}

\begin{proof}[Proof of Claim \ref{W_riemann_dopo}]
%
Assume first that $s,s'$ are divided in the real solution at time $t_{j-1}$. In this case, by Claim \ref{W_riemann_prima}, $r_1 -1$ and $r_1$ are divided by the Riemann problem $\mathcal{I}(t_{j-1},s,s')$; hence by Corollary \ref{stesso_shock} they are divided also by the Riemann problem $\mathcal{I}(t_j,s,s')$ since $\mathcal{I}(t_j,s,s') \subseteq \mathcal{I}(t_{j-1},s,s')$.

Now assume $s,s'$ are joined in the real solution at time $t_{j-1}$. In this case $\mathtt x(t_j,s) = \mathtt x(t_j,s') = x_j$. By Claim \ref{W_riemann_prima}, the Riemann problem $\mathcal{I}(t_{j-1},\tilde s,s')$ divides $r_1-1$ and $r_1$. Moreover, since $s,s'$ are divided in the real solution at time $t_j$, the Riemann problem $\W(t_j,x_j) = [r_1,r_2]$ divides them. Hence, observing that $\min \mathcal{I}(t_{j-1},\tilde s,s') = L(t_{j-1},s') = L(t_j,s')$, by Proposition \ref{incastro}, the Riemann problem $[L(t_j,s'), r_2]$ divides $r_1-1$ and $r_1$. One concludes the proof just observing that $r_2 = R(t_j,s)$ and then $[L(t_j,s'), r_2] = \mathcal{I}(t_j,s,s')$.
\end{proof}

Now we are able to conclude the proof of our proposition. Take $p,p' \in \mathcal{I}(t_j,s,s')$ and assume that $p < p'$ are divided in the real solution at time $t_j$. 
\begin{enumerate}
\item If $p \in [L(t_j,s'),r_1-1]$ and $p' \in [r_1,R(t_j,s)]$, then by Claim \ref{W_riemann_dopo} $p,p'$ are divided by the Riemann problem $\mathcal{I}(t_j,s,s')$.
\item If $p,p' \in [L(t_j,s'),r_1-1]$, then they were already divided at time $t_{j-1}$ Hence, by inductive assumption and Claim \ref{W_s_tilde_sep}, the Riemann problem $\mathcal{I}(t_{j-1},\tilde s,s')$ divides $p,p'$. Thus by Claim \ref{W_riemann_prima} and Proposition \ref{tocca} also the Riemann problem $[L(t_j,s'), r_1-1]$ divides $p,p'$ and then by Claim \ref{W_riemann_dopo} and Proposition \ref{tocca}, the Riemann problem $\mathcal{I}(t_j,s,s')$ divides them.
\item Finally assume that $p,p' \in [r_1,R(t_j,s)]$. Since $p,p'$ are divided in the real solution at time $t_j$, one has that the Riemann problem $[r_1,r_2]$ divides them, and also the Riemann problem $[r_1,R(t_j,s)] \subseteq [r_1,r_2]$  divides them. Hence, by Claim \ref{W_riemann_dopo} and Proposition \ref{tocca}, the Riemann problem $\mathcal{I}(t_j,s,s')$ divides $p,p'$. 
\end{enumerate}
This concludes the proof of Proposition \ref{W_divise_tocca}.
\end{proof}

\subsection{\texorpdfstring{The functional $\fQ$ for wavefront tracking approximation}{The functional for wavefront tracking approximation}}
\label{W_functional_Q}

We can finally define the functional $\fQ$ and prove that it satisfies inequalities \eqref{W_decrease} and \eqref{W_increase}.

\subsubsection{Definition of $\fQ$}

First of all for any pair of wave $(s,s')$, $s<s'$, define \emph{the weight $\mathfrak q(t,s,s')$ of the pair of waves $s$, $s'$ at time $t$} in the following way:
\begin{equation}
\label{W_mathfrak_q}
\mathfrak{q}(t, s, s') :=
\begin{cases}
\dfrac{|\sigma(\mathcal{I}(t,s,s'), s') - \sigma(\mathcal{I}(t,s,s'), s)|}{|\hat u(s') - (\hat u(s)-\mathcal{S}(s) \e)|} & 
s,s' \text{already interacted at time } \bar t,
\\
\lVert f'' \lVert_{L^\infty} & \text{otherwise.}
\end{cases}
\end{equation}
Recall that $\sigma(\mathcal{I}(t,s,s'), s)$ (resp. $\sigma(\mathcal{I}(t,s,s'), s')$ ) is the speed given to the wave $s$ (resp. $s'$) by the Riemann problem $\mathcal{I}(t,s,s')$.

As an easy consequence of Proposition \ref{convex_fundamental_thm_affine}, we obtain that $\mathfrak{q}$ takes values in $[0,\lVert f'' \lVert_{L^\infty}]$. 

\begin{remark}
\label{W_rem_unite_realta}
If $s,s'$ are joined in the real solution, then by Proposition \ref{W_unite_realta} $\mathfrak{q}(t, s, s') = 0$.
\end{remark}

Finally set
\[
\mathfrak{Q}(t) :=  \sum_{\substack{s,s' \in \W(t) \\ s < s'}} \mathfrak{q}(t, s, s') |s||s'|.
\]
(Recall that $|s| = |s'| = \e$ is the strength of the waves $s,s'$ respectively.)

It is immediate to see that $\fQ$ is positive, piecewise constant, right continuous, with jumps possibly located at times $t_j, j=1,\dots,J$, and $\fQ(0) \leq \lVert f'' \lVert_{L^\infty} \TV(u_\e(0,\cdot))^2 \leq \lVert f'' \lVert_{L^\infty} \TV(u(0,\cdot))^2$. In the next two sections we prove that it also satisfies inequality \eqref{W_decrease} and \eqref{W_increase}. This completes the proof of Proposition \ref{W_thm_interaction}.

\subsubsection{Decreasing part of $\fQ$}

This section is devoted to prove inequality \eqref{W_decrease}. 

\begin{theorem}
\label{W_decreasing}
For any interaction point $(t_j,x_j)$, it holds
\begin{equation*}
\sum_{s \in \W(t_j)} |\sigma(t_j, s) - \sigma(t_{j-1}, s)||s| \leq 2 \big[ \fQ(t_{j-1}) - \fQ(t_j) \big].
\end{equation*}
\end{theorem}

By direct inspection of the proof one can verify that the constant $2$ is sharp.

\begin{proof}
Assume w.l.o.g. that all the waves in $\W(t_j,x_j)$ are positive. We partition $\W(t_j,x_j)$ through the equivalence relation
\[
s \sim s' \quad \text{if and only if} \quad \mathtt x(t_{j-1},s) = \mathtt x(t_{j-1},s'). 
\]
By our assumption $\W(t_j,x_j)$ is decomposed into two disjoint intervals of waves $\W(t_j,x_j) = \mathcal{L} \cup \mathcal{R}$ such that for each $s \in \mathcal{L}$ and $s' \in \mathcal{R}$, it holds $s <s'$. 

\smallskip
{\it Step 1.} First of all observe that by formula \eqref{W_mathfrak_q} and by Proposition \ref{W_unite_realta}, if $s<s'$, $(s,s') \in \W(t_j) \times \W(t_j)$, but $(s,s') \notin \mathcal{L} \times \mathcal{R}$, then $\mathfrak{q}(t_{j-1},s,s') = \mathfrak{q}(t_j,s,s')$. Indeed, if at least one between $s,s'$ does not belong to $\mathcal{L} \cup \mathcal{R}$, then $\mathcal{I}(t_j, s,s') = \mathcal{I}(t_j,s,s')$; on the other side, if $s,s' \in \mathcal{L}$ (or $s,s' \in \mathcal{R}$), then $s,s'$ are joined both before and after the interaction and for this reason, by Remark \ref{W_rem_unite_realta}, $\mathfrak{q}(t_{j-1},s,s') = \mathfrak{q}(t_j,s,s') = 0$. 
Now observe that , if $(s,s') \in \mathcal{L} \times \mathcal{R}$, then, by Remark \ref{W_rem_unite_realta}, $\mathfrak{q}(t_j, s,s') = 0$. Hence it is sufficient to prove that  
\begin{equation*}
\sum_{s \in \W(t_j)} |\sigma(t_j, s) - \sigma(t_{j-1}, s)||s| \leq 2 \sum_{(s,s') \in \mathcal{L} \times \mathcal{R}} \mathfrak{q}(t_{j-1}, s,s').
\end{equation*}

\smallskip
{\it Step 2.} For any positive interval of waves $\mathcal{I}$, define \emph{the strength of the interval $\mathcal{I}$} as 
\[
|\mathcal{I}| := \e \card (\mathcal{I}) = \sum_{s \in \mathcal{I}} |s|,
\]
and \emph{the mean speed of waves in $\mathcal{I}$} as 
\begin{equation*}
\sigma_m(\mathcal{I}) := 
\begin{cases}
\frac{f(\hat u(\max \mathcal{I})) - f(\hat u(\min \mathcal{I}) - \e)}{\hat u(\max \mathcal{I}) - (\hat u(\min \mathcal{I}) - \e)}
& \text{if } \mathcal{I} \neq \emptyset, \\
2 \lVert f'' \lVert_{L^\infty} & \text{if } \mathcal{I} = \emptyset.
\end{cases}
\end{equation*}

Now observe that 
\begin{equation}
\label{W_velocita}
\sigma_m(\mathcal{L} \cup \mathcal{R}) = \sigma_m(\W(t_j,x_j)) = \frac{\sigma_m(\mathcal{L})|\mathcal{L}| + \sigma_m(\mathcal{R})|\mathcal{R}|}{|\mathcal{L}|+|\mathcal{R}|}. 
\end{equation}
Hence
\begin{equation}
\label{W_zzero}
\begin{split}
\sum_{s \in \W(t_j)} |\sigma(t_j, s) - \sigma(t_{j-1}, s)||s| =&~ 
\sum_{s \in \mathcal{L}} |\sigma(t_j, s) - \sigma(t_{j-1}, s)||s|
 + \sum_{s \in \mathcal{R}} |\sigma(t_j, s) - \sigma(t_{j-1}, s)||s| \\
=&~ \Big(\sum_{s \in \mathcal{L}}|s| \Big) |\sigma_m(\mathcal{L} \cup \mathcal{R}) - \sigma_m(\mathcal{L})| +
\Big(\sum_{s \in \mathcal{L}}|s| \Big) |\sigma_m(\mathcal{L} \cup \mathcal{R}) - \sigma_m(\mathcal{R})| \\
=&~ |\sigma_m(\mathcal{L} \cup \mathcal{R}) - \sigma_m(\mathcal{L})||\mathcal{L}| +
|\sigma_m(\mathcal{L} \cup \mathcal{R}) - \sigma_m(\mathcal{R})||\mathcal{R}| \\
\text{(by \eqref{W_velocita})} \ =&~ 2 \  \frac{|\sigma_m(\mathcal{L}) - \sigma_m(\mathcal{R})|}{|\mathcal{L}| + |\mathcal{R}|} |\mathcal{L}||\mathcal{R}|.
\end{split}
\end{equation}

\smallskip
{\it Step 3.} Set $\ell := \max \mathcal{L}$, $r := \min \mathcal{R}$ and define (see Fig. \ref{fig:figura14})
\begin{align*}
\mathcal{L}_1 & := \Big\{s \in \mathcal{L} \ | \ s < L(t_{j-1},r)\Big\},  &   
\mathcal{R}_1 & := \Big\{s \in \mathcal{R} \ | \ s \leq R(t_{j-1},\ell)\Big\}, \\  
\mathcal{L}_2 & := \Big\{s \in \mathcal{L} \ | \ s \geq L(t_{j-1},r)\Big\},  &   
\mathcal{R}_2 & := \Big\{s \in \mathcal{R} \ | \ s > R(t_{j-1},\ell)\Big\}.
\end{align*}

\begin{figure}
  \begin{center}
    \includegraphics[height=7.5cm,width=12cm]{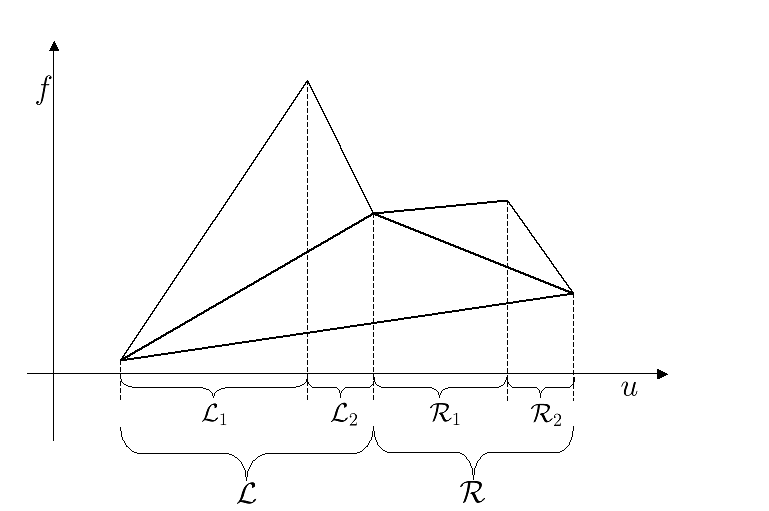}
    \caption{Families of waves which interact at point $(t_j, x_j)$.}
    \label{fig:figura14}
    \end{center}
\end{figure}

\noindent Thus
\begin{multline*}
|\sigma_m(\mathcal{L}) - \sigma_m(\mathcal{R})||\mathcal{L}||\mathcal{R}| \\
\begin{aligned}
& \leq 
|\sigma_m(\mathcal{L}_1) - \sigma_m(\mathcal{R})| |\mathcal{L}_1||\mathcal{R}|
+ |\sigma_m(\mathcal{L}_2) - \sigma_m(\mathcal{R})| |\mathcal{L}_2||\mathcal{R}| \\
& \leq 
|\sigma_m(\mathcal{L}_1) - \sigma_m(\mathcal{R})| |\mathcal{L}_1||\mathcal{R}|
+ |\sigma_m(\mathcal{L}_2) - \sigma_m(\mathcal{R}_1)| |\mathcal{L}_2||\mathcal{R}_1| 
+ |\sigma_m(\mathcal{L}_2) - \sigma_m(\mathcal{R}_2)| |\mathcal{L}_2||\mathcal{R}_2|.
\end{aligned}
\end{multline*}
Dividing by $|\mathcal{L}| + |\mathcal{R}|$ we get
\begin{multline}
\label{W_zero}
\frac{|\sigma_m(\mathcal{L}) - \sigma_m(\mathcal{R})|}{|\mathcal{L}| + |\mathcal{R}|} |\mathcal{L}||\mathcal{R}| \\
\begin{aligned}
& \leq 
\frac{|\sigma_m(\mathcal{L}_1) - \sigma_m(\mathcal{R})|}{|\mathcal{L}| + |\mathcal{R}|} |\mathcal{L}_1||\mathcal{R}|
+ \frac{|\sigma_m(\mathcal{L}_2) - \sigma_m(\mathcal{R}_1)|}{|\mathcal{L}| + |\mathcal{R}|} |\mathcal{L}_2||\mathcal{R}_1| 
+ \frac{|\sigma_m(\mathcal{L}_2) - \sigma_m(\mathcal{R}_2)|}{|\mathcal{L}| + |\mathcal{R}|} |\mathcal{L}_2||\mathcal{R}_2| \\
& \leq 
\Vert f'' \Vert_{L^\infty} |\mathcal{L}_1||\mathcal{R}|
+ \frac{|\sigma_m(\mathcal{L}_2) - \sigma_m(\mathcal{R}_1)|}{|\mathcal{L}| + |\mathcal{R}|} |\mathcal{L}_2||\mathcal{R}_1| 
+ \lVert f'' \Vert_{L^\infty} |\mathcal{L}_2||\mathcal{R}_2|, \\
\end{aligned}
\end{multline}
where the last inequality is a consequence of Lagrange's Theorem. 

\smallskip
{\it Step 4.} Let us now concentrate our attention on the second term of the last summation. Observe that waves $\ell, r$ are divided in the real solution at time $t_{j-1}$; hence, by Proposition \ref{W_divise_tocca}, they are divided by the Riemann problem $\mathcal{I}(t_{j-1},\ell,r) = \mathcal{L}_2 \cup \mathcal{R}_1$. Hence it is not difficult to see (it is the cubic estimate when the speeds are monotone) that we can write
\begin{equation}
\label{W_uno}
|\sigma_m(\mathcal{L}_2) - \sigma_m(\mathcal{R}_1)||\mathcal{L}_2||\mathcal{R}_1| = \sum_{(s,s') \in \mathcal{L}_2 \times \mathcal{R}_1} \Big( \sigma(\mathcal{I}(t_{j-1},\ell,r), s') - \sigma(\mathcal{I}(t_{j-1},\ell,r), s) \Big) |s| |s'| .
\end{equation}

Let us now observe that by definition of $\mathcal{L}_2, \mathcal{R}_1$, for any $s \in \mathcal{L}_2$ and $s' \in \mathcal{R}_1$, if $s,s'$ have already interacted at time $t_{j-1}$, then $\mathcal{I}(t_{j-1},s,s') \subseteq \mathcal{I}(t_{j-1},\ell,r)$. Together with Proposition \ref{vel_aumenta} and with the fact that $\ell, r$ are divided by the Riemann problem $\mathcal{I}(t_{j-1},\ell,r)$, this yields 
\begin{equation*}
\sigma(\mathcal{I}(t_{j-1},\ell,r), s') - \sigma(\mathcal{I}(t_{j-1},\ell,r), s) \leq 
\sigma(\mathcal{I}(t_{j-1},s,s'), s') - \sigma(\mathcal{I}(t_{j-1},s,s'), s), 
\end{equation*}
and thus
\begin{equation}
\label{W_due}
\begin{split}
\frac{\sigma(\mathcal{I}(t_{j-1},\ell,r), s') - \sigma(\mathcal{I}(t_{j-1},\ell,r), s) }{|\mathcal{L}|+|\mathcal{R}|} \leq&~
\frac{\sigma(\mathcal{I}(t_{j-1},s,s'), s') - \sigma(\mathcal{I}(t_{j-1},s,s'), s)}{|\mathcal{L}|+|\mathcal{R}|} \\ 
\leq&~ \frac{\sigma(\mathcal{I}(t_{j-1},s,s'), s') - \sigma(\mathcal{I}(t_{j-1},s,s'), s)}{\hat u(s') - (\hat u(s) -\e)} \\
=&~ \mathfrak{q}(t_{j-1},s,s').
\end{split}
\end{equation}

Instead, if $s,s'$ have not yet interacted at time $t_{j-1}$, using Lagrange's Theorem, we get
\begin{equation}
\label{W_tre}
\frac{\sigma(\mathcal{I}(t_{j-1},\ell,r), s') - \sigma(\mathcal{I}(t_{j-1},\ell,r), s) }{|\mathcal{L}|+|\mathcal{R}|} \leq \lVert f'' \lVert_{L^\infty} = \mathfrak{q}(t_{j-1},s,s').  
\end{equation}
Thus, by \eqref{W_uno}, \eqref{W_due}, \eqref{W_tre}, 
\begin{equation*}
\frac{|\sigma_m(\mathcal{L}_2) - \sigma_m(\mathcal{R}_1)|}{|\mathcal{L}| + |\mathcal{R}|} |\mathcal{L}_2||\mathcal{R}_1| \leq \sum_{(s,s') \in \mathcal{L}_2 \times \mathcal{R}_1} \mathfrak{q}(t_{j-1}, s,s')|s||s'|.
\end{equation*}

\smallskip
{\it Step 5.} Let us now observe that if $s \in \mathcal{L}_1$ and $s' \in \mathcal{R}$, then by definition of $\mathcal{L}_1$, $s,s'$ have not yet interacted at time $t_{j-1}$ and so $\mathfrak{q}(t_{j-1},s,s') = \lVert f'' \lVert_{L^\infty}$. The same holds if $s \in \mathcal{L}_2$ and $s' \in \mathcal{R}_2$. Hence, recalling \eqref{W_zzero} and \eqref{W_zero}, we get
\[
\begin{split}
\sum_{s \in \W(t_j)} |\sigma(t_j, s) - \sigma(t_{j-1}, s)||s| \leq&~
2 \Bigg[
\lVert f'' \lVert_{L^\infty} |\mathcal{L}_1||\mathcal{R}|
+ \frac{|\sigma_m(\mathcal{L}_2) - \sigma_m(\mathcal{R}_1)|}{|\mathcal{L}| + |\mathcal{R}|} |\mathcal{L}_2||\mathcal{R}_1| 
+ \lVert f'' \lVert_{L^\infty} |\mathcal{L}_2||\mathcal{R}_2| \Bigg] \\
\leq&~ 
2 \Bigg [ \sum_{(s,s') \in \mathcal{L}_1 \times \mathcal{R}} \mathfrak{q}(t_{j-1},s,s')|s||s'| 
+ \sum_{(s,s') \in \mathcal{L}_2 \times \mathcal{R}_1}\mathfrak{q}(t_{j-1},s,s')|s||s'| \\
& \quad \ + \sum_{(s,s') \in \mathcal{L}_2 \times \mathcal{R}_2}\mathfrak{q}(t_{j-1},s,s')|s||s'| \Bigg] \\
\leq&~
2 \sum_{(s,s') \in \mathcal{L} \times \mathcal{R}} \mathfrak{q}(t_{j-1},s,s')|s||s'|,
\end{split}
\]
which is what we wanted to obtain. 
\end{proof}

\subsubsection{Increasing part of $\fQ$}

This section is devoted to prove inequality \eqref{W_increase}, more precisely we will prove the following theorem.

\begin{theorem}
\label{W_increasing}
If $(t_j,x_j)$ is a cancellation point, then
\begin{equation*}
\fQ(t_j) - \fQ(t_{j-1}) \leq \log(2) \lVert f'' \lVert_{L^\infty} \TV(u(0,\cdot)) \big[ \mathcal{C}(t_j,x_j) \big],
\end{equation*}
where $\mathcal{C}(t_j,x_j)$ is the amount of cancellation at point $(t_j,x_j)$, defined in \eqref{W_canc_0}. 
\end{theorem}

\begin{proof}
To simplify the notations, w.l.o.g. assume the following wave structure (see also the proof of Proposition \ref{W_divise_tocca} and Figure \ref{fig:figura03}):
\begin{itemize}
\item in $(t_j,x_j)$ two wavefronts collide, the one coming from the left is positive and contains all and only the waves in $[r_1,r_3] \subseteq \W(t_{j-1})$, the one coming from the right is negative;
\item in $(t_j,x_j)$ all and only the waves in $[r_1,r_2] \subseteq \W(t_j)$ survive;
\item $r_1 - 1 = \max \{s \in \W(t_j) \ | \ s <r_1\}$. 
\end{itemize}
It holds
\[
\begin{split}
\fQ(t_j) - \fQ(t_{j-1}) =&~ 
\sum_{\substack{s,s' \in \W(t_j) \\ s<s'}} \mathfrak{q}(t_j,s,s')|s||s'| -
\sum_{\substack{s,s' \in \W(t_{j-1}) \\ s<s'}} \mathfrak{q}(t_{j-1},s,s')|s||s'| \\
\leq&~ \sum_{\substack{s,s' \in \W(t_j) \\ s<s'}} [\mathfrak{q}(t_j,s,s')-\mathfrak{q}(t_{j-1},s,s')]|s||s'| \\
\leq&~ \sum_{\substack{s,s' \in \W(t_j) \\ s<s'}} [\mathfrak{q}(t_j,s,s')-\mathfrak{q}(t_{j-1},s,s')]^+ \ |s||s'|.
\end{split}
\]
The proof will follow by the next three lemmas.

\begin{lemma}
\label{L_W_canc_lem_1}
If $s < s'$ are waves such that $[\mathfrak{q}(t_j,s,s')-\mathfrak{q}(t_{j-1},s,s')]^+ > 0$, then $s' \in [r_1,r_2]$ and $s \in [L(t_{j-1},r_2+1), r_2]$.
\end{lemma}

\begin{proof}[Proof of Lemma \ref{L_W_canc_lem_1}]
First let us prove $s' \in [r_1,r_2]$. By contradiction, assume $s' \notin [r_1,r_2]$. Since $[\mathfrak{q}(t_j,s,s')-\mathfrak{q}(t_{j-1},s,s')]^+ > 0$, then $\mathfrak{q}(t_j,s,s') > 0$; hence, by Remark \ref{W_rem_unite_realta}, $s,s'$ are divided in the real solution at time $t_j$. Since $s' \notin [r_1,r_2]$, $s,s'$ are divided in the real solution also at time $t_{j-1}$. Moreover, by $[\mathfrak{q}(t_j,s,s')-\mathfrak{q}(t_{j-1},s,s')]^+ > 0$ and because $(t_j,x_j)$ is a cancellation point, it must hold $\mathcal{I}(t_j,s,s') \subsetneqq \mathcal{I}(t_{j-1},s,s')$. Assume now $s,s' < r_1$, the case $s,s' > r_3$ being similar. By Proposition \ref{W_divise_tocca} the Riemann problems $\mathcal{I}(t_{j-1},s,s')$ and $\mathcal{I}(t_j,s,s')$ divide waves $r_1-1$ and $r_1$. Moreover observe that $\mathcal{I}(t_{j-1},s,s') \cap \big\{r \leq r_1-1 \big\}  = \mathcal{I}(t_j,s,s') \cap \big\{r \leq r_1-1\big\}$. Hence, by Proposition \ref{tocca},
\[
\begin{split}
\sigma(\mathcal{I}(t_{j-1},s,s'),s) =&~ \sigma(\mathcal{I}(t_{j-1},s,s') \cap \big\{r \leq r_1-1\big\},s) \\
=&~ \sigma(\mathcal{I}(t_j,s,s')  \cap \big\{r \leq r_1-1\big\},s) \\
=&~ \sigma(\mathcal{I}(t_j,s,s'),s). 
\end{split}
\]
In a similar way, $\sigma(\mathcal{I}(t_{j-1},s,s'),s') = \sigma(\mathcal{I}(t_j,s,s'),s')$; thus $\mathfrak{q}(t_{j-1},s,s') = \mathfrak{q}(t_j,s,s')$, contradicting the assumption.

Let us now prove the second part of the statement, namely $s \in [L(t_{j-1},r_2+1), r_2]$. Since $s' \leq r_2$, then $s<r_2$. Assume by contradiction $s < L(t_{j-1},r_2+1)$. In this case $R(t_{j-1},s) \leq r_2$, hence $R(t_{j-1},s) = R(t_j,s)$; thus $\mathcal{I}(t_{j-1},s,s') = \mathcal{I}(t_j,s,s')$ and so  $\mathfrak{q}(t_{j-1},s,s') = \mathfrak{q}(t_j,s,s')$, a contradiction with the initial assumption.
\end{proof}

Consider the following values: 
\begin{align*}
a &:= \hat u(L(t_{j-1},r_2+1)) -\e, & b_1 &:= \hat u(r_1) -\e, & b_2 &:= \hat u(r_2), & b_3 &:= \hat u(r_3). 
\end{align*}
Define two functions $g,h: (-\infty, b_3] \to \R$, piecewise affine, with discontinuity points of the derivative in the set $\Z\e$ and such that 
\begin{align*}
g|_{[b_1,b_3]} & = \conv_{[b_1,b_3]} f, & h|_{[b_1,b_2]} & = \conv_{[b_1,b_2]} f, & g &= h \ \text{ on } (-\infty,b_1].
\end{align*}
Clearly $g$, $h$ with the above properties exist. 

\begin{lemma}
\label{W_lemma_uno}
For any $s' \in [r_1,r_2]$ and $s \in [L(t_{j-1},r_2+1), r_2]$, $s<s'$, it holds
\begin{equation}
\label{E_variat_spe_lem1}
\begin{split}
\big[ &\mathfrak{q}(t_j,s,s')-\mathfrak{q}(t_{j-1},s,s') \big]^+  \\
\leq&~ \frac
{\bigg| \Big[ h' \big( (\hat u(s')-\e, \hat u(s')) \big) - h'\big((\hat u(s)-\e, \hat u(s))\big) \Big] - \Big[ g'\big( (\hat u(s')-\e, \hat u(s')) \big) - g'\Big( (\hat u(s)-\e, \hat u(s)) \big) \Big] \bigg|}
{|\hat u(s') - (\hat u(s)-\e)|}. 
\end{split}
\end{equation}
\end{lemma}

In other words, we are saying that the maximal variation of $\mathfrak q(t,s,s')$ is controlled by the maximal variation of speed (the numerator of the r.h.s. of \eqref{E_variat_spe_lem1}) divided by the total variation between the two waves $s$, $s'$.

\begin{proof}[Proof of Lemma \ref{W_lemma_uno}]
Fix $s,s'$ as in the hypothesis and set for shortness
\begin{align*}
\mathcal{I}_{j-1} & := \mathcal{I}(t_{j-1},s,s'),       &  
\sigma_{j-1}(s)   & := \sigma(\mathcal{I}_{j-1},s),     &
\sigma_{j-1}(s')  & := \sigma(\mathcal{I}_{j-1},s'),    \\
\mathcal{I}_j     & := \mathcal{I}(t_j,s,s'),           &
\sigma_j(s)       & := \sigma(\mathcal{I}_j,s),         &
\sigma_j(s')  & := \sigma(\mathcal{I}_j,s').            \\
\end{align*}
By multiplication for $|\hat u(s') - (\hat u(s)-\e)|$, \eqref{E_variat_spe_lem1} becomes
\begin{align*}
\Big[ \Big( \sigma_j(s') - \sigma_j(s) \Big) - \Big( \sigma_{j-1}(s') - \sigma_{j-1}(s) \Big) \Big]^+
\leq&~ \bigg| \bigg(h'\Big( \big(\hat u(s')-\e, \hat u(s')\big)\Big) - h'\Big(\big(\hat u(s)-\e, \hat u(s)\big)\Big)\bigg) \\
&~ - \bigg(g'\Big(\big(\hat u(s')-\e, \hat u(s')\big)\Big) - g'\Big(\big(\hat u(s)-\e, \hat u(s)\big)\Big)\bigg) \bigg|.
\end{align*}

We consider two cases depending on the position of the wave $s$.

\smallskip
{\it Case 1.} Assume first $s \in [L(t_{j-1},r_2+1), r_1-1]$. In this case $s,s'$ are divided in the real solution both at time $t_{j-1}$ and at time $t_j$. Hence, by Proposition \ref{W_divise_tocca} the Riemann problems $\mathcal{I}_{j-1}$, $\mathcal{I}_j$ divide the waves $r_1-1$ and $r_1$. Thus, by Proposition \ref{tocca}, and by the fact that $\mathcal{I}_{j-1} \cap \big\{r \leq r_1-1\big\} = \mathcal{I}_j \cap \big\{r \leq r_1-1\big\}$, 
\begin{equation*}
\sigma_{j-1}(s) = \sigma(\mathcal{I}_{j-1} \cap \big\{r \leq r_1-1\big\},s) = \sigma(\mathcal{I}_j \cap \big\{r \leq r_1-1\big\},s) = \sigma_j(s),
\end{equation*}
whence 
\begin{equation}
\label{W_diff_sigma}
\sigma_{j-1}(s) - \sigma_j (s) = 0 = h'\Big((\hat u(s)-\e, \hat u(s))\Big) - g'\Big((\hat u(s)-\e, \hat u(s) )\Big)
\end{equation}
(recall that $g = h$ on $(-\infty, b_1]$).
Now distinguish two subcases:
	\begin{enumerate}
	\item If $\max \mathcal{I}_{j-1} \leq r_2$, then $\mathcal{I}_{j-1} \cap [r_1,r_2] = \mathcal{I}_j \cap [r_1,r_2]$, and so, arguing as above, $\sigma_{j-1}(s') = \sigma_j (s')$. 
	\item If $\max \mathcal{I}_{j-1} > r_2$, then $\max \mathcal{I}_j = r_2$ and so 
	\begin{equation*}
	\sigma_j(s') \overset{\text{(Prop. \ref{tocca})}}{=} \sigma(\mathcal{I}_j \cap [r_1,r_2],s') = \sigma([r_1,r_2], s') = h'\Big((\hat u(s')-\e, \hat u(s'))\Big),
	\end{equation*}
  	while
	\begin{equation*}
	\sigma_{j-1}(s') \overset{\text{(Prop. \ref{tocca})}}{=} \sigma(\mathcal{I}_{j-1} \cap [r_1,r_3],s') \overset{\text{(Prop. \ref{vel_aumenta})}}{\geq} \sigma([r_1,r_3],s') = g'\Big((\hat u(s')-\e, \hat u(s'))\Big), 
	\end{equation*}
	where the inequality is an easy consequence of Proposition \ref{vel_aumenta}.
	\end{enumerate}
In both case (a) and (b), 
\[
\sigma_j(s') - \sigma_{j-1}(s') \leq  h'\Big((\hat u(s')-\e, \hat u(s'))\Big) - g'\Big((\hat u(s')-\e, \hat u(s'))\Big).
\]
Together with \eqref{W_diff_sigma}, this yields the thesis.

\smallskip
{\it Case 2.} Assume now $s \in [r_1, r_2]$. In this case 
\begin{align*}
\min \mathcal{I}_{j-1}         & \leq r_1, &  \max \mathcal{I}_{j-1}      & =  r_3, \\
\min \mathcal{I}_j             & \leq r_1,   &  \max \mathcal{I}_j          & =  r_2.
\end{align*}
Since $s,s'$ are divided at time $t_j$ in the real solution, we can argue as in the previous case and use Proposition \ref{W_divise_tocca} to obtain
\begin{equation}
\label{W_sigma_n+1}
\sigma_{j}(s') - \sigma_{j}(s) = h'\Big((\hat u(s')-\e, \hat u(s'))\Big) - h'\Big((\hat u(s)-\e, \hat u(s))\Big).
\end{equation}
Observe that since $s,s'$ are joined before the collision in the real solution,
\[
g'\Big((\hat u(s')-\e, \hat u(s'))\Big) - g'\Big((\hat u(s)-\e, \hat u(s))\Big) = 0.
\]
Moreover  by Proposition \ref{W_unite_realta}, $\sigma_{j-1}(s') - \sigma_{j-1}(s) = 0$, which, together with \eqref{W_sigma_n+1}, yields the thesis.
\end{proof}

\begin{lemma}
\label{W_lemma_due}
Define $\varphi := h-g$. Then
\begin{multline*}
\sum_{\substack{s \in [L(t_{j-1},r_2+1), r_2] \\ s' \in [r_1,r_2], \ s<s'}} 
\frac
{\Big|\varphi' \big( (\hat u(s')-\e, \hat u(s') ) \big) - \varphi' \big( (\hat u(s)-\e, \hat u(s) ) \big) \Big| |s||s'|}
{|\hat u(s') - (\hat u(s)-\e)|} \\
\leq \log(2) \TV(u_\e(0,\cdot)) \Big[ h'(b_2-) - g'(b_2-)\Big].
\end{multline*}
\end{lemma}

\begin{proof}
First of all, let us observe that, since $\varphi'$ is increasing and $s<s'$ with positive sign $\mathcal S$, we can forget about the absolute value and get 
\begin{align*}
& \sum_{\substack{s \in [L(t_{j-1},r_2+1), r_2] \\ s' \in [r_1,r_2], \ s<s'}}
\frac
{\bigg[ \varphi' \Big( (\hat u(s')-\e, \hat u(s') ) \Big) - \varphi' \Big( (\hat u(s)-\e, \hat u(s) ) \Big) \bigg] |s||s'|}
{\hat u(s') - (\hat u(s)-\e)} \\
& \quad \leq 
\sum_{\substack{s \in [L(t_{j-1},r_2+1), r_2] \\ s' \in [L(t_{j-1},r_2+1), r_2] \\ s<s'}} 
\frac
{\bigg[\varphi' \Big( (\hat u(s')-\e, \hat u(s') ) \Big) - \varphi' \Big( (\hat u(s)-\e, \hat u(s) ) \Big) \bigg] |s||s'|}
{\hat u(s') - (\hat u(s)-\e)} \\
& \quad = 
\sum_{\substack{s \in [L(t_{j-1},r_2+1), r_2] \\ s' \in [L(t_{j-1},r_2+1), r_2] \\ s<s'}} 
\frac
{\bigg[ \varphi' \Big( (\hat u(s')-\e, \hat u(s') ) \Big) - \varphi' \Big( (\hat u(s)-\e, \hat u(s) ) \Big) \bigg]}
{\hat u(s') - (\hat u(s)-\e)}
\int_{\hat u(s)-\e}^{\hat u(s)} du \int_{\hat u(s')-\e}^{\hat u(s')} du'. \\
\end{align*}
Since $\varphi'\big((\hat u(s)-\e, \hat u(s) )\big) = \varphi'(u)$ for each $u \in (\hat u(s) -\e, \hat u(s))$, and similarly for $s'$, we can continue the chain of inequality in the following way:
\begin{align*}
\sum_{\substack{s \in [L(t_{j-1},r_2+1), r_2] \\ s' \in [L(t_{j-1},r_2+1), r_2] \\ s<s'}} &
\frac
{\bigg[\varphi' \Big( (\hat u(s')-\e, \hat u(s') ) \Big) - \varphi' \Big( (\hat u(s)-\e, \hat u(s) ) \Big) \bigg]}
{\hat u(s') - (\hat u(s)-\e)}
\int_{\hat u(s)-\e}^{\hat u(s)} du \int_{\hat u(s')-\e}^{\hat u(s')} du' \\
=&~ \sum_{\substack{s \in [L(t_{j-1},r_2+1), r_2] \\ s' \in [L(t_{j-1},r_2+1), r_2] \\ s<s'}} 
\frac{1}{\hat u(s') - (\hat u(s)-\e)}
\int_{\hat u(s)-\e}^{\hat u(s)} \int_{\hat u(s')-\e}^{\hat u(s')} (\varphi'(u') - \varphi'(u)) du'du \\
\leq&~ \sum_{\substack{s \in [L(t_{j-1},r_2+1), r_2] \\ s' \in [L(t_{j-1},r_2+1), r_2] \\ s<s'}} 
\int_{\hat u(s)-\e}^{\hat u(s)} \int_{\hat u(s')-\e}^{\hat u(s')} \frac{\varphi'(u') - \varphi'(u)}{u'-u} du'du \\
\leq&~  
\int_{a}^{b_2} \int_{u}^{b_2} \frac{\varphi'(u') - \varphi'(u)}{u'-u} du'du. \\
\end{align*}
Since $\varphi = h-g$ is piecewise affine, we can write
\begin{align*}
\int_{a}^{b_2} \int_{u}^{b_2} \frac{\varphi'(u') - \varphi'(u)}{u'-u} du'du 
= &~
\int_{a}^{b_2} \int_{u}^{b_2} \frac{1}{u'-u} \bigg[\sum_{\substack{m \in \Z \\ u < m\e < u'}} \varphi'(m\e+) - \varphi'(m\e-) \bigg]du'du \\
= &~
\sum_{\substack{m \in \Z \\ a < m\e < b_2}} [\varphi'(m\e+) - \varphi'(m\e-)] \int_{a}^{m\e} \int_{m\e}^{b_2} \frac{1}{u'-u} du'du. \\ 
\end{align*}
An elementary computation shows that 
\begin{equation}
\label{logaritmo}
\max_{\xi \in [a,b_2]} \int_{a}^{\xi} \int_{\xi}^{b_2} \frac{1}{u'-u} du'du = \log(2) (b_2-a).
\end{equation}
Hence
\begin{align*}
\sum_{\substack{m \in \Z \\ a < m\e < b_2}} [\varphi'(m\e+) - \varphi'(m\e-)] \int_{a}^{m\e} \int_{m\e}^{b_2} \frac{1}{u'-u} du'du
\leq&~ 
\log(2) (b_2 - a) \sum_{\substack{m \in \Z \\ a < m\e < b_2}} [\varphi'(m\e+) - \varphi'(m\e-)] \\ 
\leq&~
\log(2) \TV(u_\e(0,\cdot)) \big[ \varphi'(b_2-) - \varphi'(a+) \big] \\
\text{(since $\varphi'(a+) \geq 0$ by Proposition \ref{vel_aumenta})} \ \leq&~
\log(2) \TV(u_\e(0,\cdot)) \varphi'(b_2-) \\
\leq&~
\log(2) \TV(u_\e(0,\cdot)) \big[ h'(b_2-) - g'(b_2-) \big]. 
\end{align*}
This concludes the proof of Lemma \ref{W_lemma_due}.
\end{proof}

\textbf{Conclusion of proof of Theorem \ref{W_increasing}}. Observe that, by Proposition \ref{diff_vel_proporzionale_canc}, 
\begin{equation}
\label{W_lemma_tre}
h'(b_2-) - g'(b_2-) \leq \lVert f'' \lVert_{L^\infty} (b_3 - b_2) \leq \lVert f'' \lVert_{L^\infty} \ \mathcal{C}(t_j,x_j).
\end{equation}
Putting together Lemma \ref{W_lemma_uno}, Lemma \ref{W_lemma_due}, inequality \eqref{W_lemma_tre} and inequality \eqref{bd_su_dato_iniziale} one easily concludes the proof of the theorem.
\end{proof}

\section{Analysis of the Glimm scheme}
\label{S_glimm_scheme}

In this section we prove the main interaction estimate \eqref{E_G_est_fin} for an approximate solution of the Cauchy problem \eqref{cauchy} obtained by Glimm scheme. The line of the proof is very similar to the proof for the wavefront tracking case, even if a relevant number of technicalities arises. For this reason the structure of this section is equal to the structure of Section \ref{sect_wavefront}: throughout the remaining part of this paper, we will emphasize the differences in definitions and proofs between the wavefront algorithm and the Glimm scheme.

First let us briefly recall how Glimm scheme constructs an approximate solution. Fix $\e >0$. To construct an approximate solution $u_\e = u_\e(t,x)$ to the Cauchy problem \eqref{cauchy}, we start with a grid in the $(t,x)$ plane having step size $\Delta t = \Delta x = \e$, with nodes at the points
\[
P_{n,m} = (t_n,x_m) := (n \e, m \e), \qquad n,m \in \Z. 
\]
Moreover we shall need a sequence of real numbers $\vartheta_1, \vartheta_2, \vartheta_3, \dots$, uniformly distributed over the interval $[0,1]$. This means that, for every $\lambda \in [0,1]$, the percentage of points $\vartheta_n, 1 \leq n \leq N$, which fall inside $[0,\lambda]$ should approach $\lambda$ as $N \rightarrow \infty$:
\begin{equation*}
\lim_{N \rightarrow \infty} \frac{\card\{n \ | \ 1 \leq n \leq N, \vartheta_n \in [0, \lambda]\}}{N} = \lambda \hspace{1cm} \text{for each } \lambda \in [0,1]. 
\end{equation*}

At time $t=0$, the Glimm algorithm starts by taking an approximation $\bar u_\e$ of the initial datum $\bar u$, which is constant on each interval of the form $(x_{m-1}, x_m)$, and has jumps only at the nodal points $x_m := m \e$. We shall take (remember that $\bar u$ is right continuous)
\begin{equation}
\label{Glimm_initial_datum}
\bar u_\e(x) = \bar u(x_m) \hspace{1cm} \text{for all } x \in [x_m, x_{m+1}).
\end{equation}
Notice that, as in the wavefront tracking case, estimate \eqref{bd_su_dato_iniziale} holds also in this case. For times $t>0$ sufficiently small, the solution $u_\e = u_\e(t,x)$ is then obtained by solving the Riemann problems corresponding to the jumps of the initial approximation $u_\e(0, \cdot)$ at the nodes $x_m$. Since $0 < f'(u) < 1$ for any $u \in \R$, the solutions to the Riemann problems do not overlap and thus $u_\e(t)$ can be prolonged on the whole time interval $[0, \e)$. At time $t_1 = \e$ a restarting procedure is adopted: the function $u_\e(t_1-, \cdot)$ is approximate by a new function $u_\e(t_1+,\cdot)$ which is piecewise constant, having jumps exactly at the nodes $x_m = m \e$. Our approximate solution $u_\e$ can now be constructed on the further time interval $[\e, 2\e)$, again by piecing together the solutions of the various Riemann problems determined by the jumps at the nodal points $x_m$. At time $t_2 = 2\e$, this solution is again approximated by a piecewise constant function, etc...

A key aspect of the construction is the restarting procedure. At each time $t_n = n \e$, we need to approximate $u_\e(t_n-, \cdot)$ with a piecewise constant function $u_\e(t_n+, \cdot)$ having jumps precisely at the nodal points $x_m$. This is achieved by a random sampling technique. More precisely, we look at the number $\vartheta_n$ in the uniformly distributed sequence. On each interval $[x_{m-1},x_m)$, the old value of our solution at the intermediate point $\vartheta_n x_m + (1-\vartheta_n) x_{m-1}$ becomes the new value over the whole interval:
\[
u_\e(t_n+, x) := u_\e \big( t_n-, (\vartheta_n x_m + (1-\vartheta_n) x_{m-1}) \big) \hspace{1cm} \text{for all } x \in [x_{m-1},x_m).
\]

One can prove that, if the initial datum $\bar u$ has bounded total variation, then an approximate solution can be constructed by the above algorithm for all times $t \geq 0$. 
Finally set, by simplicity,
\[
u_{n,m} := u_\e(n \varepsilon, m\varepsilon).
\]

\subsection{Definition of waves for the Glimm scheme}
\label{S_front_glimm}

Similarly to Section \ref{Front_Waves}, we define the notion of \emph{wave}, the notion of \emph{position of a wave} and the notion of \emph{speed of a wave}.

\subsubsection{Enumeration of waves}

First of all, as we did for the wavefront tracking scheme in Section \ref{Front_eow}, we define the notion of \emph{enumeration of waves} related to a BV function $u: \R_x \to \R$ of the single variable $x$.

The biggest differences with the wavefront tracking case are that here we allow the set of waves $\W$ to be a subset of $\R$ (while in the wavefront case it is a subset of $\N$) and we allow the state function $\hat u$ to take values in $\R$ (while in the wavefront case it takes values in $\Z\e$). This is due to the fact that in an $\e$-approximate solution constructed with the wavefront algorithm any wavefront has strength at least equal to $\e$, while in an approximate solution obtained by the Glimm scheme rarefactions (i.e. infinitesimal wavefronts) are allowed.

\begin{definition}
\label{eow}
Let $u: \R \to \R$, $u \in BV(\R)$, be a piecewise constant, right continuous function, with jumps located at points of the form $m \varepsilon$, $m \in \Z$. An \emph{enumeration of waves} for the function $u$ is a 3-tuple $(\mathcal{W}, \mathtt x, \hat u)$, where 
\[
\begin{array}{ll}
\mathcal{W} \subseteq \R, & \text{is \emph{the set of waves}},\\
\mathtt x: \mathcal{W} \to (-\infty, +\infty], & \text{is \emph{the position function}}, \\
\hat u: \mathcal{W} \to \R, & \text{is \emph{the state function}}, \\
\end{array}
\]
with the following properties:
\begin{enumerate}
 \item the function $\mathtt x$ takes values only in the set $\Z\e \cup \{+\infty\}$;
 \item the restriction of the function $\mathtt x$ to the set of waves where it takes finite values is increasing; 
 \item for given $x_0 \in \R$, consider $\mathtt x^{-1}(x_0) = \{s \in \mathcal{W} \ | \ \mathtt x(s) = x_0\}$; then $\mathtt x^{-1}(x_0)$ is a finite union (possibily empty) of intervals of the form $(a,b]$; moreover it holds:
 \begin{enumerate}
   \item if $u(x_0-) < u(x_0)$, then $\hat u|_{\mathtt x^{-1}(x_0)}: \mathtt x^{-1}(x_0) \to (u(x_0-), u(x_0)]$ is strictly increasing and bijective, with $\frac{d \hat u}{ds}(s) = 1$ for each $s$ in the interior of $\mathtt x^{-1}(x_0)$; 
   \item if $u(x_0-) > u(x_0)$, then $\hat u|_{\mathtt x^{-1}(x_0)}: \mathtt x^{-1}(x_0) \to [u(x_0), u(x_0-))$ is strictly decreasing and bijective, with $\frac{d \hat u}{ds}(s) = -1$ for each $s$ in the interior of $\mathtt x^{-1}(x_0)$; 
   \item if $u(x_0-) = u(x_0)$, then $\mathtt x^{-1}(x_0) = \emptyset$.
 \end{enumerate}
\end{enumerate}
\end{definition}

Given an enumeration of waves as in Definition \ref{eow}, we define \emph{the sign} of a wave $s \in W$ with finite position (i.e. such that $\mathtt x(s) < +\infty$) as follows:
\begin{equation}
\label{sign}
\mathcal{S}(s) := \sign \Big[u(\mathtt x(s)) - u(\mathtt x(s)-) \Big].
\end{equation}


The analog of Example \ref{W_initial_eow} is the next example. 

\begin{example}
\label{initial_eow}
Fix $\e>0$ and let $\bar u_\e \in BV(\R)$ be the approximate initial datum of the Cauchy problem \eqref{cauchy} constructed in \eqref{Glimm_initial_datum}. Let
\[
U: \R \to [0, \TV(\bar u_\e)]
\]
be the total variation function, defined as $U(x) := \TV(\bar u_\e; (-\infty, x])$. Then define:
\[
\mathcal{W} := (0, \TV(\bar u_\e)]
\]
and
\[
\mathtt x_0 : \mathcal{W} \to (-\infty, +\infty], \qquad \mathtt x_0(s) := \inf\Big\{x \in (-\infty, +\infty] \ | \ s \leq U(x) \Big\}.
\]
Moreover, recalling \eqref{sign}, we define 
\begin{equation*}
\hat u: \mathcal{W} \to \R, \qquad \hat u(s) := \bar u_\e(\mathtt x_0(s)-) + \mathcal{S}(s)\Big[s - U(\mathtt x_0(s)-) \Big ].
\end{equation*}
One  can easily prove that $\hat u$ is continuous.
\end{example}


We also adapt Definition \ref{W_speed_function} in the following way.

\begin{definition}
\label{speed_function}
Consider a function $u$ as in Definition \ref{eow} and let $(\mathcal{W}, \mathtt x, \hat u)$ be an enumeration of waves for $u$. The \emph{speed function} $\sigma: \mathcal{W} \to [0,1] \cup \{+\infty\}$ is defined as follows:
\[
\sigma(s) := \left\{
\begin{array}{ll}
+\infty & \text{if } \mathtt x(s) = +\infty, \\
\Big(\frac{d}{du}\conv_{[u(\mathtt x(s)-),u(\mathtt x(s))]}f\Big)\big(\hat u(s)\big) & \text{if } \mathcal{S}(s) = +1, \\ 
\Big(\frac{d}{du}\conc_{[u(\mathtt x(s)),u(\mathtt x(s)-)]}f\Big)\big(\hat u(s)\big) & \text{if } \mathcal{S}(s) = -1. 
\end{array}
\right.
\]
\end{definition}


\subsubsection{Position and speed of the waves}
\label{pswaves}

Consider the Cauchy problem \eqref{cauchy} and fix $\e >0$; let $u_\e = u_\e(t,x)$ be the $\varepsilon$-Glimm solution, with sampling sequence $\{\vartheta_n\}_n$. For the (piecewise constant) initial datum $u_\e(0,\cdot)$, consider the enumeration of waves $(\mathcal{W}, \mathtt x_0, \hat u)$ provided in Example \ref{initial_eow}; let $\mathcal{S}$ be the sign function defined in \eqref{sign} for this enumeration of waves. 

Our aim is to define two functions
\[
\mathtt x: [0, +\infty)_t \times \mathcal{W} \to \R_x \cup \{+\infty\}, \qquad
\sigma: [0, +\infty)_t \times \mathcal{W} \to [0,1] \cup \{+\infty\},
\]
called \emph{the position at time $t \in [0, +\infty)$ of the wave $s \in \mathcal{W}$} and \emph{the speed at time $t \in [0, +\infty)$ of the wave $s \in \mathcal{W}$}.  

Let us begin with the definition of $\mathtt x$. We first define $\mathtt x$ on the domain $\N \varepsilon \times \mathcal{W}$ and then we extend it over all the domain $[0, +\infty)_t \times \mathcal{W}$.
The definition of $\mathtt x(n \varepsilon, s)$ is given by recursion on $n$, i.e. for each $n \in \N$ we define a function $\mathtt x(n \varepsilon, \cdot) : \mathcal{W} \to (-\infty, + \infty]$ such that the 3-tuple $(\mathcal{W}, \mathtt x(n \varepsilon, \cdot), \hat u)$ is an enumeration of waves for the piecewise constant function $u_\e(n \varepsilon, \cdot)$.

For $n = 0$ we set $\mathtt x(0,s) := x_0(s)$, where $x_0(\cdot)$ is the position function in the enumeration of waves of Example \ref{initial_eow}. Clearly $(\mathcal{W}, \mathtt x(0, \cdot), \hat u)$ is an enumeration of waves for the function $u_\e(0, \cdot)$.

Assume now to have obtained $\mathtt x(n \varepsilon, \cdot)$ with the property stated above and define $\mathtt x((n+1) \varepsilon, \cdot)$ in the following way. 

Let $\sigma(n \varepsilon, \cdot)$ be the speed function related to the piecewise constant function $u_\e(n \varepsilon, \cdot)$ and the enumeration of waves $(\mathcal{W}, \mathtt x(n \varepsilon, \cdot), \hat u)$. Then set
\begin{align*}
\mathcal{W}(n \varepsilon, m\e) :=&~ \{s \in \mathcal{W} \ | \ \mathtt x(n\varepsilon, s) = m\varepsilon\}, \\
\mathcal{W}^{(0)}(n \varepsilon, m\e) :=&~ \{s \in \mathcal{W} \ | \ \mathtt x(n\varepsilon, s) = m\varepsilon \text{ and } \sigma(n\varepsilon, s) \leq \vartheta_{n+1}\}, \\
\mathcal{W}^{(1)}(n \varepsilon, m\e) :=&~ \{s \in \mathcal{W} \ | \ \mathtt x(n\varepsilon, s) = m\varepsilon \text{ and } \sigma(n\varepsilon, s) > \vartheta_{n+1}\}. \\
\end{align*}

Define 
\begin{align*}
\mathcal{W}((n+1) \varepsilon, m\varepsilon) := \bigg\{ s \in \mathcal{W}^{(1)}(n \varepsilon, (m-1)\e) \cup \mathcal{W}^{(0)}(n \varepsilon, m\e) \ \Big| \ 
\mathcal{S}(s) = \sign({u_{n+1,m} - u_{n+1,m-1}}) \text{ and } & \\
\mathcal{S}(s)u_{n+1,m-1} \leq \mathcal{S}(s) \hat u(s) \leq \mathcal{S}(s)u_{n+1,m} & \bigg\}.
\end{align*}
where $\mathcal{S}(s)$ was defined in \eqref{sign}, with respect to the enumeration of waves for the initial datum, and  set
\begin{equation*}
\mathtt x((n+1)\varepsilon, s) :=
\begin{cases}
m \varepsilon & \text{if } s \in \mathcal{W}((n+1) \varepsilon, m\varepsilon), \\
+ \infty & \text{otherwise}.
\end{cases}
\end{equation*}
Observe that for each $s \in \mathcal{W}$, $\mathtt x((n+1) \varepsilon, s)$ is uniquely defined since, for $m_1 \neq m_2$, sets $\mathcal{W}^{(1)}(n \varepsilon, (m_i-1)\e) \cup \mathcal{W}^{(0)}(n \varepsilon, m_i\e)$, $i=1,2$, are disjoint.

A fairly easy extension of Lemma \ref{W_lemma_eow} shows that the recursive procedure generates an enumeration of waves at the times $n\e$. The only differences in the proof are that now the map $\hat u$ takes continuous values and that at each time $n\e$ you have to consider countably many disjoint Riemann problems, instead of a single one. 

We have just defined the position function $\mathtt x$ on the domain $\N \varepsilon \times \mathcal{W}$. Let us now extend it over all the domain $[0,+\infty)_t \times \mathcal{W}$, by setting, for $t \in [n\varepsilon, (n+1)\varepsilon)$,
\[
\mathtt x(t,s) :=
\begin{cases}
+\infty & \text{if } \mathtt x(n\varepsilon, s) = + \infty, \\
\mathtt x(n\varepsilon,s)  & \text{if } s \in \mathcal{W}^{(0)}(n\varepsilon, m\varepsilon), \\
\mathtt x(n\varepsilon,s) + t-n\varepsilon & \text{if } s \in \mathcal{W}^{(1)}(n\varepsilon, m\varepsilon). 
\end{cases}
\]
See Figure \ref{fig:figura32}.

\begin{figure}
  \begin{center}
    \includegraphics[height=5cm,width=8cm]{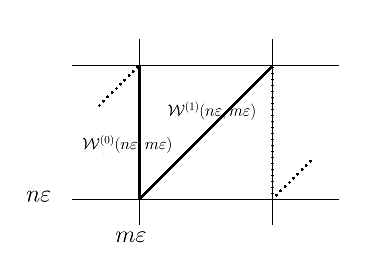}
    \caption{Graphical description of the definition of position of waves.}
    \label{fig:figura32}
    \end{center}
\end{figure}

The speed $\sigma: [0, +\infty)_t \times \W \to [0,1] \cup \{+\infty\}$ is defined setting for $t \in [n \varepsilon, (n+1) \varepsilon)$ and for $s \in \W$, $\sigma(t,s) := \sigma(n \varepsilon, s)$.
As in Section \ref{W_pswaves}, we set 
\[
\W(t) := \big\{ s \in \W \ | \ \mathtt x(t,s) < +\infty \big\}, \qquad
\W(t,x) := \big\{ s \in \W \ | \ \mathtt x(t,s) = x \big\}.
\]
and we will call $\W(t)$ the set of the \emph{real waves} at time $t$, while we will say that a wave $s$ is \emph{removed} at time $t$ if $\mathtt x(t,s) = + \infty$. Observe that the definition of $\W(t,x)$ agrees with the definition of $\W(n\varepsilon, m\varepsilon)$ given before, for $t=n \varepsilon, x=m\varepsilon$.
Finally for any $s \in \W(0)$ set
\[
T(s) := \sup \Big\{t \in [0, +\infty) \ | \ \mathtt x(t,s) < +\infty\Big\}.
\]
The curve $t \mapsto \mathtt x(t,s)$, for $t \in [0, T(s))$, as a curve in the $(t,x)$-plane, is piecewise linear with nodes in the grid points. 

\subsubsection{Interval of waves}

This section is analog of Section \ref{W_iow}. 

\begin{definition}
Let $\bar t$ be a fixed time and $\mathcal{I} \subseteq \W(\bar t)$. We say that $\mathcal{I}$ is an \emph{interval of waves} at time $\bar t$ if for any given $s_1, s_2 \in \mathcal{I}$, with $s_1 < s_2$, and for any $p \in \W(\bar t)$
\begin{equation*}
s_1 \leq p \leq s_2 \quad \Longrightarrow \quad p \in \mathcal{I}.
\end{equation*}
We say that an interval of waves $\mathcal{I}$ is \emph{homogeneous} if for each $s,s' \in \mathcal{I}$, $\mathcal{S}(s) = \mathcal{S}(s')$.  If waves in $\mathcal{I}$ are positive (resp. negative), we say that $\mathcal{I}$ is a \emph{positive} (resp. \emph{negative}) interval of waves. 
\end{definition}


The following proposition is the continuous version of Proposition \ref{W_interval_waves}. Its proof is very similar and for this reason we omit it.

\begin{proposition}
\label{interval_waves}
Let $\mathcal{I} \subseteq \W(\bar t)$ be an homogeneous interval of waves and assume that waves in $\mathcal{I}$ are positive (resp. negative). Then the restricion of $\hat u$ to $\mathcal{I}$ is strictly increasing (resp. decreasing) and its image is an interval (in $\R$). 
\end{proposition}

A priori an interval of waves could be a wild subset of $\R$: we show that it enjoys a particularly nice structure.

\begin{proposition}
\label{interval_countable_prop}
Let $\mathcal{I} \subseteq \W(n\e)$ be an interval of waves. Then $\mathcal{I}$ is an (at most countable) union of mutually disjoint intervals of real numbers.
\end{proposition}


\begin{proof}
Write
\[
\mathcal{I} =\bigcup_{m \in \Z}  \mathcal{I} \cap \W(n\e,m\e).
\]
Clearly $\mathcal{I} \cap \W(n\e,m\e) \subseteq \W(n\e,m\e)$. If equality holds, then, by Definition \ref{eow} of enumeration of waves, $\mathcal{I} \cap \W(n\e,m\e)$ is a finite union of intervals of real numbers. Otherwise, if $\mathcal{I} \cap \W(n\e,m\e) \subsetneqq \W(n\e,m\e)$, then define
\begin{equation*}
a := \inf \big( \mathcal{I} \cap \W(n\e,m\e) \big), \qquad
b := \sup \big( \mathcal{I} \cap \W(n\e,m\e) \big).
\end{equation*}
By definition of interval of waves,
\begin{equation}
\label{interval_countable_eq}
\big( \mathcal{I}  \cap \W(n\e,m\e) \big) \setminus \{a,b\} = (a,b) \cap \W(n\e,m\e).
\end{equation}
The r.h.s. of \eqref{interval_countable_eq} is a finite union of mutually disjoint intervals of real numbers.
\end{proof}

\begin{corollary}
\label{iow_are_measurable}
Any interval of waves is Lebesgue-measurable.
\end{corollary}

The next proposition is the area formula adapted to our setting.

\begin{proposition}
\label{cambio_di_variabile}
Let $g: \R \to \R$ be continuous, $g \geq 0$. Let $\mathcal{I} \subseteq \W(n\e,m\e)$ be a homogeneous interval of waves and let $I := \hat u(\mathcal{I})$. Then
\[\
\int_{\mathcal{I}} g(\hat u(s)) ds = \int_I g(u) du.
\]
\end{proposition}
Observe that the function $g \circ \hat u$ is Lebesgue-measurable on $\mathcal{I}$ since it is composition of continuous functions.

Definition \ref{W_artificial_speed} and Remark \ref{W_artificial_speed_remark} in this context become respectively the following.

\begin{definition}
\label{artificial_speed}
Let $I \subseteq \R$ be an interval in $\R$. Let $s$ be any positive (resp. negative) wave such that $\hat u(s) \in I$. The quantity $\sigma(I,s) := \frac{d}{du} \conv_I f(\hat u(s))$ (resp. $\sigma(I,s) := \frac{d}{du} \conc_I f(\hat u(s))$) is called \emph{the (artificial) speed given to the wave $s$ by the Riemann problem $I$}. 

Moreover, if $s,s'$ are two waves such that $\hat u(s), \hat u(s') \in I$, we say that \emph{the Riemann problem $I$ divides $s,s'$} if $\hat u(s), \hat u(s')$ do not belong to the same wavefront of $\conv_I f$ (resp. $\conc_I f$).
\end{definition}

\begin{remark}
\label{artificial_speed_remark}
In the previous definition, if $I = \hat u(\mathcal{I})$, where $\mathcal{I}$ is any interval of waves at fixed time $\bar t$, we will also write $\sigma(\mathcal{I},s)$ instead of $\sigma(I,s)$ and call it the speed given to the waves $s$ by the Riemann problem $\mathcal{I}$. Moreover, we will also say that the Riemann problem $\mathcal{I}$ divides $s,s'$ if the Riemann problem $I$ does.
\end{remark}

\subsection{The main theorem in the Glimm scheme approximation}

Now we state the main interaction estimate for Glimm scheme. Fix a wave $s \in \W(0)$ and consider the function $t \mapsto \sigma(t,s)$. By construction it is finite valued until the time $T(s)$, after which its value becomes $+\infty$; moreover it is piecewise constant, right continuous, with jumps located at times $t = n\varepsilon, n \in \N$.

\begin{theorem1}
The following estimate holds:
\[
\sum_{n=1}^{+\infty} \int_{\W(n\varepsilon)} \big| \sigma(n \varepsilon, s) - \sigma((n-1)\varepsilon, s) \big| \, ds \leq (3 + 2 \log(2)) \lVert f'' \lVert_{L^\infty} \TV(u(0,\cdot))^2.
\]
\end{theorem1}

\noindent Corollary \ref{iow_are_measurable}, Definition \ref{speed_function} and Proposition \ref{cambio_di_variabile} gives that the integral is meaningful.

To prove Theorem \ref{main_thm} let us write
\begin{equation}
\label{da tempo a punti}
\sum_{n=1}^{+\infty} \int_{\W(n \varepsilon)} \big| \sigma(n \varepsilon, s) - \sigma((n-1)\varepsilon, s) \big| \, ds = 
\sum_{n=1}^{+\infty} \sum_{m \in \Z} \int_{\W(n \varepsilon, m\varepsilon)} \Big|\sigma(n \varepsilon, s) - \sigma((n-1)\varepsilon, s) \Big| \, ds.
\end{equation}
We spit the grid points in interaction points, cancellation points and no-collision points. We thus adapt Definition \ref{W_int_canc_points} as follows. 
\begin{definition}
Let $(n\varepsilon, m\varepsilon)$ be a grid point, $n \geq 1, m \in \Z$.

\noindent We say that $(n\varepsilon, m\varepsilon)$ is an \emph{interaction point} if $\W^{(1)}((n-1)\varepsilon, (m-1)\varepsilon))$ and $\W^{(0)}((n-1)\varepsilon, m\varepsilon)$ are both nonempty and have the same sign.

\noindent We say that $(n\varepsilon, m\varepsilon)$ is a \emph{cancellation point} if $\W^{(1)}((n-1)\varepsilon, (m-1)\varepsilon))$ and $\W^{(0)}((n-1)\varepsilon, m\varepsilon)$ are both nonempty and have different sign.

\noindent We say that $(n\varepsilon, m\varepsilon)$ is a \emph{no-collision point} if at least one among $\W^{(1)}((n-1)\varepsilon, (m-1)\varepsilon))$ and $\W^{(0)}((n-1)\varepsilon, m\varepsilon)$ is empty. 
\end{definition}

Accordingly, we split \eqref{da tempo a punti} into three terms:
\begin{align*}
\sum_{n=1}^{+\infty}& \sum_{m \in \Z} \int_{\W(n \varepsilon, m\varepsilon)} \big| \sigma(n \varepsilon, s) - \sigma((n-1)\varepsilon, s) \big| \, ds \\
=&~ 
\bigg( \sum_{\substack{(n\e,m\e) \\ \text{interaction}}} + \sum_{\substack{(n\e,m\e) \\ \text{cancellation}}} + \sum_{\substack{(n\e,m\e) \\ \text{no-collision}}} \bigg) \int_{\W(n \varepsilon, m\varepsilon)} \big| \sigma(n \varepsilon, s) - \sigma((n-1)\varepsilon, s) \big| \, ds. 
\end{align*}
The three terms are estimated in the next three propositions.

\begin{proposition}
\label{P_no_interaction}
Let $(n\varepsilon, m\varepsilon)$ be a no-collision point. Then
\[
\int_{\W(n \varepsilon, m\varepsilon)} \big| \sigma(n \varepsilon, s) - \sigma((n-1)\varepsilon, s) \big| \, ds = 0.
\]
\end{proposition}

\begin{proof}
Assume first $\W^{(1)}((n-1)\varepsilon, (m-1)\varepsilon)$ empty (the case $\W^{(0)}((n-1)\varepsilon, m\varepsilon) = \emptyset$ is similar). See Figure \ref{fig:figura33}. In this case
\[
u_{n,m-1} = u_{n-1,m-1}.
\]
Moreover, due to the restarting procedure of the Glimm scheme,
\begin{equation}
\label{f_=_inv}
f(u_{n,m}) = \frac{d}{du}\conv_{[u_{n-1,m-1},u_{n-1,m}]}f(u_{n,m}).
\end{equation}
Hence, according to Definition \ref{speed_function}, and by Proposition \ref{tocca}, for each $s \in \W(n \varepsilon, m \varepsilon) = \W^{(0)}((n-1)\varepsilon, m \varepsilon)$,
\begin{align*}
\sigma(n\varepsilon, s)                            =&~  \frac{d}{du}\conv_{[u_{n,m-1},u_{n,m}]}f(\hat u(s)) \\
\text{(by Proposition \ref{tocca} and \eqref{f_=_inv})}  =&~  \frac{d}{du}\conv_{[u_{n-1,m-1},u_{n-1,m}]}f(\hat u (s)) \\
                                                   =&~  \sigma((n-1)\e, s), 
\end{align*}
concluding the proof.
\end{proof}

\begin{figure}
  \begin{center}
    \includegraphics[height=5cm,width=10cm]{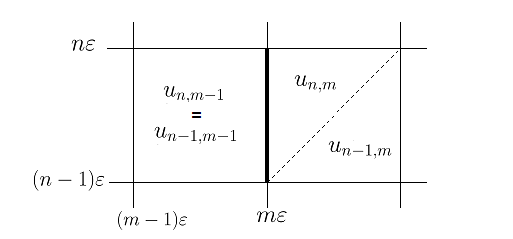}
    \caption{No-collision case when $\W^{(1)}((n-1)\e,(m-1)\e) = \emptyset$.}
    \label{fig:figura33}
    \end{center}
\end{figure}

For each grid point $(n\e,m\e)$, the amount of cancellation is defined by:
\begin{align*}
\mathcal{C}(n\e,m\e) :=&~ |u_{n,m-1} - u_{n-1,m-1}| + |u_{n-1,m-1} - u_{n,m}| - |u_{n,m-1} - u_{n,m}| \\
                     =&~ \Big| \W^{(1)}((n-1)\e, (m-1)\e) \cup \W^{(0)}((n-1)\e,m\e)\Big| - \big| \W(n\e,m\e) \big|.
\end{align*}
It is easy to see that for fixed $n \geq 1$,
\begin{equation*}
\sum_{m \in \Z} \mathcal{C}(n\e,m\e) = \TV(u_\e((n-1)\e, \cdot)) - \TV(u_\e(n\e, \cdot)).
\end{equation*}

\begin{proposition}
\label{canc_3}
Let $(n\varepsilon, m\varepsilon)$ be a cancellation point. Then
\[
\int_{\W(n \varepsilon, m\varepsilon)} \big| \sigma(n \varepsilon, s) - \sigma((n-1)\varepsilon, s) \big| \, ds \leq \lVert f'' \lVert_{L^\infty} \TV(u(0,\cdot)) \C(n\e, m\e).
\]
\end{proposition}

\noindent The proof is completely similar to Proposition \ref{W_canc_3}.

Using previous proposition one gets the same estimate as in Corollary \ref{W_canc_4}.

\begin{corollary}
It holds
\[
\sum_{\substack{(n\e,m\e) \\ \mathrm{cancellation}}} \int_{\W(n \varepsilon, m\varepsilon)} \big| \sigma(n \varepsilon, s) - \sigma((n-1)\varepsilon, s) \big| \, ds \leq \lVert f'' \lVert_{L^\infty} \TV(u(0,\cdot))^2.
\]
\end{corollary}

The last estimate to prove is 
\begin{equation*}
\sum_{\substack{(n\e,m\e) \\ \text{interaction}}} \int_{\W(n\e, m\e)} \big| \sigma(n\e,s) - \sigma((n-1)\e,s) \big|\, ds \leq \const \lVert f'' \lVert_{L^\infty}  \TV(u(0,\cdot))^2.
\end{equation*}

As in the wavefront case, in the subsequent sections we will define a functional $\fQ: \N\e \to [0, +\infty)$ of the form
\[
\fQ = \fQ(n\e) = \iint_{\mathcal{D}(n\e)} \mathfrak{q}(n\e, s,s') \, ds ds',
\]
where 
\begin{equation}
\label{domain_of_integration}
\mathcal{D}(n\e) := \Big \{(s,s') \in \W(n\e) \times \W(n\e) \ \Big| \ s < s'\Big \}
\end{equation}
is \emph{the domain of integration} of the functional $\fQ$ at time $n\e$ and
\[
\mathfrak{q}: \N\e \times \mathcal{D}(n\e) \to \big[ 0,\lVert f'' \lVert_{L^\infty} \big]
\]
is a measurable function on $\mathcal{D}(n\e)$ for each $n \in \N$, called \emph{the weight of the pair of waves $(s,s')$ at time $n\e$}.
Such a functional will have two properties: for each time $n\e$
\begin{enumerate}[(a)]
\item there is a measurable subset $\mathcal{D}^I(n\e) \subseteq \mathcal{D}(n\e)$ such that
\begin{equation}
\label{decrease}
\begin{split}
\sum_{\substack{m \in \Z \text{ such that}\\ (n\e,m\e) \text{ is} \\ \text{interaction point}}} \int_{\W(n\e, m\e)}& \big| \sigma(n\e,s) - \sigma((n-1)\e,s) \big|\, ds \\ 
\leq&~ 2 \bigg[ \iint_{\mathcal{D}^I(n\e)} \mathfrak{q}((n-1)\e, s,s')\, dsds' - \iint_{\mathcal{D}^I(n\e)} \mathfrak{q}(n\e, s,s')\, dsds' \bigg],
\end{split}
\end{equation}
Theorem \ref{decrease_thm_enunciato} and Corollary \ref{decrease_thm_corollary};
\item on the complement $\mathcal{D}(n\e) \setminus \mathcal{D}^I(n\e)$,
\begin{equation}
\label{increase}
\begin{split}
\iint_{\mathcal{D}(n\e) \setminus \mathcal{D}^I(n\e)}& \mathfrak{q}(n\e, s,s')\, dsds' - 
\iint_{\mathcal{D}(n\e) \setminus \mathcal{D}^I(n\e)} \mathfrak{q}((n-1)\e, s,s')\, dsds'\\
&\leq \log(2) \lVert f'' \lVert_{L^\infty}  \TV(u(0,\cdot)) \Big[\TV(u_\e((n-1)\e,\cdot) - \TV(u_\e(n\e, \cdot)) \Big],
\end{split}
\end{equation}
Theorem \ref{increase_thm_enunciato} and Corollary \ref{increase_thm_corollary}.
\end{enumerate}

These two properties are the analogs of \eqref{W_decrease} and \eqref{W_increase}. As in Section \ref{section_W_main_thm}, \eqref{decrease} and \eqref{increase} implies the following proposition, which is the analog of Proposition \ref{W_thm_interaction}. 

\begin{proposition}
It holds
\begin{equation*}
\sum_{\substack{(n\e,m\e) \\ \mathrm{interaction}}} \int_{\W(n\e, m\e)} \big| \sigma(n\e,s) - \sigma((n-1)\e,s) \big|\, ds  \leq
2 (1 + \log(2)) \lVert f'' \lVert_{L^\infty}  \TV(u(0,\cdot))^2.
\end{equation*}
\end{proposition}

\begin{proof}
For fixed time $n\e$,
\begin{align*}
\sum_{\substack{m \in \Z \\ \text{ such that} (n\e,m\e) \text{ is}\\ \text{interaction point}}}& \int_{\W(n\e, m\e)} \big| \sigma(n\e,s) - \sigma((n-1)\e,s) \big|\, ds \\ 
\text{(using \eqref{decrease})} 
\leq&~ 2 \bigg [\iint\limits_{\mathcal{D}^I(n\e)} \mathfrak{q}((n-1)\e, s,s')\, dsds' - \iint\limits_{\mathcal{D}^I(n\e)} \mathfrak{q}(n\e, s,s')\, dsds' \bigg ] \\
=&~ 2 \bigg [\iint\limits_{\mathcal{D}^I(n\e)} \mathfrak{q}((n-1)\e, s,s')\, dsds' - \iint\limits_{\mathcal{D}^I(n\e)} \mathfrak{q}(n\e, s,s')\, dsds' \\
 &~ \qquad + \iint\limits_{\mathcal{D}(n\e) \setminus \mathcal{D}^I(n\e)} \mathfrak{q}((n-1)\e, s,s')\, dsds' - \iint\limits_{\mathcal{D}(n\e) \setminus \mathcal{D}^I(n\e)} \mathfrak{q}((n-1)\e, s,s')\, dsds'\\ 
 &~ \qquad + \iint\limits_{\mathcal{D}(n\e) \setminus \mathcal{D}^I(n\e)} \mathfrak{q}(n\e, s,s')\, dsds' - \iint\limits_{\mathcal{D}(n\e) \setminus \mathcal{D}^I(n\e)} \mathfrak{q}(n\e, s,s')\, dsds' \bigg ]\\
=&~ 2 \bigg[ \iint\limits_{\mathcal{D}(n\e)} \mathfrak{q}((n-1)\e, s,s')\, dsds'  -  \iint\limits_{\mathcal{D}(n\e)} \mathfrak{q}(n\e, s,s')\, dsds'\bigg] \\
 &~ + 2 \bigg[ \iint\limits_{\mathcal{D}(n\e) \setminus \mathcal{D}^I(n\e)} \mathfrak{q}(n\e, s,s')\, dsds'  - \iint\limits_{\mathcal{D}(n\e) \setminus \mathcal{D}^I(n\e)} \mathfrak{q}((n-1)\e, s,s')\, dsds' \bigg] \\
\leq&~ 2 \big[ \fQ((n-1)\e) - \fQ(n\e) \big]\\
&~ + 2 \bigg[ \iint\limits_{\mathcal{D}(n\e) \setminus \mathcal{D}^I(n\e)} \mathfrak{q}(n\e, s,s')\, dsds' - \iint\limits_{\mathcal{D}(n\e) \setminus \mathcal{D}^I(n\e)} \mathfrak{q}((n-1)\e, s,s')\, dsds'\bigg] \\
\text{(using \eqref{increase})} \leq & 2 \big[ \fQ((n-1)\e) - \fQ(n\e) \big] \\
&~ + 2 \log(2) \lVert f'' \lVert_{L^\infty} \TV(u(0,\cdot)) \Big[\TV(u_\e((n-1)\e)) - \TV(u_\e(n\e))\Big]. 
\end{align*}
Hence, exactly as in proof of Proposition \ref{W_thm_interaction}, 
\begin{align*}
\sum_{n = 1}^N& \sum_{\substack{m \in \Z \text{ s.t.} \\ (n\e,m\e) \text{ is}\\ \text{interaction point}}} \int_{\W(n\e, m\e)} \big| \sigma(n\e,s) - \sigma((n-1)\e,s) \big|\, ds \\
&\leq 2 \sum_{n = 1}^N \Big[\fQ((n-1)\e) - \fQ(n\e)\Big] \\
     & \quad  + 2 \log(2) \sum_{n=1}^N \lVert f'' \lVert_{L^\infty} \TV(u(0,\cdot)) \Big[\TV(u_\e((n-1)\e)) - \TV(u_\e(n\e))\Big] \\
&= 2 \Big[\fQ(0) - \fQ(N\e) \Big] + 2 \log(2) \lVert f'' \lVert_{L^\infty}  \TV(u(0,\cdot)) \Big[ \TV(u_\e(0,\cdot)) - \TV(u_\e(N\e)) \Big] \\
&\leq 2 \Big [\fQ(0) + \log(2) \lVert f'' \lVert_{L^\infty}  \TV(u(0,\cdot))^2 \Big ]\\
&\leq 2 (1 + \log(2)) \lVert f'' \lVert_{L^\infty}  \TV(u(0,\cdot))^2,
\end{align*}
where the last inequality is justified by the fact that
\begin{equation*}
\fQ(0) = \iint_{\mathcal{D}(0)} \mathfrak{q}(0,s,s') \, ds ds' \leq \lVert f'' \lVert_{L^\infty} |\mathcal{D}(0)| \leq \lVert f'' \lVert_{L^\infty}  \TV(u(0,\cdot))^2.
\end{equation*}
Passing to the limit as $N \rightarrow +\infty$, one completes the proof. 
\end{proof}

\subsection{Waves collisions in the Glimm scheme case}

This section introduces the notion of pairs of waves which have \emph{already interacted/not yet interacted}. Even if the definitions and the propositions are completely similar to the ones in Section \ref{W_waves_collision}, some technicalities arise, for example in the proof of Proposition \ref{divise_tocca}, which is substantially longer than the proof of the analogous Proposition \ref{W_divise_tocca}. The following definition is the same as Definition \ref{interagite_non_interagite}.

\begin{definition}
\label{D_Glimm_interacted}
Let $\bar t$ be a fixed time and let $s,s' \in \W(\bar t)$. We say that \emph{$s,s'$ interact at time $\bar t$} if $\mathtt x(\bar t, s) = \mathtt x(\bar t, s')$.  

We also say that \emph{they have already interacted at time $\bar t$} if there is $t \leq \bar t$ such that $s,s'$ interact at time $t$. Moreover we say that \emph{they have not yet interacted at time $\bar t$} if for any $t \leq \bar t$, they do not interact at time $t$. 
\end{definition}

\noindent Lemma \ref{W_interagite_stesso_segno} and Lemma \ref{W_quelle_in_mezzo_hanno_int} hold also in this case, namely:
\begin{enumerate}
 \item if $s, s'$ interact at time $\bar t$, then they have the same sign (Lemma \ref{W_interagite_stesso_segno});
 \item let $\bar t$ be a fixed time, $s,s' \in \W(\bar t)$, $s < s'$; assume that $s, s'$ have already interacted at time $\bar t$; if $p, p' \in [s, s'] \cap \W(\bar t)$, then $p, p'$ have already interacted at time $\bar t$ (Lemma \ref{W_quelle_in_mezzo_hanno_int}).
\end{enumerate}

\noindent As in Section \ref{W_waves_collision}, for any fixed time $\bar t \geq 0$ and for any $\bar s \in \W(\bar t)$, define 
\[
\mathcal{I}(\bar t, \bar s) := \Big\{s \in \W(\bar t) \ \Big| \ s \text{ has already interacted with } \bar s \text{ at time } \bar t\Big\};
\]
by Lemmas \ref{W_interagite_stesso_segno} and \ref{W_quelle_in_mezzo_hanno_int}, this is an homogeneous interval of waves.

If $s < s'$ have already interacted at a fixed time $\bar t$, set as before 
\[
\mathcal{I}(\bar t, s, s') : = \mathcal{I}(\bar t, s) \cap \mathcal{I}(\bar t, s').
\]
This is clearly an interval of waves and so by Proposition \ref{interval_waves} the image of 
\[
\hat u: \mathcal{I}(\bar t, s, s') \to \R
\]
is an interval in $\R$.

The following definition is the same as Definition \ref{W_waves_divided}.

\begin{definition}
\label{waves_divided}
Let $s,s' \in \W(\bar t)$ be two waves which have already interacted at time $\bar t$. We say that \emph{$s,s'$ are divided in the real solution at time $\bar t$} if 
\[
(\mathtt x(\bar t, s), \sigma(\bar t, s)) \neq (\mathtt x(\bar t, s'), \sigma(\bar t, s')),
\]
i.e. if at time $\bar t$ they have either different position, or the same position, but different speed.

\noindent If they are not divided in the real solution, we say that \emph{they are joined in the real solution}.
\end{definition}

\begin{remark}
As noted in Remark \ref{rem_divise_solo_in_cancellazioni}, in the wavefront tracking algorithm two waves can have same position but different speed only in cancellation points; on the contrary, in the Glimm scheme two waves can have same position but different speed at every time step.
\end{remark}

The analog of Proposition \ref{W_unite_realta} is the following proposition. We omit the proof.

\begin{proposition}
\label{unite_realta}
Let $\bar t = n\e, n \in \N$ be a fixed time. Let $s,s' \in \W(n\e)$. If $s,s'$ are not divided in the real solution at time $n\e$, then the Riemann problem $\mathcal{I}(n\e,s,s')$ does not divide them. 
\end{proposition}

Also Proposition \ref{W_divise_tocca} holds in this framework, but the proof is slightly different and more technical.

\begin{proposition}
\label{divise_tocca}
Let $\bar t = n\e, n \in \N$ be a fixed time. Let $s,s'$ be two waves which have already interacted at time $\bar t$. Assume that $s,s'$ are divided in the real solution, and let $p,p' \in \mathcal{I}(\bar t, s,s')$. If $p,p'$ are divided in the real solution at time $\bar t$, then the Riemann problem $\mathcal{I}(\bar t, s,s')$ divides them.
\end{proposition}

\begin{proof}
The proof is by induction on the time step $n$. For $n = 0$, the statement is obvious. Let us assume the proposition is true for $t = n\e$ and let us prove it for $t = (n+1)\e$. Let $s,s' \in \W((n+1)\e)$ be two waves which have already interacted at time $(n+1)\e$ and assume $s,s'$ to be divided in the real solution at time $(n+1)\e$. We can also assume w.l.o.g. that $s,s'$ are both positive.

Let us define
\begin{align}
m_1 \e := &~ \min \Big\{\lim_{t \nearrow (n+1)\e} \mathtt x(t,p) \ \Big| \ p \in \mathcal{I}(n\e,s,s') \Big\}, \label{m_1} \\
m_2 \e := &~ \max \Big\{\lim_{t \nearrow (n+1)\e} \mathtt x(t,p) \ \Big| \ p \in \mathcal{I}(n\e,s,s') \Big\} - \e. \label{m_2} 
\end{align}
It is easy to see that $m_1, m_2$ exist and $m_1 \leq m_2 +1$. Assume first $m_1 = m_2 +1$, which means that the $\min$ and the $\max$ above coincide. See Figure \ref{fig:figura34}.

\begin{figure}
  \begin{center}
    \includegraphics[height=6cm,width=11cm]{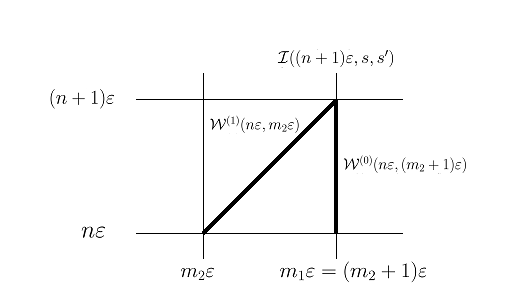}
    \caption{Case $m_1 = m_2 +1$.}
    \label{fig:figura34}
    \end{center}
\end{figure}

\noindent In this case 
\[
\mathcal{I}(n\e,s,s') \subseteq \W^{(1)}(n\e, m_2 \e) \cup \W^{(0)}(n\e, (m_2+1) \e).
\]
This implies $\mathcal{I}((n+1)\e,s,s') = \W(n\e, (m_2+1)\e)$ and so the thesis is easily proved, because the artificial Riemann problem coincides with the real one.

We can thus assume $m_1 < m_2+1$, i.e. $m_1 \leq m_2$. See Figure \ref{fig:figura35}.  Under this assumption, we first write down some useful claims. 

\begin{figure}
  \begin{center}
    \includegraphics[height=7cm,width=13cm]{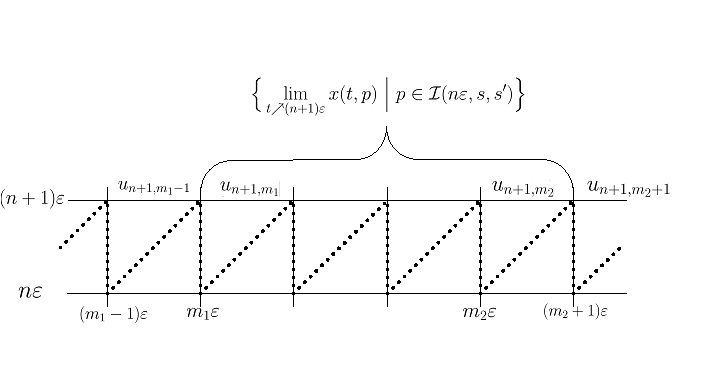}
    \caption{Case $m_1 < m_2 +1$.}
    \label{fig:figura35}
    \end{center}
\end{figure}

\begin{claim}
\label{bounds_on_I}
It holds
\[
\mathcal{I}(n\e,s,s') \subseteq \W^{(1)}(n\e, (m_1-1) \e) \cup \bigg[ \bigcup_{m_1 \leq m \leq m_2}\W(n\e, m \e)  \cup \W^{(0)}(n\e, (m_2+1) \e) \bigg].
\]
\end{claim}

\begin{proof}[Proof of Claim \ref{bounds_on_I}]
Let $p \in \mathcal{I}(n\e,s,s')$. By \eqref{m_1} and \eqref{m_2}, $\lim_{t \nearrow (n+1)\e} \mathtt x(t,p) = m\e$, for some $m_1 \leq m \leq m_2 +1$. Hence, $p \in \W^{(1)}(n\e, (m_1-1) \e) \cup \bigcup_{m_1 \leq m \leq m_2}\W(n\e, m \e)  \cup \W^{(0)}(n\e, (m_2+1) \e)$. 
\end{proof}

\begin{claim}
\label{segno}
Each wave in $\displaystyle{\bigcup_{m_1 \leq m \leq m_2}} \W(n\e, m \e)$ has positive sign.
\end{claim}

\begin{proof}[Proof of Claim \ref{segno}]
Use Definition \ref{eow} of enumeration of waves and Lemma \ref{W_interagite_stesso_segno}.
\end{proof}

\begin{claim}
\label{divise_dopo}
Let $p,p' \in \mathcal{I}((n+1)\e, s, s') \cap \mathcal{I}(n\e,s,s')$. Assume $p,p'$ are divided at time $t = (n+1)\e$, but not divided at time $t = n\e$. Then 
\begin{enumerate}[a)]
 \item either $p,p' \in \W^{(0)}(n\e, m_1\e)$ and waves in $\W^{(1)}(n\e, (m_1-1)\e)$ are negative;
 \item or $p,p' \in \W^{(1)}(n\e, m_2\e)$ and waves in $\W^{(0)}(n\e, (m_2+1)\e)$ are negative. 
\end{enumerate}
\end{claim}

\begin{proof}[Proof of Claim \ref{divise_dopo}]
Since $p,p'$ are not divided at time $t = n\e$, there is $\bar m \in \{m_1-1, \dots, m_2+1\}$ such that $p,p' \in \W^{(\alpha)}(n\e, \bar m \e)$, for some $\alpha \in \{0,1\}$. 
Assume $\alpha = 0$, the other case is completely similar. If $\W^{(1)}(n\e, (\bar m -1)\e) = \emptyset$ or $\mathcal{S} (\W^{(1)}(n\e, (\bar m -1)\e)) = +1$, this means that $((n+1)\e, \bar m \e)$ is either a no-collision point or an interaction point. By Proposition \ref{differenza_vel}, since $p,p'$ are not divided at time $n\e$, they cannot be divided at time $(n+1)\e$. Hence $\W^{(1)}(n\e, (\bar m -1)\e)) \neq \emptyset$ and $\mathcal{S} (\W^{(1)}(n\e, (\bar m -1)\e)) = -1$. Thus by Claim \ref{segno}, $\bar m -1 \notin [m_1, m_2]$, i.e. $\bar m \notin [m_1 +1, m_2 +1]$; on the other hand, by Claim \ref{bounds_on_I}, $\bar m \in [m_1, m_2+1]$. Hence $\bar m = m_1$ and the thesis is proved.
\end{proof}

Now for $j = n,n+1$, set
\begin{align*}
L(j) := & \inf \hat u(\mathcal{I}(j\e,s,s')), \qquad R(j) := \sup \hat u(\mathcal{I}(j\e,s,s')).
\end{align*}

\begin{claim}
\label{interazione}
The following holds:
\begin{enumerate}
 \item If $R(n+1) > R(n)$, then $s,s' \in \W^{(1)}(n\e, m_2 \e) \cup \W^{(0)}(n\e, (m_2+1) \e)$ and $R(n+1) = u_{n+1, m_2+1}$;  
 \item If $L(n+1) < L(n)$, then $s,s' \in \W^{(1)}(n\e, (m_1-1) \e) \cup \W^{(0)}(n\e, m_1 \e)$ and $L(n+1) = u_{n+1, m_1-1}$.
\end{enumerate}
\end{claim}

\begin{proof}[Proof of Claim \ref{interazione}]
Immediate from definition of $\mathcal{I}((n+1)\e,s,s')$.
\end{proof}

The following lemma provides some bounds on $L(j), R(j)$.

\begin{lemma}
\label{bdd}
The following hold:
\begin{enumerate}[a)]
 \item \label{Point_a_bdd} $L(n) \leq u_{n+1,m_1} \leq u_{n+1,m_2} \leq R(n)$;
 \item \label{Point_b_bdd} $\min\{u_{n+1,m_1}, u_{n+1,m_1-1}\}  \leq L(n+1) \leq u_{n+1, m_1}$;
 \item \label{Point_c_bdd} $u_{n+1, m_2} \leq R(n+1) \leq \max\{u_{n+1, m_2}, u_{n+1, m_2+1}\}$.
\end{enumerate}
\end{lemma}

\begin{proof}[Proof of Lemma \ref{bdd}]
We prove each point separately.

\smallskip
{\it Proof of Point \ref{Point_a_bdd}).} Define 
\begin{align*}
A :=&~ \Big\{p \in \mathcal{I}(n\e,s,s') \ \Big| \ \lim_{t \nearrow (n+1)\e}\mathtt x(t,p) = m_1 \e \Big\}, \\
B :=&~ \Big\{p \in \mathcal{I}(n\e,s,s') \ \Big| \ \lim_{t \nearrow (n+1)\e}\mathtt x(t,p) = (m_2+1) \e \Big\}.
\end{align*}
Clearly $A,B \neq \emptyset$ and $A,B \subseteq \mathcal{I}(n\e,s,s')$. Using the definition of $m_1,m_2, A,B$ and the fact that $\mathcal{I}(n\e,s,s')$ is an interval of waves, it is not difficult to see that 
\begin{align*}
\inf \hat u(A) = & \ \inf \hat u (\mathcal{I}(n\e,s,s')), & \inf \hat u(B) = & \ u_{n+1, m_2}, \\
\sup \hat u(A) = & \ u_{n+1, m_1},                        & \sup \hat u(B) = & \ \sup \hat u (\mathcal{I}(n\e,s,s')). 
\end{align*}

Hence, $L(n) = \inf \hat u(A) \leq \sup \hat u(A) = u_{n+1,m_1} \leq u_{n+1,m_2} = \inf \hat u(B) \leq \sup \hat u(B) = R(n)$.

\smallskip
{\it Proof of Point \ref{Point_b_bdd}).} If either $((n+1)\e, m_1 \e)$ is a no-collision point, or it is an interaction point, but $L(n+1) = L(n)$, then clearly $u_{n+1,m_1-1} \leq L(n) = L(n+1) \leq u_{n+1, m_1}$ and then the thesis holds; if $((n+1)\e, m_1 \e)$ is an interaction point and $L(n+1) < L(n)$, then it must hold $L(n+1) = u_{n+1,m_1-1} \leq L(n) < u_{n+1,m_1}$, and hence the thesis holds; if $((n+1)\e, m_1 \e)$ is a cancellation point, then either $u_{n+1,m_1-1} < u_{n+1,m_1}$ and so $u_{n+1,m_1-1} \leq L(n+1) \leq u_{n+1,m_1}$, or $u_{n+1,m_1-1} \geq u_{n+1,m_1}$ and so  $u_{n+1,m_1} = L(n+1)$, concluding the proof.

\smallskip
{\it Proof of Point \ref{Point_c_bdd}).} Same proof as the previous point. 
\end{proof}

\begin{lemma}
\label{ultimo}
Let $\tilde s, \tilde s' \in \mathcal{I}(n\e,s,s')$ two waves which have already interacted at time $t = n\e$, but are divided at time $n\e$. Take any $u_{j,m}$, with $j \in \{n,n+1\}, m \in \Z$, $m_1 \leq m \leq m_2$. If 
\[
\inf \hat u(\mathcal{I}(n\e,\tilde s, \tilde s')) \leq u_{j,m} \leq \sup \hat u(\mathcal{I}(n\e,\tilde s, \tilde s')),
\] 
then 
\begin{equation*}
\conv_{[\inf \hat u(\mathcal{I}(n\e,\tilde s, \tilde s')), \sup \hat u(\mathcal{I}(n\e,\tilde s, \tilde s'))]} f
= \conv_{[\inf \hat u(\mathcal{I}(n\e,\tilde s, \tilde s')), u_{j,m}]} f \cup 
\conv_{[u_{j,m}, \sup \hat u(\mathcal{I}(n\e,\tilde s, \tilde s'))]} f.
\end{equation*}
\end{lemma}

\begin{proof}[Proof of Lemma \ref{ultimo}]
We can assume that both open intervals
\[
\big( \inf \hat u(\mathcal{I}(n\e,\tilde s, \tilde s')), u_{j,m} \big) \quad \text{and} \quad \big( u_{j,m}, \sup \hat u(\mathcal{I}(n\e,\tilde s, \tilde s')) \big)
\]
are non-empty (otherwise the proof is trivial). Take a sequence $\{p_k\}_k \subseteq \mathcal{I}(n\e,\tilde s,\tilde s')$ such that $\hat u(p_k) < u_{j,m}$ and $\hat u(p_k) \nearrow u_{j,m}$ and a sequence $\{p'_k\}_k \subseteq \mathcal{I}(n\e,\tilde s, \tilde s')$ such that $\hat u(p'_k) > u_{j,m}$ and $\hat u(p'_k) \searrow u_{j,m}$. Since $\hat u(p_k) < u_{j,m} < \hat u(p'_k)$, then $p_k, p'_k$ are divided at time $t = n\e$ and then, since by inductive assumption the statement of the proposition holds at time $n\e$, there is $u_k \in (\hat u(p_k), \hat u(p'_k))$ such that 
\[
\conv_{[\inf \hat u(\mathcal{I}(n\e,\tilde s, \tilde s')), \sup \hat u(\mathcal{I}(n\e,\tilde s, \tilde s'))]}f(u_k) = f(u_k).
\]
Passing to the limit as $k \rightarrow +\infty$ and using Proposition \ref{tocca}, we get the thesis. 
\end{proof}

We now construct two waves $\tilde s, \tilde s'$ with the following two properties:
\begin{enumerate}[a)]
 \item \label{a} $\tilde s, \tilde s'$ are divided at time $n\e$,
 \item \label{b} $\tilde L(n) := \inf \hat u(\mathcal{I}(n\e,\tilde s, \tilde s')) \leq u_{n+1,m_1} \leq u_{n+1,m_2} \leq \sup \hat u(\mathcal{I}(n\e,\tilde s, \tilde s')) =: \tilde R(n)$. 
\end{enumerate}
Notice that we do not require $\tilde s, \tilde s'$ to exist at time $(n+1)\e$. The procedure is:
\begin{enumerate}
\item if at least one of the states $\hat u(s), \hat u(s')$ belongs to $(u_{n+1,m_1},u_{n+1,m_2}]$, then set $\tilde s := s, \tilde s':=s'$. By Claim \ref{divise_dopo}, Property \ref{a}) holds. By Lemma \ref{bdd}, also Property \ref{b}) holds;
\item similarly, if $\hat u(s) \in [L(n), u_{n+1,m_1}]$ and $\hat u(s') \in (u_{n+1,m_2}, R(n)]$, set $\tilde s := s, \tilde s':=s'$;
\item if $\hat u(s), \hat u(s') \in (u_{n+1,m_2}, R(n)]$, then set $\tilde s':= s'$, and take as $\tilde s$ any wave such that 
$$
\lim_{t \nearrow (n+1)\e} \mathtt x(t, \tilde s) = m_1\e.
$$ 
By definition of $m_1$, and the assumption that $m_1 \leq m_2$, Property \ref{a}) holds; moreover observe that in this case $\tilde L(n) = L(n)$, while $\hat u(s') \leq \tilde R(n)$ and thus also Property \ref{b}) holds;
\item if $\hat u(s), \hat u(s') \in [L(n), u_{n+1,m_1}]$, then set $\tilde s:= s$, and take as $\tilde s'$ any wave such that 
$$
\lim_{t \nearrow (n+1)\e} \mathtt x(t, \tilde s') = m_2\e.
$$
As in the previous point, one can show that Properties \ref{a}) and \ref{b}) holds.
\end{enumerate}

Now, by \ref{a}), \ref{b}) and Lemma \ref{ultimo} one gets
\begin{equation}
\label{s_tilde_sep}
\conv_{[\tilde L(n), \tilde R(n)]}f = \conv_{[\tilde L(n), u_{n+1,m_1}]} f \cup \conv_{[u_{n+1,m_1},u_{n+1,m_2}]} f \cup \conv_{[u_{n+1,m_2}, \tilde R(n)]} f.
\end{equation}

\begin{lemma}
\label{lemma_riemann_dopo_sep}
It holds 
\begin{equation}
\label{riemann_dopo_sep}
\conv_{[L(n+1), R(n+1)]} f = \conv_{I_1} f \cup \conv_{I_2} f \cup \conv_{I_3} f,
\end{equation}
where 
\begin{align*}
I_1 := [L(n+1), u_{n+1,m_1}], \qquad
I_2 := [u_{n+1,m_1},u_{n+1,m_2}], \qquad
I_3 := [u_{n+1,m_2}, R(n+1)].
\end{align*}
\end{lemma}

\begin{proof}
We consider four cases.

\smallskip
\noindent {\it 1.} If at least one among $\hat u(s), \hat u(s')$ belongs to $(u_{n+1,m_1},u_{n+1,m_2}]$, then by our definition $\tilde s = s, \tilde s'=s'$. Moreover, by Claim \ref{interazione}, $L(n+1) \geq L(n) = \tilde L(n)$, $R(n+1) \leq R(n) = \tilde R(n)$; hence, by \eqref{s_tilde_sep} and Corollary \ref{stesso_shock}, we get the thesis.

\smallskip
\noindent {\it 2.} If $\hat u(s) \in [L(n), u_{n+1,m_1}]$ and $\hat u(s') \in (u_{n+1,m_2}, R(n)]$, argue as in previous case.

\smallskip
\noindent {\it 3.} If $\hat u(s), \hat u(s') \in (u_{n+1,m_2}, R(n)]$, then by our definition $\tilde s':= s'$, and $\tilde s$ is any wave such that $\lim_{t \rightarrow (n+1)\e} \mathtt x(t, \tilde s) = m_1\e$. Observe that in this case the Riemann problem $[u_{n+1,m_2}, u_{n+1, m_2+1}]$ is not solved by a single wavefront (since $s,s'$ at time $(n+1)\e$ are divided). Thus, using this fact, \eqref{s_tilde_sep} and Proposition \ref{incastro}, we get
\[
\conv_{[\tilde L(n), u_{n+1, m_2+1}]} f = \conv_{[\tilde L(n), u_{n+1,m_1}]} f \cup \conv_{I_2} f \cup \conv_{[u_{n+1,m_2}, u_{n+1, m_2+1}]} f.
\]
Since $\tilde L(n) = L(n) \leq L(n+1) \leq u_{n+1, m_1}$ and $u_{n+1, m_2} \leq R(n+1) = u_{n+1, m_2+1}$, using Corollary \ref{stesso_shock}, we get the thesis.

\smallskip
\noindent {\it 4.} The case $\hat u(s), \hat u(s') \in [L(n), u_{n+1,m_1}]$ is similar to previous point. 
\end{proof}

\noindent {\it Conclusion of the proof of Proposition \ref{divise_tocca}.} Take $p,p' \in \mathcal{I}((n+1)\e,s,s')$, divided at time $(n+1)\e$. We have to prove that the Riemann problem $[L(n+1), R(n+1)]$ divides them. We can assume $p,p' \in I_j$ for some $j=1,2,3$, where $I_j$ are the intervals defined in Lemma \ref{lemma_riemann_dopo_sep} (otherwise the proof is trivial). 

If $p,p' \in I_2 = [u_{n+1,m_1},u_{n+1,m_2}]$, then by Claim \ref{divise_dopo} $p,p'$ are divided at time $n\e$ and so by inductive assumption the Riemann problem $[\tilde L(n), \tilde R(n)]$ divides them. By \eqref{s_tilde_sep}, the Riemann problem $I_2 = [u_{n+1,m_1}, u_{n+1,m_2}]$ divides them and so by \eqref{riemann_dopo_sep}, also the Riemann problem $[L(n+1), R(n+1)]$ divides $p,p'$, which is what we wanted to prove.

Assume now $p, p' \in I_3 = [u_{n+1,m_2}, R(n+1)]$, the case $p,p' \in I_1$ being  similar. We know $p,p'$ are divided at time $(n+1)\e$. This means that the Riemann problem $[u_{n+1,m_2}, u_{n+1, m_2+1}]$ divides them. By Lemma \ref{bdd}, $R(n+1) \leq u_{n+1, m_2+1}$ and so also the Riemann problem $I_3 = [u_{n+1,m_2}, R(n+1)]$ divides $p,p'$. Hence, by \eqref{riemann_dopo_sep}, the Riemann problem $[L(n+1), R(n+1)]$ divides $p,p'$, which is what we wanted to prove.
\end{proof}

\subsection{\texorpdfstring{The functional $\fQ$ for the Glimm scheme}{The functional for the Glimm scheme}}
\label{Ss_glimm_funct_Q}

In this last section we consider the same functional $\fQ$ defined in Section \ref{W_functional_Q} for the wavefront tracking algorithm, adapted to the Glimm scheme and we prove inequalities \eqref{decrease} and \eqref{increase}. Differently from the wavefront tracking, where $\fQ$ is defined as a finite sum of weights, here $\fQ$ is defined as an integral and for this reason we have also to prove that it is well defined. 

\subsubsection{Definition of $\fQ$}

Recall from \eqref{domain_of_integration} the definition of the domain of integration:
\[
\mathcal{D}(n\e) := \Big\{(s,s') \in \W(n\e) \times \W(n\e) \ \Big| \ s < s'\Big\}.
\]
Next define the \emph{weight of the pair of waves $(s,s')$ at fixed time $n\e$} as
\[
\mathfrak{q}(n\e, s, s') :=
\begin{cases}
\dfrac{|\sigma(\mathcal{I}(n\e,s,s'), s') - \sigma(\mathcal{I}(n\e,s,s'), s)|}{|\hat u(s') - \hat u(s)|} 
& s, s' \text{ already interacted at time } n\e,
\\
\lVert f'' \lVert_{L^\infty} & \text{otherwise.}
\end{cases}
\]
As an easy consequence of Theorem \ref{convex_fundamental_thm}, Point (\ref{convex_fundamental_thm_3}), we obtain that $\mathfrak{q}$ takes values in $[0, \lVert f'' \lVert_{L^\infty}]$. Finally set
\[
\mathfrak{Q}(n\e) := \iint_{\mathcal{D}(n\e)} \mathfrak{q}(n\e, s, s') \, dsds'.
\]

First of all we have to prove that $\fQ$ is well defined, i.e. $\mathfrak q$ is integrable. 
Actually in the following proposition we prove an additional regularity property on $\W(n\e)$.  

\begin{proposition}
\label{W_partitioned}
It is possible to write $\W(n\e)$ as an (at most) countable union of mutually disjoint interval of waves $E_k$, 
\[
\W(n\e) = \bigcup_{k \in \N} E_k
\]
such that for each $k \in \N$ and for each $s,p \in E_k$, $\mathcal{I}(t,s) = \mathcal{I}(t,p)$.
\end{proposition}

Since derivation of convex envelopes is a Borel operation, it follows that $\mathfrak{q}(n\e,\cdot, \cdot)$ is Borel.

\begin{proof}
For each $s \in \W(n\e)$, define
\[
E(n\e,s) := \Big\{p \in \W(n\e) \ \Big| \ \mathcal{I}(n\e,s) = \mathcal{I}(n\e,p)\Big\}.
\]
We claim that each $E(n\e,s)$ is an interval of waves and for fixed $n \in \N$, the cardinality of $\{E(n\e,s) \ | \ s \in \W(n\e)\}$ is at most countable. This is proved using the following lemmas.

\begin{lemma}
\label{L_eh_1}
$E(n\e,s)$ is an interval of waves at time $n\e$. 
\end{lemma}

\begin{proof}[Proof of Lemma \ref{L_eh_1}]
Let $p,p' \in E(n\e,s)$ and let $r$ such that $p<r<p'$. We have to prove that $r \in E(n\e,s)$. We have
\[
\mathcal{I}(n\e,p) = \mathcal{I}(n\e,s) = \mathcal{I}(n\e,p').
\]
By contradiction, assume $r \notin E(n\e,s)$, i.e. $\mathcal{I}(n\e,s) \neq \mathcal{I}(n\e,r)$. Hence, there is 
\[
r' \in \Big (\mathcal{I}(n\e, s) \setminus \mathcal{I}(n\e,r) \Big) \cup \Big (\mathcal{I}(n\e, r) \setminus \mathcal{I}(n\e,s) \Big).
\]
Assume $r' \in \mathcal{I}(n\e, s) \setminus \mathcal{I}(n\e,r)$. Hence $r' \in \mathcal{I}(n\e,p) = \mathcal{I}(n\e,p')$. Thus, whatever the position of $r'$ is, since $p < r < p'$, $r'$ must have already interacted with $r$, a contradiction. 

On the other hand, if $r' \in \mathcal{I}(n\e, r) \setminus \mathcal{I}(n\e,s)$, whatever the position of $r'$ is, $r'$ must have already interacted either with $p$ or with $p'$. Thus $r' \in \mathcal{I}(n\e,p) = \mathcal{I}(n\e,p') = \mathcal{I}(n\e,s)$, a contradiction.
\end{proof}

\begin{lemma}
\label{L_open}
Let $s \in \W(n\e)$. There exist $s_1, s_2 \in \R$ such that 
\begin{enumerate}
 \item $s \in (s_1, s_2]$;
 \item for each $p \in (s_1, s_2]$, $\mathcal{I}(n\e,s) = \mathcal{I}(n\e,p)$.
\end{enumerate}
\end{lemma}

\begin{proof}[Proof of Lemma \ref{L_open}]
The proof is by induction on the time step $n \in \N$. For $n=0$, assume $s \in \W(0,m\e)$; it is sufficient to choose $(s_1, s_2] := \W(0,m\e)$. 

Now assume the lemma holds for $n$ and prove it for $n+1$. Assume 
$s \in \W^{(0)}((n+1)\e, m\e)$ at time $(n+1)\e$ (the case $s \in \W^{(1)}((n+1)\e, m\e)$ is completely similar). Moreover assume $(s_1,s_2]$ to be the interval at time $n\e$ with the Properties 1, 2 of the statement of the Lemma. For time $(n+1)\e$ define 
\[
(\tilde s_1, \tilde s_2] := (s_1, s_2] \cap 
\W^{(0)}((n+1)\e, m\e).
\] 
Clearly Property (1) holds. Moreover, if Property (2) does not hold, there must be a wave $p \in (\tilde s_1, \tilde s_2]$ and another wave $r$ which has interacted neither with $s$ nor with $p$ at time $n\e$, but interacts with only one among $p,s$ at time $(n+1)\e$. This is impossible, since at time $(n+1)\e$, $s,p$ have the same position.
\end{proof}

\begin{lemma}
\label{L_eh_2}
For fixed $n \in \N$, for each $s,p \in \W(n\e)$, either $E(n\e,s) = E(n\e,p)$ or $E(n\e,s) \cap E(n\e,p) = \emptyset$. 
\end{lemma}

\begin{proof}[Proof of Lemma \ref{L_eh_2}]
Assume there is $p' \in E(n\e,s) \cap E(n\e,p)$. By simmetry, it is sufficient to prove that $E(n\e,s) \subseteq E(n\e,p)$. Take $p'' \in E(n\e,s)$. Then $\mathcal{I}(n\e,p'') = \mathcal{I}(n\e, s) = \mathcal{I}(n\e,p') = \mathcal{I}(n\e,p)$. Hence $p'' \in E(n\e,p)$.
\end{proof}

\begin{lemma}
For fixed $n \in \N$, $\{E(s) \ | \ s \in \W(n\e)\}$ is at most countable.
\end{lemma}
\begin{proof}

By Lemma \ref{L_open}, the interior of $E(n\e,s)$ is not empty. The conclusion follows from separability of $\R$. 
\end{proof}

\noindent {\it Conclusion of the proof of Proposition \ref{W_partitioned}.} As an immediate consequence, we have that $\Big\{E(n\e,s) \ \Big| \ s \in \W(n\e)\Big\}$ is a countable family of pairwise disjoint sets. 
\end{proof}

\subsubsection{Proof of \eqref{decrease}}

Let us fix $n \geq 0$ and let us define $\mathcal{D}^I((n+1)\e)$. For each $m \in \Z$ such that $((n+1)\e,m\e)$ is an interaction point, let us define
\[
\mathcal{L}((n+1)\e,m\e) :=
\Big \{s \in \W^{(1)}(n\e,(m-1)\e) \ \Big| \ \sigma(n\e,s) \neq \sigma((n+1)\e,s) \Big\},
\]
and
\[
\mathcal{R}((n+1)\e,m\e) :=
\Big\{s \in \W^{(0)}(n\e,m\e) \ \Big| \ \sigma(n\e,s) \neq \sigma((n+1)\e,s) \Big\}.
\]

Then define
\begin{equation}
\label{dominio_interazione}
\mathcal{D}^I((n+1)\e) := \bigcup_{\substack{m \ \text{s.t.} \\ ((n+1)\e,m\e) \ \text{is} \\ \text{interaction point}}}\mathcal{L}((n+1)\e,m\e) \times \mathcal{R}((n+1)\e,m\e).
\end{equation}

\begin{theorem}
\label{decrease_thm_enunciato}
For any interaction point $((n+1)\e,m\e)$ it holds
\begin{equation*}
\begin{split}
\int_{\W((n+1)\e, m\e)} \big| \sigma((n+1)\e,s&) - \sigma(n\e,s) \big|\, ds \\
\leq&~ 2 \bigg[ \iint_{\mathcal{L}((n+1)\e,m\e) \times \mathcal{R}((n+1)\e,m\e)} \mathfrak{q}(n\e,s,s')\, dsds' \\ 
&~ \quad - \iint_{\mathcal{L}((n+1)\e,m\e) \times \mathcal{R}((n+1)\e,m\e)} \mathfrak{q}((n+1)\e,s,s') \, dsds' \bigg]. 
\end{split}
\end{equation*}
\end{theorem}

\begin{proof}
For simplicity, let us set $\mathcal{L} := \mathcal{L}((n+1)\e,m\e)$, $\mathcal{R} := \mathcal{R}((n+1)\e,m\e)$. We can assume $\mathcal{L}, \mathcal{R} \neq \emptyset$ (otherwise the statement is trivial) and all the waves in $\mathcal{L}, \mathcal{R}$ to be positive. Moreover define, as in proof of Theorem \ref{W_decreasing},
\begin{align*}
\mathcal{L}_2 :=&~ \Big\{s \in \mathcal{L} \ \Big| \ \text{there is } s' \in \mathcal{R} \text{ such that } s' \in \mathcal{I}(n\e, s)\Big\}, \qquad \mathcal{L}_1 := \mathcal{L} \setminus \mathcal{L}_2, \\
\mathcal{R}_1 :=&~ \Big\{s' \in \mathcal{R} \ \Big| \ \text{there is } s \in \mathcal{L} \text{ such that } s \in \mathcal{I}(n\e, s')\Big\}, \qquad \mathcal{R}_2 := \mathcal{R} \setminus \mathcal{R}_1.
\end{align*}
Then set
\begin{align*}
u_L := \inf \hat u(\mathcal{L}), \qquad
u_M := \sup \hat u(\mathcal{L}) = \inf \hat u(\mathcal{R}), \qquad
u_R := \sup \hat u(\mathcal{R}),
\end{align*}
and
\begin{align*}
u_1 := 
\begin{cases}
\inf \hat u(\mathcal{L}_2) & \text{ if } \mathcal{L}_2 \neq \emptyset, \\
u_M                        & \text{ if } \mathcal{L}_2 =    \emptyset, \\
\end{cases} \qquad
u_2 := 
\begin{cases}
\sup \hat u(\mathcal{R}_1) & \text{ if } \mathcal{R}_1 \neq \emptyset, \\
u_M                        & \text{ if } \mathcal{R}_1 =    \emptyset. \\
\end{cases}
\end{align*}

\begin{figure}
  \begin{center}
    \includegraphics[height=7.5cm,width=12cm]{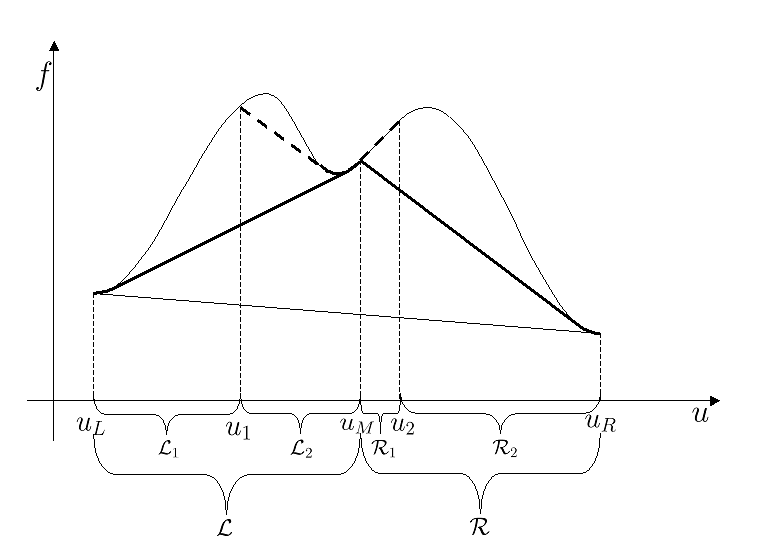}
    \caption{Families of waves which interact at point $((n+1)\e, m\e)$.}
    \label{fig:figura30}
    \end{center}
\end{figure}

\noindent See Figure \ref{fig:figura30}.

Moreover, given any positive interval of waves with nonempty interior, define \emph{the mean speed of waves in $\mathcal{I}$ at time $n\e$} as 
\begin{equation*}
\sigma_m(\mathcal{I}) := 
\begin{cases}
\displaystyle{\frac{f(\sup \hat u(\mathcal{I})) - f(\inf \hat u(\mathcal{I}))}{\sup \hat u(\mathcal{I}) - \inf \hat u(\mathcal{I})}} & \text{if } \mathcal{I} \neq \emptyset, \\
2 \lVert f'' \lVert_{L^\infty} & \text{if } \mathcal{I} = \emptyset.
\end{cases}
\end{equation*}

Observe that for each $s \in \mathcal{L} \cup \mathcal{R}$, $\sigma((n+1)\e,s)$ is equal to some constant $\lambda$, while 
\[
\sigma(n\e,s) = 
\begin{cases}
\frac{d}{du}\conv_{[u_L, u_M]} f(\hat u(s)) & \text{if } s \in \mathcal{L}, \\
\frac{d}{du}\conv_{[u_M, u_R]} f(\hat u(s)) & \text{if } s \in \mathcal{R}. 
\end{cases}
\]

Let us prove now the following 
\begin{lemma}
\label{um_tocca}
It holds $\conv_{[u_1,u_2]} f(u_M) = f(u_M)$.
\end{lemma}

\begin{proof}[Proof of Lemma \ref{um_tocca}]
Assume $u_1 < u_M < u_2$ (otherwise trivial). Take any sequence $(s_k)_{k \in \N}$ in $\mathcal{L}_2$ such that $\hat u(s_k) < u_M$, $\hat u(s_k) \nearrow u_M$ and $\sup \hat u(\mathcal{I}(n\e,s_k)) \nearrow u_2$. In a similar way, take another sequence $(s'_k)_{k \in \N}$ in $\mathcal{R}_1$ such that $\hat u(s'_k) > u_M$, $\hat u(s'_k) \searrow u_M$ and $\inf \hat u(\mathcal{I}(n\e,s'_k)) \searrow u_1$. For $k$ sufficiently large, $s'_k \in \mathcal{I}(n\e,s_k)$ and so, by Proposition \ref{divise_tocca}, $s_k, s'_k$ are divided by the Riemann problem $\mathcal{I}(n\e,s_k,s'_k)$; this means that there is some point $u_k$ such that 
\[
\conv_{\hat u (\mathcal{I}(n\e,s_k, s'_k))} f(u_k) = f(u_k).
\]
Finally, using Proposition \ref{convex_approximation}, one completes the proof.
\end{proof}

We  now complete the proof of the theorem. The steps are the same as in proof of Theorem \ref{W_decreasing}.

\textit{Step 1.} By definition of $\mathcal{L}, \mathcal{R}$, for any $(s,s') \in \mathcal{L} \times \mathcal{R}$, $s,s'$ are not divided in the real solution at time $(n+1)\e$; then, by Proposition \ref{unite_realta}, they are not divided by Riemann problem $\mathcal{I}((n+1)\e,s,s')$; hence, $\mathfrak{q}((n+1)\e,s,s') = 0$.

\textit{Step 2 and 3.} As in Theorem \ref{W_decreasing}, 
\begin{equation}
\label{zero}
\begin{split}
\int_{\W((n+1)\e, m\e)} &\big| \sigma((n+1)\e,s) - \sigma(n\e,s) \big|\, ds \\
 = &~ \int_{\mathcal{L}} \Big|\sigma((n+1)\e,s) - \sigma(n\e,s)\Big|\, ds + \int_\mathcal{R} \Big|\sigma((n+1)\e,s) - \sigma(n\e,s)\Big|\, ds \\
 = &~ \int_\mathcal{L} \Big(\sigma(n\e,s) - \sigma((n+1)\e,s)\Big)\, ds + \int_\mathcal{R} \Big(\sigma((n+1)\e,s) - \sigma(n\e,s)\Big)\, ds\\
 = &~  2 \cdot \Bigg [ \frac{f(u_M) - f(u_L)}{u_M - u_L} - \frac{f(u_R) - f(u_M)}{u_R - u_M}\Bigg ] \cdot \frac{(u_M - u_L)(u_R - u_M)}{u_R - u_L} \\
 = &~  2 \cdot \frac{|\sigma_m(\mathcal{L}) - \sigma_m(\mathcal{R})|}{|\mathcal{L}| + |\mathcal{R}|} \  |\mathcal{L}|  |\mathcal{R}| \\
 \leq &~  2 \Bigg [\frac{|\sigma_m(\mathcal{L}_1) - \sigma_m(\mathcal{R})|}{|\mathcal{L}| + |\mathcal{R}|} \  |\mathcal{L}_1|  |\mathcal{R}| 
+ \frac{|\sigma_m(\mathcal{L}_2) - \sigma_m(\mathcal{R}_1)|}{|\mathcal{L}| + |\mathcal{R}|} \  |\mathcal{L}_2|  |\mathcal{R}_1| \\ 
&~ + \frac{|\sigma_m(\mathcal{L}_2) - \sigma_m(\mathcal{R}_2)|}{|\mathcal{L}| + |\mathcal{R}|} \  |\mathcal{L}_2|  |\mathcal{R}_2| \Bigg]\\
\leq &~  2\Bigg [\lVert f'' \lVert_{L^\infty} |\mathcal{L}_1| |\mathcal{R}| + \frac{|\sigma_m(\mathcal{L}_2) - \sigma_m(\mathcal{R}_1)|}{|\mathcal{L}| + |\mathcal{R}|} \  |\mathcal{L}_2| |\mathcal{R}_1| + \lVert f'' \lVert_{L^\infty} |\mathcal{L}_2| |\mathcal{R}_2| \Bigg].
\end{split}
\end{equation}

\textit{Step 4.} Let us now concentrate our attention on the second term of the last summation. As a consequence of Lemma \ref{um_tocca}, we get
\begin{equation}
\label{uno}
\begin{split}
|\sigma(\mathcal{L}_2) - \sigma(\mathcal{R}_1)||\mathcal{L}_2| |\mathcal{R}_1|  
=&~ \Big|(f(u_2)-f(u_M))(u_M-u_1) - (f(u_M)-f(u_1))(u_2-u_M)\Big| \\
=&~ \Bigg | \int_{u_1}^{u_M} \int_{u_M}^{u_2} 
\bigg[\frac{d}{du}\conv_{[u_1,u_2]}f(u') - \frac{d}{du}\conv_{[u_1,u_2]}f(u) \bigg]dudu'
\Bigg | \\
=&~ \Bigg | \int_{\mathcal{L}_2} \int_{\mathcal{R}_1} 
\bigg[\frac{d}{du}\conv_{[u_1,u_2]}f(\hat u(s')) - \frac{d}{du}\conv_{[u_1,u_2]}f(\hat u(s)) \bigg]\, dsds' \Bigg |.
\end{split}
\end{equation}

We observe that, by definition of $u_1, u_2$, for any $s \in \mathcal{L}_2$ and $s' \in \mathcal{R}_1$, if $s,s'$ have already interacted at time $n\e$, then $\hat u(\mathcal{I}(n\e,s,s')) \subseteq [u_1,u_2]$. Together with Lemma \ref{um_tocca} and Proposition \ref{vel_aumenta}, this yields
\begin{equation*}
\frac{d}{du}\conv_{[u_1,u_2]}f(\hat u(s')) - \frac{d}{du}\conv_{[u_1,u_2]}f(\hat u(s)) 
 \leq \sigma(\mathcal{I}(n\e,s,s'),s') - \sigma(\mathcal{I}(n\e,s,s'),s),
\end{equation*}
and hence
\begin{equation}
\label{due}
\begin{split}
\frac{\frac{d}{du}\conv_{[u_1,u_2]}f(\hat u(s')) - \frac{d}{du}\conv_{[u_1,u_2]}f(\hat u(s))}{|\mathcal{L}|+|\mathcal{R}|}
\leq&~ \frac{\sigma(\mathcal{I}(n\e,s,s'),s') - \sigma(\mathcal{I}(n\e,s,s'),s)}{|\mathcal{L}|+|\mathcal{R}|} \\
\leq&~ \frac{\sigma(\mathcal{I}(n\e,s,s'),s') - \sigma(\mathcal{I}(n\e,s,s'),s)}{\hat u(s') - \hat u(s)} \\
=&~ \mathfrak{q}(n\e,s,s').
\end{split}
\end{equation}
Instead, if $s,s'$ have not yet interacted at time $n\e$, using Lagrange's Theorem, we get
\begin{equation}
\label{tre}
\frac{\frac{d}{du}\conv_{[u_1,u_2]}f(\hat u(s')) - \frac{d}{du}\conv_{[u_1,u_2]}f(\hat u(s))}{|\mathcal{L}|+|\mathcal{R}|} \leq \lVert f'' \lVert_{L^\infty} = \mathfrak{q}(n\e,s,s').
\end{equation}

Thus, by \eqref{uno}, \eqref{due}, \eqref{tre},
\begin{equation*}
\frac{|\sigma_m(\mathcal{L}_2) - \sigma_m(\mathcal{R}_1)||\mathcal{L}_2| |\mathcal{R}_1|}{|\mathcal{L}|+|\mathcal{R}|} 
\leq \iint_{\mathcal{L}_2 \times \mathcal{R}_1}\mathfrak{q}(n\e,s,s')\, dsds'.
\end{equation*}

\textit{Step 5.} Finally observe that if $s \in \mathcal{L}_1$ and $s' \in \mathcal{R}$, then by definition of $\mathcal{L}_1$, $s,s'$ have not yet interacted at time $n\e$. The same holds if $s \in \mathcal{L}_2$ and $s' \in \mathcal{R}_2$. Hence, recalling \eqref{zero}, we get
\begin{multline*}
\int_{\W((n+1)\e, m\e)} \Big|\sigma((n+1)\e,s) - \sigma(n\e,s)\Big|\, ds \\  
\begin{aligned}
& \leq 2\bigg [\lVert f'' \lVert_{L^\infty}  |\mathcal{L}_1| |\mathcal{R}| + \frac{|\sigma_m(\mathcal{L}_2) - \sigma_m(\mathcal{R}_1)|}{|\mathcal{L}| + |\mathcal{R}|} \cdot |\mathcal{L}_2| |\mathcal{R}_1| + \lVert f'' \lVert_{L^\infty} |\mathcal{L}_2| |\mathcal{R}_2| \bigg], 
\\
& \leq 2 \bigg[ \iint_{\mathcal{L}_1 \times \mathcal{R}} \mathfrak{q}(n\e,s,s')\, dsds' 
+ \iint_{\mathcal{L}_2 \times \mathcal{R}_1} \mathfrak{q}(n\e,s,s')\, dsds' 
+ \iint_{\mathcal{L}_2 \times \mathcal{R}_2} \mathfrak{q}(n\e,s,s')\, dsds' 
\bigg] \\
& = 2\iint_{\mathcal{L} \times \mathcal{R}} \mathfrak{q}(n\e,s,s')\, dsds' \\
& = 2 \bigg[ \iint_{\mathcal{L} \times \mathcal{R}} \mathfrak{q}(n\e,s,s')\, dsds' - \iint_{\mathcal{L} \times \mathcal{R}} \mathfrak{q}((n+1)\e,s,s')\, dsds'\bigg],
\end{aligned} 
\end{multline*}
where the last equality is a direct consequence of Point 1.
\end{proof}

As a corollary, we immediately obtain inequality \eqref{decrease}.

\begin{corollary}
\label{decrease_thm_corollary}
It holds
\begin{align*}
\sum_{\substack{m \in \Z \ \mathrm{ s.t.}\\ ((n+1)\e,m\e)  \\ \emph{interaction}}}& \int_{\W((n+1)\e, m\e)} \big| \sigma((n+1)\e,s) - \sigma(n\e,s) \big|\, ds \\ 
\leq&~ 2 \Bigg [ \iint_{\mathcal{D}^I((n+1)\e)} \mathfrak{q}(n\e, s,s')\, dsds' - 
\iint_{\mathcal{D}^I((n+1)\e)} \mathfrak{q}((n+1)\e, s,s')\, dsds' \Bigg ].
\end{align*}
\end{corollary}

\begin{proof}
Immediate consequence of Theorem \ref{decrease_thm_enunciato}, definition of $\mathcal{D}^I((n+1)\e)$ in \eqref{dominio_interazione} and the fact that for $m \neq m'$, $\mathcal{L}((n+1)\e,m\e) \times \mathcal{R}((n+1)\e,m\e)$ and $\mathcal{L}((n+1)\e,m'\e) \times \mathcal{R}((n+1)\e,m'\e)$ are disjoint. 
\end{proof}


\subsubsection{Proof of \eqref{increase}}

Let $n\e$ be a fixed time. First of all let us prove the following

\begin{proposition}
\label{maximal_hom_iow}
It is possible to write $\W(n\e)$ as a (locally finite) countable union of mutually disjoint, maximal (with respect to set inclusion) homogenous interval of waves $\mathcal M_k$,
\begin{equation*}
\W(n\e) = \bigcup_{k \in \Z} \mathcal{M}_k.
\end{equation*}
\end{proposition}

\begin{proof}
It is sufficient to partition $\W(n\e)$ with respect to the equivalence relation
\begin{equation*}
s \sim s' \quad \Longleftrightarrow \quad [s,s'] \cap \W(n\e) \text{ is an homogeneous interval of waves}.
\end{equation*}
This partition is countable, since any two waves having the same position belong to the same equivalence class.
\end{proof}

For any interval $\mathcal{M}_k$, let us define the following quantities:
\begin{align*}
&L^{(k)}\e := \inf \big\{\mathtt x(n\e,p) \ \big| \ p \in \mathcal{M}_k \big\}, & a_1^{(k)} :=&~ \inf \hat u(\mathcal{M}_k), \\
& R^{(k)}\e := \sup \big\{\mathtt x(n\e,p) \ \big| \ p \in \mathcal{M}_k \big\}, & b_1^{(k)} :=&~ \sup \hat u(\mathcal{M}_k).
\end{align*}
Then set
\begin{equation*}
a_3^{(k)} = 
\begin{cases}
\sup \hat u \big( \W^{(0)}(n\e,L^{(k)}\e) \big) & \text{if } L^{(k)} > -\infty \text{ and } \W^{(0)}(n\e,L^{(k)}\e) \neq \emptyset, \\
a_1^{(k)}                            & \text{if } L^{(k)} = -\infty \text{ or } \W^{(0)}(n\e,L^{(k)}\e) = \emptyset, 
\end{cases}
\end{equation*}
and, in a similar way,
\begin{equation*}
b_3^{(k)} = 
\begin{cases}
\inf \hat u \big(\W^{(1)}(n\e,R^{(k)}\e) \big) & \text{if } R^{(k)} < +\infty \text{ and } \W^{(1)}(n\e,R^{(k)}\e) \neq \emptyset, \\
b_1^{(k)}                            & \text{if } R^{(k)} = +\infty \text{ or } \W^{(1)}(n\e,R^{(k)}\e) = \emptyset. 
\end{cases}
\end{equation*}
Finally define
\begin{equation*}
a_2^{(k)} = 
\begin{cases}
\inf \hat u \big(\W^{(0)}(n\e,L^{(k)}\e) \cap \W((n+1)\e)\big)
& 
\begin{aligned} 
& \text{if } L^{(k)} > -\infty \text{ and} \\ & \W^{(0)}(n\e,L^{(k)}\e) \cap \W((n+1)\e) \neq \emptyset,
\end{aligned}
\\ \\
a_3^{(k)}             &
\begin{aligned} 
& \text{if } L^{(k)} = -\infty \text{ or} \\ & \W^{(0)}(n\e,L^{(k)}\e) \cap \W((n+1)\e) = \emptyset, 
\end{aligned}
\\
\end{cases}
\end{equation*}
and
\begin{equation*}
b_2^{(k)} = 
\begin{cases}
\sup \hat u \big(\W^{(1)}(n\e,R^{(k)}\e)  \cap \W((n+1)\e)\big)
& 
\begin{aligned} 
& \text{if } R^{(k)} < +\infty \text{ and} \\ & \W^{(1)}(n\e,R^{(k)}\e) \cap \W((n+1)\e) \neq \emptyset, 
\end{aligned}
\\ \\
b_3^{(k)}                            & 
\begin{aligned} 
& \text{if } R^{(k)} = +\infty \text{ or} \\ & \W^{(1)}(n\e,R^{(k)}\e) \cap \W((n+1)\e) = \emptyset. 
\end{aligned}
\\
\end{cases}
\end{equation*}
Clearly $ a_1^{(k)} \leq a_2^{(k)} \leq a_3^{(k)} \leq b_3^{(k)} \leq b_2^{(k)} \leq b_1^{(k)}$.

\begin{theorem}
\label{increase_thm_enunciato}
For any maximal interval $\mathcal{M}_k$, it holds
\begin{align*}
\iint_{(\mathcal{M}_k \times \mathcal{M}_k) \cap \mathcal{D}((n+1)\e)} &\big[ \mathfrak{q}((n+1)\e,s,s') - \mathfrak{q}(n\e,s,s') \big]^+ \, dsds' \\
\leq&~ \log(2) \lVert f'' \lVert_{L^\infty}  \mathcal{L}^1(\mathcal{M}_k) \Big [ (b_1^{(k)}-b_2^{(k)}) + (a_2^{(k)} -a_1^{(k)})\Big ].
\end{align*}
\end{theorem}

\begin{proof}
We will not write the index $k$ and we assume that waves in $\mathcal{M} = \mathcal{M}_k$ are positive, the other case being entirely similar. First of all let $g,h: [a_1, b_1] \to \R$ be two $\mathcal{C}^{1,1}$ function with the following properties:
\begin{align*}
g_{|[a_1,a_3]} & = \conv_{[a_1,a_3]}f,     &     h_{|[a_2,a_3]} & = \conv_{[a_2,a_3]}f,   \\
g_{|[b_3,b_1]} & =  \conv_{[b_3,b_1]}f,    &     h_{|[b_3,b_2]} & =  \conv_{[b_3,b_2]}f, \\
\end{align*}
and set
\[
g_{|[a_3,b_3]} = h_{[a_3,b_3]}.
\]
See Figure \ref{fig:figura31}. It is easy to see that some $g,h$ with the above properties exist.

\begin{figure}
  \begin{center}
    \includegraphics[height=7.5cm,width=12cm]{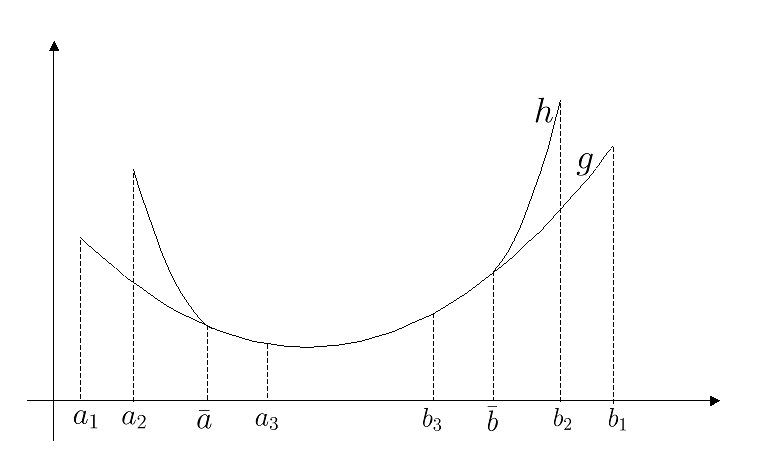}
    \caption{Graph of $g$ and $h$.}
    \label{fig:figura31}
    \end{center}
\end{figure}

\begin{lemma}
\label{lemma_uno}
For any $(s,s') \in (\mathcal{M} \times \mathcal{M}) \cap \mathcal{D}((n+1)\e)$,
\begin{equation}
\label{passaggio_gh}
\big[ \mathfrak{q}((n+1)\e, s,s') - \mathfrak{q}(n\e,s,s') \big]^+  
\leq \frac{\Big| \big( h'(\hat u(s')) - h'(\hat u(s)) \big) - \big( g'(\hat u(s')) - g'(\hat u(s)) \big) \Big|}{|\hat u(s')- \hat u(s)|}. 
\end{equation}
\end{lemma}

\begin{proof}
If $s,s'$ have not yet interacted at time $n\e$, then the l.h.s. of \eqref{passaggio_gh} is equal to zero. Thus we can assume $s,s'$ have already interacted at time $n\e$. In this case set for simplicity 
\begin{align*}
\mathcal{I}_{n+1} :=&~ \mathcal{I}((n+1)\e,s,s'), & 
\sigma_{n+1}(s')  :=&~ \sigma(\mathcal{I}_{n+1},s'), &
\sigma_{n}(s')    :=&~ \sigma(\mathcal{I}_{n},s'), \\
\mathcal{I}_{n}   :=&~ \mathcal{I}(n\e,s,s'), & 
\sigma_{n+1}(s)   :=&~ \sigma(\mathcal{I}_{n+1},s), &
\sigma_{n}(s)     :=&~ \sigma(\mathcal{I}_{n},s).
\end{align*}
The proof reduces to prove the following inequality:
\begin{equation}
\label{da_provare}
\Big[ \big( \sigma_{n+1}(s') - \sigma_{n+1}(s) \big) - \big( \sigma_n(s') - \sigma_n(s) \big) \Big]^+ 
\leq \Big| \big( h'(\hat u(s')) - h'(\hat u(s)) \big) - \big( g'(\hat u(s')) - g'(\hat u(s)) \big) \Big|.
\end{equation}

If $s,s'$ are not divided in the real solution at time $(n+1)\e$, then, by Proposition \ref{unite_realta}, $\sigma_{n+1}(s') - \sigma_{n+1}(s) = 0$; in this case the l.h.s. of \eqref{da_provare} is equal to zero.

Let us then assume $s,s'$ are divided at time $(n+1)\e$. We consider four main cases.

\smallskip
{\it Case 1.} Assume $\hat u(s), \hat u(s') \in (a_3,b_3]$. In this case let us observe what follows:
	\begin{enumerate}
	 \item The r.h.s. of \eqref{da_provare} is equal to zero;
	 \item $s,s'$ are divided also at time $n\e$, due to their position;
	 \item $\mathcal{I}_n \cap (a_3,b_3] \subseteq \mathcal{I}_{n+1} \cap (a_3,b_3]$, because in $(a_3,b_3]$ only interactions can occur;
	 \item for $j \in \{n,n+1\}$ and $u \in \{a_3,b_3\}$, if $u \in \overline{\hat u(\mathcal{I}_{j})}$, then $\conv_{\hat u(\mathcal{I}_{j})} f(u) = f(u)$. Indeed let us prove this equality for $j = n+1, u=a_3$, the other ones being similar. If $a_3 = \inf \hat u(\mathcal{I}_{n+1})$, we have done. So let us assume $a_3 > \inf \hat u(\mathcal{I}_{n+1})$. Then one can find two sequences $(p_k)_k, (p'_k)_k$ in $\mathcal{I}_{n+1}$ such that $\hat u(p_k) < a_3$ and $\hat u(p_k) \nearrow a_3$, while $\hat u(p'_k) > a_3$ and $\hat u(p'_k) \searrow a_3$ and, for each $k$, $p_k,p'_k$ have already interacted at time $(n+1)\e$. Hence, since $s,s'$ are divided at time $(n+1)\e$, by Proposition \ref{divise_tocca}, $p_k, p_k'$ are divided by the Riemann problem $\mathcal{I}_{n+1}$, i.e. there is $u_k \in (\hat u(p_k), \hat u(p'_k))$ such that $\conv_{\hat u (\mathcal{I}_{n+1})} f(u_k) = f(u_k)$. Passing to the limit, we get the thesis.
	\end{enumerate}
\noindent Thus for $j \in \{n,n+1\}$, $\conv_{\hat u(\mathcal{I}_j)} f = \conv_{\hat u(\mathcal{I}_j) \cap [a_3,b_3]} f$ on $\hat u(\mathcal{I}_j) \cap [a_3,b_3]$. Hence, 
\begin{align*}
\sigma_n(s') - \sigma_n(s) 
= &~ \bigg (\frac{d}{du}\conv_{\hat u(\mathcal{I}_n)} f \bigg) \big(\hat u(s')\big) - \bigg (\frac{d}{du}\conv_{\hat u(\mathcal{I}_n)} f \bigg) \big(\hat u(s)\big) \\
= &~ \bigg (\frac{d}{du}\conv_{\hat u(\mathcal{I}_n) \cap [a_3,b_3]} f \bigg) \big(\hat u(s')\big) - \bigg (\frac{d}{du}\conv_{\hat u(\mathcal{I}_n)\cap [a_3,b_3]} f \bigg) \big(\hat u(s)\big) \\
(\text{by Prop. \ref{vel_aumenta}}) \ \geq &~ \bigg (\frac{d}{du}\conv_{\hat u(\mathcal{I}_{n+1}) \cap [a_3,b_3]} f \bigg) \big(\hat u(s')\big) - \bigg (\frac{d}{du}\conv_{\hat u(\mathcal{I}_{n+1})\cap [a_3,b_3]} f \bigg) \big(\hat u(s)\big) \\
=&~ \bigg (\frac{d}{du}\conv_{\hat u(\mathcal{I}_{n+1})} f \bigg) \big(\hat u(s')\big) - \bigg (\frac{d}{du}\conv_{\hat u(\mathcal{I}_{n+1})} f \bigg) \big(\hat u(s)\big) \\
=&~ \sigma_{n+1}(s') - \sigma_{n+1}(s).
\end{align*}
Thus, the l.h.s. of \eqref{da_provare} is equal to zero.

\smallskip
{\it Case 2.} Assume now $\hat u(s), \hat u(s') \in (b_3, b_2]$. In this case 
\begin{align*}
\inf \hat u(\mathcal{I}_n), \inf \hat u(\mathcal{I}_{n+1})    \leq b_3, \qquad \sup \hat u(\mathcal{I}_n) = b_1, \qquad \sup \hat u(\mathcal{I}_{n+1}) =  b_2.
\end{align*}
Since $s,s'$ are divided at time $(n+1)\e$ in the real solution, we can argue as in the previous case and use Proposition \ref{divise_tocca} to obtain
\begin{eqnarray*}
\conv_{\hat u(\mathcal{I}_{n+1})} f = \conv_{\hat u(\mathcal{I}_{n+1}) \cap [b_3,b_2]} f = \conv_{[b_3,b_2]} f = h \hspace{1cm} \text{on } [b_3,b_2].
\end{eqnarray*}
Hence 
\begin{equation}
\label{sigma_n+1}
\sigma_{n+1}(s') - \sigma_{n+1}(s) = h'(\hat u(s')) - h'(\hat u(s)).
\end{equation}
Now distinguish two possibilities:
\begin{enumerate}
\item $g'(\hat u(s')) - g'(\hat u(s)) = 0$: in this case $\hat u(s), \hat u(s')$ belong to the same wavefront interval of $g|_{[b_3,b_2]} = \conv_{[b_3,b_1]}f$; thus they are not divided in the real solution at time $n\e$ and so by Proposition \ref{unite_realta}, $\sigma_n(s') - \sigma_n(s) = 0$, which, together with \eqref{sigma_n+1}, yields the thesis.
\item $g'(\hat u(s')) - g'(\hat u(s)) > 0$: this means that $s,s'$ are divided in the real solution at time $n\e$ and so, using the same argument as before, by Proposition \ref{divise_tocca},
\begin{equation*}
\conv_{\hat u(\mathcal{I}_{n})} f = \conv_{\hat u(\mathcal{I}_{n}) \cap [b_3,b_1]} f = \conv_{[b_3,b_1]} f = g \hspace{1cm} \text{on } [b_3,b_1],
\end{equation*} 
which, together with \eqref{sigma_n+1}, yields the thesis.
\end{enumerate}

\smallskip
{\it Case 3.} Assume now $\hat u(s) \in (a_3, b_3]$ and $\hat u(s') \in (b_3, b_2]$. In this case $s,s'$ are divided in the real solution also at time $n\e$. Hence, arguing as in the first point, 
\begin{equation}
\label{b1}
\conv_{\hat u(\mathcal{I}_n)}f |_{\hat u(\mathcal{I}_n) \cap [a_3,b_1]} =  \conv_{\hat u(\mathcal{I}_n) \cap [a_3,b_3]}f \cup \conv_{\hat u(\mathcal{I}_n) \cap [b_3,b_1]}f,
\end{equation}
and
\begin{equation}
\label{b2}
\conv_{\hat u(\mathcal{I}_{n+1})}f |_{\hat u(\mathcal{I}_{n+1}) \cap [a_3,b_2]} 
=  \conv_{\hat u(\mathcal{I}_{n+1}) \cap [a_3,b_3]}f \cup \conv_{\hat u(\mathcal{I}_{n+1}) \cap [b_3,b_2]}f.
\end{equation}
Thus, observing that 
\begin{equation}
\label{b3}
\hat u(\mathcal{I}_{n+1}) \cap [a_3,b_3] = \hat u(\mathcal{I}_{n}) \cap [a_3,b_3]
\end{equation}
($s,s'$ can not interact at time $(n+1)\e$),
\begin{align*}
\sigma_{n+1}(&s) = \frac{d}{du}\conv_{\hat u(\mathcal{I}_{n+1})} f(\hat u(s)) 
                \overset{\eqref{b2}}{=} \frac{d}{du}\conv_{\hat u(\mathcal{I}_{n+1}) \cap [a_3,b_3]} f(\hat u(s)) \\
                &\, \overset{\eqref{b3}}{=} \frac{d}{du}\conv_{\hat u(\mathcal{I}_{n}) \cap [a_3,b_3]} f(\hat u(s)) 
                \overset{\eqref{b1}}{=} \frac{d}{du}\conv_{\hat u(\mathcal{I}_{n})} f(\hat u(s)) 
                = \sigma_n(s),
\end{align*}
from which we deduce
\begin{equation}
\label{diff_sigma}
\sigma_{n+1}(s) - \sigma_n(s) = 0 = h'(\hat u(s)) - g'(\hat u(s)).
\end{equation}
For $s'$, distinguish two cases:
	\begin{enumerate}
	 \item If $\sup \hat u(\mathcal{I}_n) < b_2$, then $\hat u(\mathcal{I}_{n+1}) \cap [b_3,b_2] = \hat u(\mathcal{I}_{n}) \cap [b_3,b_1]$, and so
	 \begin{align*}
\sigma_{n+1}(s') =&~ \frac{d}{du}\conv_{\hat u(\mathcal{I}_{n+1})} f(\hat u(s')) 
                 \overset{\eqref{b2}}{=} \frac{d}{du}\conv_{\hat u(\mathcal{I}_{n+1}) \cap [b_3,b_2]} f(\hat u(s')) \\
                 =&~ \frac{d}{du}\conv_{\hat u(\mathcal{I}_{n}) \cap [b_3,b_1]} f(\hat u(s')) 
                 \overset{\eqref{b1}}{=} \frac{d}{du}\conv_{\hat u(\mathcal{I}_{n})} f(\hat u(s')) 
                 = \sigma_n(s').
	\end{align*}
	
	\item If $\sup \hat u(\mathcal{I}_n) \geq b_2$, then $\sup \hat u(\mathcal{I}_{n+1}) = b_2$ and so 
	 \begin{align*}
\sigma_{n+1}(s') =&~ \frac{d}{du}\conv_{\hat u(\mathcal{I}_{n+1})} f(\hat u(s')) 
                 \overset{\eqref{b2}}{=} \frac{d}{du}\conv_{\hat u(\mathcal{I}_{n+1}) \cap [b_3,b_2]} f(\hat u(s')) \\
                 =&~ \frac{d}{du}\conv_{[b_3,b_2]} f(\hat u(s')) 
                 = h(\hat u(s')),
	\end{align*}
while
	 \begin{align*}
\sigma_{n}(s') =&~ \frac{d}{du}\conv_{\hat u(\mathcal{I}_{n})} f(\hat u(s')) 
               \overset{\eqref{b1}}{=} \frac{d}{du}\conv_{\hat u(\mathcal{I}_{n}) \cap [b_3,b_1]} f(\hat u(s')) \\
\geq&~ \frac{d}{du}\conv_{[b_3,b_1]} f(\hat u(s')) 
                                          = g(\hat u(s')),
	\end{align*}
	where the inequality is due to Proposition \ref{vel_aumenta}.
	\end{enumerate}
In both cases (1) and (2), $\sigma_{n+1}(s') - \sigma_n(s') \leq h(\hat u(s')) - g(\hat u(s'))$. Together with \eqref{diff_sigma}, this yields the thesis.

\smallskip
{\it Case 4.} Assume now $\hat u(s) \in (a_2,a_3]$, $\hat u(s') \in (b_3,b_2]$. As in the previous point, 
\begin{align*}
\sigma_{n+1}(s') - \sigma_n(s') \leq h'(\hat u(s')) - g'(\hat u(s')), \qquad
\sigma_{n+1}(s) - \sigma_n(s) \geq h'(\hat u(s)) - g'(\hat u(s)),
\end{align*}
and so the thesis follows.

\smallskip
The proof of the remaining cases $\hat u(s) \in (a_2,a_3]$, $\hat u(s') \in (a_3,b_3]$ and $\hat u(s) \in (a_2,a_3]$, $\hat u(s') \in (a_2,a_3]$ is similar to Case (3) and Case (2) respectively.
\end{proof}

\begin{lemma}
\label{lemma_due}
We have
\begin{align*}
\iint_{(\mathcal{M} \times \mathcal{M}) \cap \mathcal{D}((n+1)\e)} \frac{\big| \big( h'(\hat u(s')) - h'(\hat u(s)) \big) - \big( g'(\hat u(s')) - g'(\hat u(s)) \big) \big|}{|\hat u(s')- \hat u(s)|}\, dsds'& \\
\leq \const \mathcal{L}^1(\mathcal{M}) \Big[ \big(g'(a_2) - h'(a_2)\big) + \big(&h'(b_2) - g'(b_2)\big) \Big].
\end{align*}
\end{lemma}

\begin{proof}
The proof is the analog of Lemma \ref{W_lemma_due}. In fact
\begin{multline*}
\iint_{(\mathcal{M} \times \mathcal{M}) \cap \mathcal{D}((n+1)\e)} \frac{\big| \big( h'(\hat u(s')) - h'(\hat u(s)) \big) - \big( g'(\hat u(s')) - g'(\hat u(s)) \big) \big|}{|\hat u(s')- \hat u(s)|} dsds' \\
\begin{aligned}
\text{(by Proposition \ref{cambio_di_variabile})} & = \int_{a_2}^{b_2} \int_{u}^{b_2} \frac{\big| \big( h'(u') - h'(u) \big) - \big( g'(u') - g'(u) \big) \big|}{u'- u} du'du \\
& = \int_{a_2}^{b_2} \int_{u}^{b_2} \frac{1}{u'- u} \bigg| \int_u^{u'} [h''(\xi) - g''(\xi)]d\xi
 \bigg|du'du \\
& \leq \int_{a_2}^{b_2} \int_{u}^{b_2} \frac{1}{u'- u} \int_u^{u'} |h''(\xi) - g''(\xi)|d\xi
 du'du \\
& \leq \int_{a_2}^{b_2} \big|h''(\xi) - g''(\xi)\big| \bigg(\int_{a_2}^\xi \int_\xi^{b_2}  \frac{1}{u'- u}    du'du \bigg) d\xi.
\end{aligned}
\end{multline*}
Hence, by \eqref{logaritmo},
\begin{multline*}
\iint_{(\mathcal{M} \times \mathcal{M}) \cap \mathcal{D}((n+1)\e)} \frac{\big| \big( h'(\hat u(s')) - h'(\hat u(s)) \big) - \big( g'(\hat u(s')) - g'(\hat u(s)) \big) \big|}{|\hat u(s')- \hat u(s)|} dsds' \\
\begin{aligned}
& \leq \log(2) (b_2 - a_2) \int_{a_2}^{b_2} |h''(\xi) - g''(\xi)|  d\xi \\
& \leq \log(2) \mathcal{L}^1(\mathcal{M})  \bigg[ \int_{a_2}^{a_3} |h''(\xi) - g''(\xi)|  d\xi
+ \int_{a_3}^{b_3} |h''(\xi) - g''(\xi)|  d\xi + \int_{b_3}^{b_2} |h''(\xi) - g''(\xi)|  d\xi \bigg].
\end{aligned}
\end{multline*}
Now remember that $g$ and $f$ coincide on $[a_3,b_3]$. Moreover, by Corollary \ref{RP_ridotto}, there are $\bar a \in [a_2, a_3], \bar b \in [b_3,b_2]$ such that $g = h$ on $[\bar a, a_3] \cup [b_3, \bar b]$ and $g$ is affine on $[a_2, \bar a]$ and $[\bar b, b_2]$. 
Thus
\begin{multline*}
\iint_{(\mathcal{M} \times \mathcal{M}) \cap \mathcal{D}((n+1)\e)} \frac{\big| \big( h'(\hat u(s')) - h'(\hat u(s)) \big) - \big( g'(\hat u(s')) - g'(\hat u(s)) \big) \big|}{|\hat u(s')- \hat u(s)|} dsds' \\
\begin{aligned}
& \leq \log(2) \mathcal{L}^1(\mathcal{M}) \bigg[ \int_{a_2}^{\bar a} |h''(\xi)|  d\xi + \int_{\bar b}^{b_2} |h''(\xi)|  d\xi \bigg] \\
& =    \log(2) \mathcal{L}^1(\mathcal{M}) \bigg[ \int_{a_2}^{\bar a} h''(\xi)  d\xi + \int_{\bar b}^{b_2} h''(\xi)  d\xi \bigg] \\
& = \log(2) \mathcal{L}^1(\mathcal{M}) \Big[ \big(h'(\bar a) - h'(a_2)\big) + \big(h'(b_2) - h'(\bar b)\big) \Big] \\
& = \log(2) \mathcal{L}^1(\mathcal{M}) \Big[ \big(g'(a_2) - h'(a_2)\big) + \big(h'(b_2) - g'(b_2)\big) \Big], 
\end{aligned}
\end{multline*}
concluding the proof of Lemma \ref{lemma_due}.
\end{proof}

\noindent{\it End of the proof of Theorem \ref{decrease_thm_enunciato}}. Let us observe that, by Proposition \ref{diff_vel_proporzionale_canc}, 
\begin{align}
\label{lemma_tre}
(g'(a_2) - h'(a_2)) & \leq \lVert f'' \lVert_{L^\infty}  (a_2 - a_1),  &  (h'(b_2) - g'(b_2)) & \leq \lVert f'' \lVert_{L^\infty}  (b_1 - b_2).
\end{align}
Putting together Lemma \ref{lemma_uno}, Lemma \ref{lemma_due} and inequalities \eqref{lemma_tre}, one easily concludes the proof of the theorem.
\end{proof}

\begin{corollary}
\label{increase_thm_corollary}
It holds
\begin{align*}
\iint_{\mathcal{D}((n+1)\e) \setminus \mathcal{D}^I((n+1)\e)}& \big[ \mathfrak{q}((n+1)\e, s,s') - 
\mathfrak{q}(n\e, s,s') \big]\, dsds'\\
&\leq \log(2) \lVert f'' \lVert_{L^\infty}  \TV(u(0,\cdot)) \Big[\TV(u_\e(n\e,\cdot) - \TV(u_\e((n+1)\e, \cdot))\Big].  
\end{align*}
\end{corollary}

\begin{proof}
We first observe that
\begin{equation}
\label{decr_1}
\begin{split}
\iint_{\mathcal{D}((n+1)\e) \setminus \mathcal{D}^I((n+1)\e)}& \big[ \mathfrak{q}((n+1)\e, s,s') - 
\mathfrak{q}(n\e, s,s') \big]\, dsds'\\
\leq&~ \iint_{\mathcal{D}((n+1)\e) \setminus \mathcal{D}^I((n+1)\e)} \big[ \mathfrak{q}((n+1)\e, s,s') - 
\mathfrak{q}(n\e, s,s') \big]^+\, dsds'\\ 
\leq&~ \iint_{\mathcal{D}((n+1)\e)} \big[ \mathfrak{q}((n+1)\e, s,s') - 
\mathfrak{q}(n\e, s,s') \big]^+\, dsds'. 
\end{split}  
\end{equation}

Now, if $s \in \mathcal{M}_k$, $s' \in \mathcal{M}_{k'}$, with $k \neq k'$, then $s,s'$ have not yet  interacted at time $n\e$; this means that $\mathfrak{q}(n\e,s,s') = \lVert f'' \lVert_{L^\infty} $ and hence
\[
\big[ \mathfrak{q}((n+1)\e, s,s') - \mathfrak{q}(n\e, s,s') \big]^+ = 0.
\]
Thus we can continue the chain of inequalities in \eqref{decr_1} as follows
\begin{align*}
\iint_{\mathcal{D}((n+1)\e)}& \big[ \mathfrak{q}((n+1)\e, s,s') - 
\mathfrak{q}(n\e, s,s') \big]^+\, dsds' \\
& \leq \sum_{k \in \Z} \iint_{(\mathcal{M}_k \times \mathcal{M}_k) \cap \mathcal{D}((n+1)\e)} \big[ \mathfrak{q}((n+1)\e, s,s') - 
\mathfrak{q}(n\e, s,s') \big]^+\, dsds' \\
& \leq \log(2) \lVert f'' \lVert_{L^\infty}  \TV(u(0,\cdot))\sum_{k \in \Z} \big[ (b_1^{(k)}-b_2^{(k)}) + (a_2^{(k)} -a_1^{(k)}) \big],
\end{align*}  
where the last inequality is an immediate consequence of Theorem \ref{decrease_thm_enunciato}.

By the definition of $a_1^{(k)}$, $a_2^{(k)}$, $b_1^{(k)}$, $b_2^{(k)}$,
\begin{align*}
\sum_{k \in \Z} \big[ (b_1^{(k)}-b_2^{(k)}) + (a_2^{(k)} -a_1^{(k)}) \big] 
=&~ \big| \W(n\e) \setminus \W((n+1)\e) \big| \\
=&~ \TV(u_\e(n\e,\cdot)) - \TV(u_\e((n+1)\e, \cdot)). 
\end{align*}
and this concludes the proof. 
\end{proof}

\appendix

\section{A counterexample to an estimate in the proof of Lemma 2 in \cite{anc_mar_11_CMP}}
\label{App_conter}

As we said in the Introduction, in order to obtain the estimate \eqref{E_quadratic_G} and hence to prove Theorem \ref{main_thm}, the authors restrict the evaluation of the oscillation of $G$ to regions belonging to a disjoint partition of the plane: each region $\mathcal{T}_i$ is called a \emph{tree} in \cite{anc_mar_11_CMP}. We give the definition of tree only for a wavefront solution to a scalar conservation law, because it is definitely simpler.

We recall that in \cite{anc_mar_11_CMP} a \emph{splitting} is a cancellation point $(\bar t, \bar x)$ where the solution to the Riemann problem $[u(\bar t,\bar x-),u(\bar t,\bar x+)]$ is made of more than one wavefront.

\begin{definition}
\label{Def_tree}
If at $(\bar t, \bar x)$ a splitting occurs, then the \emph{tree starting at $(\bar t,\bar x)$} is the collection of all the backward trajectories $\{x_i(t), t \in (t_i,\bar t)\}$ of the waves $s_i$ surviving after the cancellation, such that either $t_i = 0$ or $t_i$ is the last time when a splitting involving $s_i$ occurs before $\bar t$.
\end{definition}


It is fairly easy to see that the trees are disjoint, since in each tree only interactions and cancellations occur but no splitting.

The fundamental estimate is thus the one contained in \cite{anc_mar_11_CMP}, Lemma 2: the oscillation of $G(t)$ restricted to a tree $\mathcal T$ is bounded by
\begin{equation}
\label{E_wrong_tree_G}
\mathrm{Osc}(G,\mathcal T) \leq \mathcal O(\|f''\|_{L^\infty}) \Big\{ \text{cubic interactions and [cancellations$\times \TV(u)$] in } \mathcal T \Big\}.
\end{equation}
where the 'cubic interactions' term corresponds by definition to the decreases of the functional $Q^{\text{BB}}$ defined in \eqref{E_cubic_Q_def}, on the tree $\mathcal{T}$. In fact in the scalar case no waves of other families are present.
They thus conclude that
\begin{align*}
G(0) =&~ \sum_{\text{trees } \mathcal T} \mathrm{Osc}(G,\mathcal T) \crcr
\leq&~ \mathcal O(\|f''\|_{L^\infty}) \sum_{\text{trees } \mathcal T} \Big\{ \text{cubic interaction estimates and [cancellations$\times \TV(u)$] in } \mathcal T \Big\} \crcr
\leq&~ \mathcal O(\|f''\|_{L^\infty}) \Big\{ \text{cubic interaction estimates and [cancellations$\times \TV(u)$] in } \R^+ \times \R \Big\} \crcr
\leq&~ \mathcal O(\|f''\|_{L^\infty}) \TV(\bar u)^2.
\end{align*}

A key point in the proof is estimates (4.84), (4.85) in \cite{anc_mar_11_CMP}, where the authors bounds the variation in time of the speeds of waves belonging to a given tree in terms of the cubic amount of interaction and the cancellation.

Now we present an explicit counterexample of a tree in which this estimate fails. 
The idea is that there can be a large interaction in a tree, and a vary small cancellation at its top, corresponding to a splitting of some waves. Then the quantity
\[
\frac{|\sigma(w) - \sigma(w')| |w| |w'|}{|w| + |w'|}
\]
estimated at the interaction time is of the order of $\TV(u)^2$, and thus the variation of speed of the wavefronts is of order $\TV(u)$ (see \eqref{W_zzero}).
The cancellation can be made arbitrarily small, so that the estimate which should be proved becomes
\begin{align*}
\text{variation of speed of $w$} \simeq &~ \mathcal O(\|f''\|_{L^\infty}) \TV(u) \\
\leq &~ \mathcal O(\|f''\|_{L^\infty})  \Big[ \TV(u)^3 + \text{ cancellation} \Big] \\
\leq &~ \mathcal O(\|f''\|_{L^\infty}) \TV(u)^3.
\end{align*}

Clearly, being the l.h.s. linear and the r.h.s. cubic, this estimate sounds wrong, and this is proved in the explicit example right below.

Define the flux $f$ (it is only $C^{1,1}$, but it is clear that we can make it $C^2$): for $0 < \e \ll L \ll 1$ and $\alpha>0$,
\begin{equation}
\label{E_controesempio_f}
f(u) :=
\begin{cases}
- \frac{1}{2}\alpha (u + 2L)^2 + \alpha L^2 & u \leq - L, \crcr
\frac{1}{2} \alpha u^2 & -L < u \leq \e, \crcr
- \frac{1}{2} \alpha (u - 2\e)^2 + \alpha \e^2 & \e < u \leq 3 \e, \crcr
- \alpha \e (u - 3 \e) + \frac{1}{2} \alpha \e^2 & u > 3\e.
\end{cases}
\end{equation}

\begin{figure}
  \begin{center}
    \includegraphics[height=7.5cm,width=12cm]{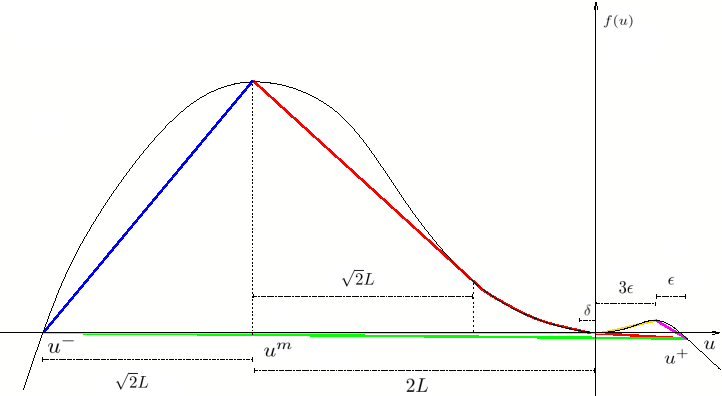}
    \caption{The flux $f$ defined in \eqref{E_controesempio_f}.}
    \label{fabio01}
    \end{center}
\end{figure}

\noindent For $\delta = \e (\sqrt{17} - 4)$, define also
\[
u^- = - L \big( 2 + \sqrt{2} \big), \quad u^m = - 2L, \quad u^+ = \frac{7 \e^2 + \delta^2}{2(\e - \delta)} = 4\e.
\]
and consider the waves obtained by the solution to the Riemann problems $[u^-,u^m]$ in $x = 0$ and $[u^m,u^+]$ in $x=1$.

The solution of the Riemann problem at $x=0$ is a single shock $\bar w_1$ of strength $\sqrt{2} L$ traveling with speed $\alpha L/\sqrt{2}$, while the solution of the Riemann problem at $x=1$ has a first shock $\bar w_2=[-2L, - L (2 - \sqrt{2})]$ of strength $\sqrt{2} L$ traveling with speed
\[
- \alpha L \big( 2 - \sqrt{2} \big) < - \frac{1}{2} \alpha L,
\]
a rarefaction $[-L(2-\sqrt(2)), -\delta]$ and a second shock $[-\delta,u^+]$ traveling with speed $-\alpha \delta$.

Since every wavefront starting from $u^-$ and connecting to a point $u \leq - \delta$ has positive speed, then after a time $\frac{1}{\alpha \delta}$ all the waves have collided into a single wavefront with a speed
\[
\frac{-1/2 \alpha \e^2}{4 \e + L (2 + \sqrt{2})}.
\]

Now consider a contact discontinuities with size $[4\e, 3\e]$, located at $x = 1 + \e/\delta$. Since it is traveling with speed $-\alpha \e$, this discontinuity meets the other waves at $\bar t \geq 1/(\alpha \delta)$, i.e. after the large wavefront $[-L(2+\sqrt{2}),4\e]$ has formed, and it is fairly easy to see that this point corresponds to a splitting. Hence this wave configuration corresponds to a tree.

We now estimate the amount of interaction calculated as the decrease of the cubic functional $Q^{\text{BB}}$ of \cite{bia_bre_02}: since it is bounded by the area, then we can write
\begin{equation}
\label{AM_cubic}
\begin{split}
Q^{\text{BB}} \leq&~ \text{area of the triangle } (u^-,f(u^-)), (u^m,f(u^m)), (u^+,f(u^+)) \crcr
=&~ \frac{1}{2} \alpha \bigg( (2 + \sqrt{2}) L^3 + 4 \e L^2 + \frac{1}{\sqrt{2}} \e^2 L \bigg) \\
= & ~ \const \Big[ \alpha  (L+\e)^3 + \e \Big].
\end{split}
\end{equation}
The amount of cancellation is $\e$.

For the couple of wavefronts $\bar w_1, \bar w_2$, formula (4.84) of \cite{anc_mar_11_CMP} bounds the variation of speed of $\bar w_1$ between $t=0$ and the splitting time $\bar t+$ by
\begin{align*}
\big\|\sigma^{\bar w_1}(\bar t+) - \sigma^{\bar w_1}(0) \big\|_{L^\infty} = & ~ \const \Big[ \text{interaction + cancellation} \Big] \\
= & ~ \const \Big[ \alpha  (L+\e)^3 + \e \Big].
\end{align*}
We have used \eqref{AM_cubic}.

On the other hand, an explicit computation gives 
\[
\big\|\sigma^{\bar w_1}(\bar t+) - \sigma^{\bar w_1}(0) \big\|_{L^\infty} = \frac{\alpha L}{\sqrt(2)}.
\]
It is clear that since we can choose $\e \leq L^3$ and $L \ll 1$, then we cannot bound the difference in speed with the amount of interaction and cancellation.

The same reasoning holds for formula (4.85) in \cite{anc_mar_11_CMP}.


\end{document}